\documentclass[11pt,twoside]{article} 
\usepackage[utf8]{inputenc} \usepackage[T1]{fontenc} 
\usepackage[pagewise,switch,mathlines]{lineno}
\usepackage{etoolbox}

\usepackage{tabularx,array}
\usepackage[a4paper, left=25mm, right=25mm, top=25mm, bottom=25mm]{geometry} 
\usepackage{graphicx} \usepackage{setspace} \usepackage{xcolor} \usepackage{colortbl} 
\makeatletter
\providecommand\insert@pcolumn{\insert@column}
\makeatother
\usepackage{amsmath,amssymb,amsthm} \usepackage{mathtools} \mathtoolsset{showonlyrefs} \usepackage{mathrsfs} \usepackage{bm} 
\usepackage{circuitikz} 
\usepackage[all,2cell]{xy}
\usepackage{tikz-cd} \usetikzlibrary{cd} 
\usepackage{enumerate}
\usepackage{marginnote}
\usepackage{microtype}
\microtypesetup{expansion=false}
\usepackage{booktabs}

\usepackage[colorinlistoftodos]{todonotes}
\usepackage{comment} 

\usepackage{imakeidx} \makeindex[intoc] 

\usepackage[nottoc]{tocbibind} 

\usepackage{xurl} 
\usepackage[colorlinks=true]{hyperref} \hypersetup{urlcolor=blue, citecolor=red, linkcolor=blue} 
\usepackage[
  backend=biber,
  style=numeric-comp,
  sorting=nyt,
  sortcites=true,
  giveninits=true,
  maxnames=99
]{biblatex}

\theoremstyle{plain}
\newtheorem{theorem}{Theorem}[section]
\newtheorem{proposition}[theorem]{Proposition} 
\newtheorem{lemma}[theorem]{Lemma}

\theoremstyle{definition}
\newtheorem{definition}[theorem]{Definition}
\AtBeginEnvironment{definition}{%
  \begingroup
  \pushQED{\qed}%
}

\AtEndEnvironment{definition}{%
  \popQED
  \endgroup
}
\theoremstyle{remark}
\newtheorem{remark}[theorem]{Remark}
\newtheorem{example}[theorem]{Example} 

\AtBeginEnvironment{example}{%
  \pushQED{\qed}%
}

\AtEndEnvironment{example}{%
  \popQED%
}

\newcommand{\dd}{\mathrm{d}}

\newcommand{\T}{\mathrm{T}}

\usepackage{fancyhdr}
\usepackage{listings}

\AtBeginBibliography{\sloppy}

\begin{document} \parskip=3pt \vspace{5em} {\huge\sffamily\raggedright \begin{spacing}{1.1} Jet Bundles as Higher-Order Polarised \(k\)-Contact Manifolds\end{spacing} } \vspace{2em} {\large\raggedright \today } \vspace{3em} \\
{\Large\raggedright\sffamily Javier de Lucas }\vspace{1mm}\newline {\raggedright Department of Mathematical Methods in Physics, University of Warsaw,\\ ul. Pasteura 5, 02-093 Warszawa, Poland.\\Centre de Recherches Math\'ematiques, Universit\'e de Montr\'eal,\\ Succ. Centre-Ville, CP6128, chemin de la Tour, Montréal (Québec), Canada H3C 3J7. \\ e-mail: \href{mailto:javier.de.lucas@fuw.edu.pl}{javier.de.lucas@fuw.edu.pl} --- orcid: \href{https://orcid.org/0000-0001-8643-144X}{0000-0001-8643-144X} } \bigskip \\\vspace{2em}

{\large\bfseries\raggedright Abstract}\vspace{1mm}\newline
{ 
We prove that the Cartan distribution \(\mathcal{C}^r_\pi\) on the  jet bundle \(J^r\pi\), for \(\pi:E\rightarrow Q\), is an \(N^r_\pi\)-contact distribution, where \(N^r_\pi\) is the corank of $\mathcal{C}^r_\pi$. In other words, \(\mathcal{C}^r_\pi\) has non-degenerate distributional curvature, and admits, locally, a complementary distribution spanned by \(N^r_\pi\) commuting infinitesimal symmetries. Consequently, there exists around every point of \(J^r\pi\) an open  \(U\subset J^r\pi\) and an \(\mathbb{R}^{N^r_\pi}\)-valued one-form \(\bm\eta\in\Omega^1(U,\mathbb{R}^{N^r_\pi})\) such that \(\ker\bm\eta=\mathcal{C}^r_\pi|_U\) and \(\ker\bm\eta\oplus\ker\mathrm{d}\bm\eta=\mathrm{T}U\). This recovers the canonical structure of \(J^r\pi\), algebraic Spencer contractions, and other mathematical entities in a simple unifying manner while providing a local Hamiltonian structure and other advantages. We will call a manifold endowed with a \(k\)-contact distribution a \(k\)-contact manifold.

New types of polarisations for \(k\)-contact distributions lead to our main recognition theorem, which shows that a polarised \(k\)-contact distribution is locally equivalent to a Cartan distribution precisely when its polarisation is of jet type. This characterises finite-order jet geometry as polarised \(N_\pi^r\)-contact geometry of jet type. Moreover, the highest-order vertical polarisation, holonomic submanifolds, characteristics for Lie symmetries and their brackets, and other structures are reconstructed and generalised via \(k\)-contact geometry. Solutions of PDEs become polarised Legendrian submanifolds, jet prolongations are polarised $k$-contact Legendrian prolongations, etc.

Our $k$-contact formalism gives a common language for constructions that are awkward in a single fixed jet presentation, and extends jet geometry to new problems. Our techniques provide new general reduction methods for PDEs retrieving standard and non-local reductions as particular cases, and are applied to PDEs.
}

\bigskip


\medskip






{\setcounter{tocdepth}{2}
\small
\parskip 0pt
\tableofcontents
}

\pagestyle{fancy}

\fancyhead[LE]{ } 
\fancyhead[RE]{}
\fancyhead[RO]{J. de Lucas} 
\fancyhead[LO]{}    

\fancyfoot[L]{}     
\fancyfoot[C]{\thepage}                  
\fancyfoot[R]{}            


\renewcommand{\headrulewidth}{0.1pt}  
\renewcommand{\footrulewidth}{0pt}    

\renewcommand{\headrule}{%
    \vspace{3pt}                
    \hrule width\headwidth height 0.4pt 
    \vspace{0pt}                
}

\setlength{\headsep}{30pt}  

 
 \section{Introduction}

Jet bundles provide one of the standard geometric languages for differential equations, variational calculus, field theory and symmetry theory. Given a fibred manifold \(\pi:E\to Q\), the finite-order jet bundle \(J^r\pi\) records the Taylor data of local sections of \(\pi\) up to order \(r\). Its Cartan distribution \(\mathcal C^r_\pi\) encodes holonomicity via the usual contact forms, it allows one to describe prolongations of vector fields,  the Spencer operator measures the failure of a tangent vector to be holonomic. These structures underlie the geometric theory of PDEs, formal integrability, symbols, characteristics, Lie symmetries, and their physical applications \cite{Saunders_89,Olver1993,KV_99,Anderson_92,BCG+_91}.

Meanwhile, a contact manifold is a $(2n+1)$-dimensional manifold with a corank one maximally non-integrable distribution \cite{Geiges2008}. In mechanics, contact manifolds provide a natural setting for dissipative and nonconservative systems \cite{Riv_22}. Contact forms in this precise mathematical sense should not be confused with the {\it contact forms} appearing in finite-order jet manifold theory, which give rise to the relevant Cartan distributions but are not contact forms in the geometric sense unless $r=1$ and $m=1$ (cf. \cite{Saunders_89,Geiges2008}). Nevertheless, this work will unveil how these contact forms in jet geometry are related to their geometric counterparts in contact geometry and, more specifically, to their generalisations in field theory.

The field-theoretic analogues of the geometric contact forms have led to \(k\)-contact geometry, where one replaces the ordinary corank-one contact distribution by structures adapted to several independent variables; see, among others, \cite{LGLRR_20,LL_20,GGMRR_20,GGMRR_21,Riv_22,Riv_23}. These formalisms have been mainly developed for first-order field theories, Herglotz-type variational principles, dissipation laws, nonconservative Hamiltonian and Lagrangian first-order field theories. Until now, however, the relation between \(k\)-contact geometry and  higher-order jet geometry   was essentially completely unknown. Only a Darboux theorem leading to a relation of $k$-contact manifolds to first-order jet manifolds via a particular type of polarisation was known \cite{Riv_22}. 

A {\it $k$-contact manifold} is a manifold that is endowed with a distribution of non-degenerate distributional curvature \cite{PiccioneTausk2005} of arbitrary constant corank $k$, which admits, locally around every point, a complement spanned by a frame of commuting infinitesimal symmetries of the distribution \cite{LRS_24}. More classically \cite{Riv_22}, a co-oriented $k$-contact manifold is a manifold $M$  endowed with a vector-valued one-form $\bm\eta$ on $M$, called an {\it adapted $k$-contact form}, such that $\ker \bm\eta\oplus \ker \dd\bm\eta=\T M$ and $\ker\bm\eta\neq 0$. The two definitions are locally equivalent and interesting, but they have different applications. Indeed, as shown here, the distributional definition is a key to show that jet manifolds of any finite order can be described via $k$-contact geometry. This fact was also relevant to derive some previous applications of $k$-contact geometry to integrable systems and the theory of distributions \cite{LRS_24,LRS_26,SF_25}. But the use of $\bm\eta$ is extremely useful for local computations, Hamiltonian theory, and other properties. The two definitions are therefore complementary and both are used in this work.
 
The first purpose of this paper is to show that $k$-contact geometry recovers finite-order jet geometry as a particular instance and provides a  larger framework that allows for its description, deeper understanding, and generalisations. In short, we first recover finite jets via $k$-contact geometry and then extend them. Some relevant applications are schematically described, but the range of potential applications is much wider. Indeed, our second purpose is to show how the new approach extends the classical theory of finite-order jet bundles, while providing techniques that are not natural in the classical jet-bundle setting. Nevertheless, these methods are very useful in the applications of jet geometry to PDEs, symmetries, reductions, transformations, etc. In this sense, the paper also describes briefly further directions of research, including the variational theory of polarised \(k\)-contact manifolds, the extension of the Riemann invariant method for PDEs \cite{GL_23}, the relation to multisymplectic and multicontact formalisms, the treatment of non-conservative field theories, a general theory of reduction for PDEs, the analysis of $\lambda$ and $\mu$ symmetries, the description of contact geometry with differential one-forms with values in a vector bundle, etc.
 
A particularly important consequence of this viewpoint is a reduction theory
that is not restricted to ordinary Lie symmetries of the Cartan distribution.
Instead of requiring a Lie group action preserving
\(\mathcal C^r_\pi\) itself, we require the enlarged distributions
\(\mathcal C^r_\pi+\mathcal D_G\) and \(\mathcal G^r_\pi+\mathcal D_G\) to be
invariant. This is expressed in terms of the distribution $\mathcal{D}_G$ of fundamental vector fields of the Lie group action of $G$. When the quotient carries a jet-type polarisation induced by the projection of \(\mathcal C^r_\pi+\mathcal D_G\) and \(\mathcal G^r_\pi+\mathcal D_G\), our recognition theorem reconstructs it as a finite-order jet bundle, and the reduced equation becomes an equation on this reconstructed jet space. In the scalar ordinary differential case, this mechanism naturally includes the classical \(\lambda\)-symmetries: although a \(\lambda\)-prolonged vector field is generally not a Cartan symmetry its flow  gives the usual order-reducing geometry. Moreover, other particular reduction schemes on submanifolds of $J^r\pi$ are also investigated.

Let us describe the contents of the paper in more detail. 
After recalling some basic theory on jet manifolds and $k$-contact geometry, we present $k$-contact manifolds as a category for the first time with morphisms given by the so-called {\it transcontact maps}, which relate \(k\)-contact manifolds with different values of \(k\). In particular, one introduces $k$-contact morphisms as transcontact maps having an inverse that is a transcontact map.

Then, as a first main result, Theorem \ref{cor:Cartan_Nr_contact} proves that every finite-order jet bundle carries a canonical  \(N^r_\pi\)-contact distribution and a local $N^r_\pi$-contact form. More precisely, if \(\dim Q=n\) and $\pi:E\rightarrow Q$ has rank \(m\), then the Cartan distribution \(\mathcal C^r_\pi\) on \(J^r\pi\) is an \(N^r_\pi\)-contact distribution, where
\[
N^r_\pi=m\binom{n+r-1}{r-1}.
\]
This integer is exactly the corank of \(\mathcal C^r_\pi\). The result is immediate for $r=1$, and it was previously known \cite{Riv_22} as contact forms of a first-order jet manifold $J^1\pi$ can be gathered to give an adapted \(N^1_\pi\)-contact form. For \(r>1\), the result is not immediate, and only the understanding of $k$-contact distributions in a distributional context gives the answer appropriately. Indeed, the standard higher-order contact forms do not themselves form an adapted \(N^r_\pi\)-contact form. Instead, one can construct \(N^r_\pi\) commuting vector fields \(R^J_\alpha\) transverse to \(\mathcal C^r_\pi\) which are infinitesimal symmetries of \(\mathcal C^r_\pi\). These vector fields constitute a so-called {\it local Reeb frame} and are constructed on each adapted chart of $J^r\pi$ naturally. The dual coframe, vanishing on \(\mathcal C^r_\pi\), gives a local adapted \(N^r_\pi\)-contact form \(\bm\eta^r_\pi\). The adapted forms and Reeb frames depend on choices, but their common kernel \(\mathcal C^r_\pi\) is globally defined. This is why the distributional definition of \(k\)-contact geometry recalled in Section~\ref{sec:distributional_kcontact} is the natural one for higher-order jets. 

Our theory explains the features of the illustrating example given in Section~\ref{Sec:IntrExample}. Our construction gives a new interpretation of the {\it canonical jet structure} on $J^r\pi$ \cite{Campos_2010}. In Theorem~\ref{cor:Cartan_Nr_contact}, the adapted \(N^r_\pi\)-contact form satisfies
\(
\T\pi_{r,r-1}\circ \mathfrak{R}\circ\bm\eta^r_\pi=\mathfrak S^r_\pi,
\)
where \(\mathfrak{R}:e^\alpha_J\in \mathbb{R}^{N^r_\pi}\mapsto R_\alpha^J\in \mathfrak{X}(J^r\pi)\) is the {\it Reeb lift}, the $\pi_{r,r-1}$ is the canonical projection from $J^{r}\pi$ onto $J^{r-1}\pi$, and \(\mathfrak S^r_\pi\) is the canonical jet structure on $J^r\pi$. Moreover, the mixed part of \(\dd\bm\eta^r_\pi\) recovers the {\it algebraic Spencer contractions}, as shown in Section~\ref{sec:polarisations}. It is remarkable that, in appropriate local coordinates, $\mathfrak{R}\circ\bm\eta^r_\pi$ coincides with the canonical jet structure of $J^r\pi$. Note that  the study of the individual components of $\bm \eta^r_\pi$ alone, a so-called {\it $k$-contact structure} \cite{Riv_22}, is insufficient for recovering the full geometric structure  related to finite-order jet manifolds. In this sense, this reinforces the approach of \cite{LRS_24}, while the initial $k$-contact structure approach is of course still valid and useful for many purposes, specially in dissipative field theories \cite{Riv_22}.

The second main novelty of this work is a local recognition theorem, Theorem \ref{thm:local_split_jet_recognition}, characterising which $k$-contact manifolds are locally equivalent to finite-order jet manifolds. With this aim, we first introduce polarisations in Definition~\ref{def:higher_order_polarisation} and prove that the highest-order vertical distribution
\(
\mathcal G^r_\pi:=\ker \T\pi_{r,r-1}\subset\mathcal C^r_\pi
\)
is a polarisation of the \(N^r_\pi\)-contact manifold \((J^r\pi,\mathcal C^r_\pi)\); see Lemma~\ref{lem:jet_symbol_dominant}. This polarisation is not the whole vertical bundle over the base. It is the highest-order symbol direction inside the Cartan distribution. Starting from it, one recovers the spaces $\ker \T\pi_{r,r-s}$ for $s=1,\ldots,r$, the symbol spaces and the Spencer contractions through a derived flag generated by brackets with local horizontal Cartan directions.

The process of retrieving the full structure of the jet manifold from the Cartan distribution and the space $\ker\T\pi_{r,r-1}$ leads us to define several types of polarisations that capture, to a certain extent, the special features of the Cartan distribution. In short, a {\it polarised $k$-contact manifold} is a triple $(M,\mathcal{D},\mathcal{P})$ where $\mathcal{P}$ is subbundle of $\mathcal{D}$ that is integrable and Legendrian in a particular manner given by Definition \ref{def:higher_order_polarisation}, which retrieves the original definition in \cite{Riv_22} as a particular instance. This allows one to develop a theory of polarised $k$-contact vector fields and other related notions, which will be used to analyse Lie symmetries of Cartan distributions, holonomic sections, etc. 

In order to study jet manifolds, more specific types of polarisations appear.  In particular, one gets Definition~\ref{def:split_polarisation}, Definition~\ref{def:Spencer_type_derived_flag}, and Definition~\ref{def:polarised_jet_type}. Finally,  Theorem~\ref{thm:local_split_jet_recognition} proves that a polarised \(k\)-contact manifold is locally equivalent, namely polarised $N^r_\pi$-contactomorphic, to a finite-order jet bundle precisely when Definition~\ref{def:polarised_jet_type} concerning polarisations of jet type and finite order $r$ holds. Thus, finite-order jet bundles are not merely examples of \(k\)-contact manifolds. They form a rigid and intrinsically characterised sector of polarised \(k\)-contact geometry. This will be a key for further applications in this paper. The global version, Theorem~\ref{thm:global_jet_recognition}, explains how the local identifications glue to a global jet bundle under suitable regularity, descent assumptions, and bijectivity of the resulting map. Introduced definitions of classes of polarisation allow for necessary and sufficient conditions that can be applied to study particular types of polarisations in theoretical and practical examples.  Moreover, Theorem \ref{thm:local_split_jet_recognition} can be understood as a Darboux theorem which provides a constructing algorithm for obtaining coordinates putting $(M,\mathcal{D},\mathcal{P})$ locally as a jet manifold. This follows the idea in \cite{GLRR_24} for defining Darboux coordinates as describing canonical forms of geometric structures in different manners and generalises the original Darboux theorem for $k$-contact structures in \cite{Riv_22}.

Our recognition theorems are  conceptually important. The paper does not only reinterpret jet geometry in a different language. It shows that the Cartan geometry of finite jets is a distinguished type of a broader \(k\)-contact framework. In this sense, polarised \(k\)-contact geometry is more flexible than ordinary jet geometry: it contains finite-order jet models as particular objects, but it is not tied to one fixed fibration, one fixed choice of independent variables, or one fixed jet presentation. This extra flexibility is not merely aesthetic. It allows one to formulate operations which are not natural, if one insists on remaining inside a single jet bundle \(J^r\pi\).   

A further structural consequence of our formalism is that jet prolongation itself is recovered intrinsically. More precisely, Proposition~\ref{prop:jet_as_polarised_legendrian_prolongation} identifies \(J^{r+1}\pi\) canonically with the bundle \(\operatorname{Leg}_{\mathcal G^r_\pi}(\mathcal C^r_\pi)\) of polarised Legendrian planes in \(\mathcal C^r_\pi\), and the tautological distribution on this bundle becomes the Cartan distribution \(\mathcal C^{r+1}_\pi\). Thus, the passage from order \(r\) to order \(r+1\) is recovered as a polarised Legendrian prolongation. 

We next retrieve some common features of jet geometry in a $k$-contact manner, which also allows for extending some of these definitions for types of polarised $k$-contact manifolds. Section~\ref{subsec:holonomic_sections_jets} shows that holonomic sections in jet bundles are locally equivalent to our so-called {\it polarised Legendrian submanifolds}. A system of PDEs \(\mathcal E\subset J^r\pi\) is then treated as a submanifold with an induced distribution
\(
\mathcal C_{\mathcal E}:=\T\mathcal E\cap\mathcal C^r_\pi.
\)
Genuine solutions are polarised Legendrian submanifolds included in $\mathcal{E}$. The product-space construction in Proposition~\ref{prop:product_jet_spaces_polarised_k_contact} supplies the polarised \(k\)-contact geometry on fibre products \(J^p\pi\times_M J^q\pi'\), which provides the ambient framework used later for Bäcklund transformations, auxiliary systems, etc. This is useful to simplify artificial extensions of jet spaces, which are often used in the literature to treat Bäcklund transformations, Lax pairs, and other constructions. This supplies a common geometric framework for such extensions.

The formalism also gives a natural local Hamiltonian theory. A local coordinate system adapted to $J^r
pi$ originates a natural local adapted \(N^r_\pi\)-contact form \(\bm\eta^r_\pi\), every Cartan vector field \(X\) has a local \(k\)-contact Hamiltonian
\(
\bm h_X=-\iota_X\bm\eta^r_\pi.
\) 
The locus where \(\bm h_X\) vanishes is precisely the locus where \(X\) becomes tangent to the Cartan distribution, as proved in Proposition~\ref{prop:locus_hamiltonian_zero}. For prolonged point symmetries and evolutionary symmetries, the components of \(\bm h_X\) recover the usual Lie characteristics and their total derivatives after the triangular change from the standard contact coframe to the adapted \(k\)-contact coframe; see Proposition~\ref{prop:triangular_hamiltonian_characteristics}. Thus, the classical characteristic calculus of Lie symmetries is recovered as a distinguished sector of a wider local \(k\)-contact Hamiltonian theory. Moreover, previous results \cite{Olver1993} on characteristics of Lie symmetries, or reductions are recovered and generalised in the present framework. It is worth noting that the characteristics for general Cartan symmetries that are normally described in infinite-dimensional jet manifolds are recovered here as a $k$-contact finite-dimensional aspect. 

Section~\ref{sec:kConPDEs} develops previous commentaries further. The induced polarisation \(\mathcal G^r_{\mathcal E}=\T\mathcal E\cap\mathcal G^r_\pi\) measures the residual highest-order freedom left by the equation. The vertical principal symbol is interpreted as the restriction of the defining equations to \(\mathcal G^r_\pi\), and normal forms become transversality statements with respect to suitable subbundles of the highest-order polarisation. The same section also introduces the polarised solution set \(\operatorname{Pol}_{\mathcal E}(I_e)\) over an isotropic initial datum \(I_e\subset\mathcal C_{\mathcal E,e}\). This set consists of possible \(n\)-dimensional solution directions extending the datum, constrained by the \(k\)-contact orthogonal \(I_e^{\perp_{\mathcal C}}\) and by transversality to \(\mathcal G^r_\pi\). Thus, the algebraic part of the Cauchy problem is expressed intrinsically through Levi orthogonality and polarisation.

The same section also treats Hamiltonian loci of \(k\)-contact vector fields. After choosing a local adapted \(k\)-contact form, the zero locus of the corresponding \(k\)-contact Hamiltonian records where the vector field takes values in the relevant Cartan \(N^r_\pi\)-contact distribution. These loci provide natural differential systems carrying the induced Cartan distribution and residual polarisation. If a Lie algebra of Cartan symmetries acts on a differential equation, the common zero locus of the corresponding local \(k\)-contact Hamiltonians is an intrinsic Hamiltonian formulation of  invariant subsystems; see Theorem~\ref{thm:hamiltonian_invariant_subsystems} and related results in \cite{Olver1993}. 

Section~\ref{Sec:SymReductionJets} develops a very general reduction
theory for jet manifolds and associated PDEs. The essential point is that the reducing distribution need not consist
of Cartan symmetries. It is enough that the Cartan and polarisation data become
projectable after enlarging them by the reducing directions, namely through
\(\mathcal C^r_\pi+\mathcal D_G\), \(\T N+\mathcal D_G\), and
\(\mathcal G^r_\pi+\mathcal D_G\). If the quotient inherits a jet-type
polarisation, Theorem~\ref{thm:modular_reduction_jet_type} recognises it as a
finite-order jet bundle, so that the reduced system is again an equation on a
jet space. This gives a geometric explanation of order-reducing procedures
which are not generated by ordinary Cartan symmetries. In particular, for
scalar ODEs, classical \(\lambda\)-symmetries fit into this scheme as projective
\(k\)-contact reduction distributions. In particular, several results in the theory of $\lambda$-symmetries are recovered and generalised in the present framework. The reduction theory also applies to PDEs, including non-local reductions for PDEs \cite{gaeta2005} and other types for ODEs \cite{MurielRomero2012}, naturally.

The treatment of transformations illustrates the same point. Projectable transformations preserve the fibration and hence preserve the canonical vertical structures after prolongation. Hodograph transformations may exchange independent and dependent variables, and therefore are generally not projectable with respect to the original fibration. Nevertheless, they may preserve the Cartan distribution. In the present language, these are \(k\)-contact transformations between different jet manifolds and need not be projectable to $Q$; see Definition~\ref{def:hodograph_kcontact_transformation}. They send holonomic submanifolds to holonomic submanifolds precisely on the locus where the transformed integral submanifold is transverse to the new projection; see Proposition~\ref{prop:hodograph_preserves_holonomicity} and Theorem~\ref{thm:hodograph_transforms_solutions}. This separates the Cartan distribution, the choice of independent variables and the transversality condition that turns an integral submanifold into the graph of a section.

The same philosophy applies to Bäcklund transformations, coverings, and Lax pairs. These constructions typically introduce auxiliary variables, relate different equations, or encode a PDE through compatibility of an overdetermined system. They are therefore not naturally confined to a single fixed jet bundle. In the polarised \(k\)-contact framework, the extended spaces carry induced or enlarged Cartan distributions that can be considered as polarised $k$-contact distributions. In this new realm, many relevant vector fields are \(k\)-contact, their adapted Hamiltonians define the corresponding tangency loci and other relevant results can be achieved discussed in Section~\ref{Sec:BacklundLax}.

The last applications should be read in this spirit. They are not intended to replace the classical theories of Bäcklund transformations, Lax representations, hodograph transformations, etc. The same is true for \(\lambda\)-symmetry reduction: its role here is not merely
to provide another example, but to show that deformed prolongations can be
understood as projective reduction in the polarised \(k\)-contact
category. The role of these applications is to show that the polarised \(k\)-contact framework supplies a single geometric setting in which several constructions, usually treated by different techniques, can be described in terms of Cartan distributions, polarisations, Hamiltonian loci and Spencer-type compatibility. This is the sense in which the present paper is foundational: it identifies finite-order jet geometry as a rigid sector of polarised \(k\)-contact geometry, and at the same time shows that the \(k\)-contact viewpoint naturally accommodates operations which go beyond a fixed jet-theoretic presentation.

A further guiding principle is that some constructions typically formulated on infinite-order jet spaces admit intrinsic finite-dimensional descriptions. The characteristic description of Cartan symmetries and the scalar $\lambda$-symmetry reduction developed in this paper illustrate how the highest-order polarisation can support such finite-dimensional formulations. 

 The paper is organised as follows. Section~\ref{sec:Generalities} recalls the geometry of finite-order jet bundles, Cartan distributions, symbols and Spencer contractions. Section~\ref{sec:distributional_kcontact} reviews \(k\)-contact distributions from the distributional viewpoint. Section~\ref{Sec:IntrExample} explains, through a second-order field-theoretic example, why the standard higher-order contact forms do not directly give adapted \(k\)-contact forms. Section~\ref{Sec:jetkcontact} proves that finite-order Cartan distributions are \(N^r_\pi\)-contact distributions and identifies the adapted forms with the Spencer operator. Section~\ref{sec:polarisations} introduces higher-order polarisations and the associated tower of generalised distributions. Section~\ref{sec:jet_manifolds_as_higher_order_polarised_k_contact_manifolds} introduces jet-type polarisations and proves the recognition theorem for finite-order jet bundles. Section~\ref{subsec:holonomic_sections_jets} relates holonomic sections, polarised Legendrian submanifolds and solution-preserving transformations. Section~\ref{sec:lie_characteristics_hamiltonian} develops the Hamiltonian interpretation of Lie characteristics. Section~\ref{sec:kConPDEs} studies systems of PDEs as submanifolds of \(k\)-contact manifolds, including polarised initial data, Hamiltonian loci, normal forms and examples. Section~\ref{Sec:SymReductionJets} explains reduction between jet bundles by quotienting enlarged Cartan and polarisation data modulo group directions. Section~\ref{Sec:BacklundLax} discusses Bäcklund transformations, coverings and Lax pairs as Cartan compatibility and curvature-defect conditions. Section~\ref{Sec:Hodograph} treats hodograph transformations as non-projectable \(k\)-contact transformations. 
 The final section summarises the results and includes dictionaries between finite-order jet geometry and higher-order polarised \(k\)-contact geometry.

\section{Generalities on the geometry of jet bundles}\label{sec:Generalities}

Let us fix the notation and the basic facts on finite-order jet bundles used throughout the paper (see \cite{Saunders_89} for details). All manifolds, bundles, and maps are assumed to be smooth. Integer non-negative numbers are denoted by $\mathbb{N}_{\geq 0}$. Given a generalised subbundle $D\subset \T M$, we write $\Gamma(D)$ for the space of vector fields defined on open subsets of $M$ taking values in $D$. In particular, $
\mathfrak{X}(M)$ is the space of vector fields on $M$.

Let $\pi:E\to Q$ be a fibred manifold with $\dim Q=n$ and fibre dimension $m$.  For $r\geq 1$, the {\it $r$-th jet bundle} $J^r\pi$ is the manifold of {\it $r$-jets}, $j_x^r\phi$, of local sections $\phi$ of $\pi$ at points $x\in Q$, while it is hereafter assumed that $J^0\pi:=E$. We denote by $\pi_r:J^r\pi\to Q$, $\pi_{r,0}:J^r\pi\to E$, and $\pi_{r,s}:J^r\pi\to J^s\pi$, $0\leq s\leq r$, the canonical projections. If $(x^i,u^\alpha)$, with $i=1,\ldots,n$ and $\alpha=1,\ldots,m$, are adapted coordinates to $\pi:E\rightarrow Q$ on $U\subset E$, the  induced local adapted coordinates on $\pi^{-1}_{r,0}(U)\subset J^r\pi$ are denoted by $(x^i,u^\alpha_I)$, where $K=(k_1,\ldots,k_n)\in\mathbb N_{\geq 0}^n$ is a multi-index with $|K|:=k_1+\cdots+k_n\leq r$. We write $\widehat{e}_i$ for the multi-index whose only non-zero entry is a $1$ in the $i$-th position. Note that multi-indices can be added $K+J=(k_1+j_1,\ldots,k_n+j_n)$. If $J\geq K$, namely $j_a\geq k_a$ for every $a=1,\ldots,n$, then $J-K\in\mathbb N_{\geq 0}^n$. Meanwhile, we define $(J-K)!:=(j_1-k_1)!\cdots(j_n-k_n)!$ and $x^{J-K}:=(x^1)^{j_1-k_1}\cdots(x^n)^{j_n-k_n}$ for $J\geq K$.

 For each \(s\geq 0\), the number of multi-indices \(K\) with \(|K|=s\) is \(\binom{n+s-1}{s}\). Hence, the number of fibre jet variables \(u^\alpha_K\) with \(0\leq |K|\leq r\) is \(m\sum_{s=0}^r\binom{n+s-1}{s}=m\binom{n+r}{r}\), and therefore \(\dim J^r\pi=n+m\binom{n+r}{r}\). The number of variables $u^\alpha_J$ for   \(0\leq |J|\leq r-1\) reads 
\[
N^r_\pi:=m\sum_{s=0}^{r-1}\binom{n+s-1}{s}=m\binom{n+r-1}{r-1}.
\]

Every local section $\phi:U\subset Q\to E$ admits its canonical prolongation $j^r\phi:U\to J^r\pi$, $x\mapsto j_x^r\phi$. A section $\psi:U\subset Q\to J^r\pi$ of $\pi_r$ is called \emph{holonomic} if $\psi=j^r\phi$ for some local section $\phi:U\subset Q\rightarrow E$ of $\pi$. More generally, $L\subset J^r\pi$ is an \emph{\(n\)-dimensional holonomic submanifold} if  $L=j^r\phi(U)$ for some local section $\phi:U\subset Q\rightarrow E$.

The \emph{local contact one-forms} on $U\subset J^r\pi$ related to $(x^i,u^\alpha_K)$ are $\theta^\alpha_K:=\dd u^\alpha_K-\sum_{i=1}^nu^\alpha_{K+\widehat{e}_i}\dd x^i$, with $0\leq |K|\leq r-1$ and $\alpha=1,\ldots,m$. All local contact one-forms on $J^r\pi$ span a globally defined codistribution $\mathcal C^{r\circ}\subset \T^*J^r\pi$, although the individual forms $\theta^\alpha_K$ depend on the adapted coordinate system. The \emph{Cartan distribution} on $J^r\pi$ is
\[
 \mathcal C^r_\pi:=\bigcap_{\stackrel{1\leq\alpha\leq m}{|K|\leq r-1}}\ker\theta^\alpha_K.
\]
Equivalently, in adapted coordinates, one has on $U$ that
\begin{equation}\label{eq:notationCD}
 \mathcal C^r_\pi=\left\langle D_i,\frac{\partial}{\partial u^\alpha_K}:i=1,\ldots,n,\ |K|=r,\ \alpha=1,\ldots,m\right\rangle,\quad D_i:=\frac{\partial}{\partial x^i}+\sum_{\alpha=1}^m\sum_{|K|\leq r-1}u^\alpha_{K+\widehat{e}_i}\frac{\partial}{\partial u^\alpha_K}.
\end{equation}
The vector fields $D_1,\ldots,D_n$ are the \emph{local total derivatives} associated with the chosen adapted coordinates. Local total derivatives for different adapted coordinates on $J^r\pi$ do not span, in general, a well-defined global distribution of rank $n$ on $J^r\pi$, but the distribution $\mathcal C^r_\pi$ is globally defined. The rank and corank of $\mathcal C^r_\pi$ are
\[
 \operatorname{rank}\mathcal C^r_\pi=n+m\binom{n+r-1}{r},\qquad \operatorname{corank}\mathcal C^r_\pi=N^r_\pi=m\binom{n+r-1}{r-1}.
\]
Although the one-forms $\theta^\alpha_K$ are traditionally called contact forms, they are contact forms in the usual corank-one geometric sense only for $r=1$ and $m=1$. They should rather be regarded as local generators of the codistribution, $\mathcal{C}^{r\circ}\subset \T^*J^r\pi$, given by the annihilator of $\mathcal{C}^{r}$. Moreover, it would be more appropriate to call the $\theta^\alpha_K$ local jet contact forms.

\begin{proposition}\label{prop:holonomic_characterisations}
Let $\psi:U\subset Q\to J^r\pi$ be a local section of $\pi_r$. The following conditions are equivalent: a) $\psi$ is holonomic; b) $\T\psi(\T U)\subset\mathcal C^r_\pi$; one has $\psi^*\theta=0$ for every $\theta\in\Gamma(\mathcal C^{r\circ})$; c)  in adapted coordinates $\psi(x)=(x^i,\psi^\alpha_K(x))$, it follows that $\partial\psi^\alpha_K/\partial x^i=\psi^\alpha_{K+\widehat{e}_i}$ for all $|K|\leq r-1$, $i=1,\ldots,n$, and $\alpha=1,\ldots,m$. In particular, an $n$-dimensional submanifold $L\subset J^r\pi$ transverse to $\pi_r$, namely $\T L\cap  \ker \T\pi_r=0$, is holonomic if and only if $\T L\subset\mathcal C^r_\pi$ and $\pi_r|_L:L\rightarrow \pi_r(L)$ is a diffeomorphism.
\end{proposition}

\begin{proof}
The equivalence follows immediately from the coordinate expression of the jet contact forms, namely $(\psi^*\theta^\alpha_K)(\partial/\partial x^i)=\partial\psi^\alpha_K/\partial x^i-\psi^\alpha_{K+\widehat{e}_i}$ for $|K|<r$, $i=1,\ldots,n$, and $\alpha=1,\ldots,m$. The final statement follows by writing a submanifold transverse to $\pi_r$ locally as the image of a section of $\pi_r$, which is possible due to the transversality condition.
\end{proof}

 Since $\mathcal{G} ^r_\pi:=\ker \T\pi_{r,r-1}\subset\mathcal C^r_\pi$, one has the canonical exact sequence
\[
 0\hookrightarrow \mathcal{G} _\pi^r\hookrightarrow \mathcal C^r_\pi\xrightarrow{\T\pi_r}\T Q\longrightarrow 0.
\]
Thus, one has $\mathcal C^r_\pi/\mathcal{G} _\pi^r\simeq \pi_r^*\T Q$. Local choices of total derivatives $D_1,\ldots,D_n$ give local splittings of this sequence, but such splittings are not canonical. In adapted coordinates in some $U\subset J^r\pi$, one has
\[
 \mathcal{G}^r_\pi=\left\langle\frac{\partial}{\partial u^\alpha_K}: |K|=r, \alpha=1,\ldots,m\right\rangle,
\]
and $\mathcal{G}^r_\pi$ is canonically isomorphic to $\pi_r^*S^r\T^*Q\otimes\pi_{r,0}^*VE$, where $S^r\T_x^*Q$ is the space of totally symmetric $r$-covariant tensor products on elements of $\T_x^*Q$ and $S^r\T^*Q=\bigsqcup_{x\in Q}S^r\T^*Q$ (see  \cite[Sec.~6.]{Seiler_2007_Spencer}). We shall also use the so-called \emph{vertical tower} $\mathcal V_{r}:=\ker \T\pi_{r}$, $\mathcal V_s:=\ker \T\pi_{r,r-s-1}$ for $s=0,\ldots,r-1$, and $\mathcal{V}_{-1}:=0$. In adapted coordinates on $U\subset J^r\pi$, it follows that  $\mathcal V_s$ is spanned by $\partial/\partial u^\alpha_K$ with $r\geq |K|\geq r-s$ and $\alpha=1,\ldots,m$ for each $s=0,\ldots,r$. The graded quotients satisfy
\begin{equation}\label{eq:isovertical}
 \mathcal V_{s}/\mathcal V_{s-1}\simeq \pi_r^*S^{r-s}\T^*Q\otimes\pi_{r,0}^*VE,\qquad s=0,\ldots,r,
\end{equation}
where $S^0\T^*Q$ is a trivial line bundle over $Q$. In particular, one has the isomorphism
\[
\varepsilon_q:\mathcal V_s/\mathcal V_{s-1}\longrightarrow \pi_r^*S^qT^*Q\otimes\pi_{r,0}^*VE,\qquad \varepsilon_q\!\left(\left[\frac{\partial}{\partial u^\alpha_K}\right]\right):=\odot\dd x^{\langle K\rangle}\otimes\frac{\partial}{\partial u^\alpha},\quad q=r-s,\ |K|=q,
\]
where we denote
\[
\odot\,\dd x^{\langle J\rangle}(X_1,\ldots,X_r)
:=
\frac{1}{J!}\sum_{\sigma\in S_r}
\bigl((\dd x^1)^{\otimes j_1}\otimes\cdots\otimes(\dd x^n)^{\otimes j_n}\bigr)
(X_{\sigma(1)},\ldots,X_{\sigma(r)}),
\]
for all $ X_1,\ldots,X_r\in \mathfrak{X}(M)$. More simply, $\dd x^{i_1}\odot \dd x^{i_2}=\odot\,\dd x^{\langle \widehat{e}_{i_1}+\widehat{e}_{i_2}\rangle}$.
Moreover, for $|K|\leq r$ and $\alpha=1,\ldots,m$, it follows that
\[
 \left[D_j,\frac{\partial}{\partial u^\alpha_K}\right]=
 \begin{cases}
 -\dfrac{\partial}{\partial u^\alpha_{K-\widehat{e}_j}},& k_j>0,\\[1ex]
 0,& k_j =0.
 \end{cases}
\]
This allows us to define, via the standard isomorphisms \eqref{eq:isovertical},  the usual contraction $\pi_r^*\T Q\times \pi_r^*S^{r-s}\T^*Q\otimes \pi^*_{r,0}VE\to \pi^*_{r}S^{r-s-1}\T^*Q\otimes \pi^*_{r,0}VE$, up to the conventional sign, for $s=0,\ldots,r-1$. This has appeared in literature related to the algebraic Spencer contraction for jet manifolds \cite[Def.~2.1 and Sec.~6]{Seiler_2007_Spencer}.

The \emph{canonical structure form on \(J^r\pi\)} (see \cite[Definition 4.8]{Campos_2010}) or {\it Cartan one-form} \cite[Section 2.1]{CattafiCrainicSalazar_2020}, is the vector-valued one-form with local expression
\[
\mathfrak{S}_\pi^r\in\Omega^1\!\left(J^r\pi,\pi_{r,r-1}^*V\pi_{r-1}\right),\qquad \mathfrak{S}_\pi^r=\sum_{\alpha=1}^m\sum_{|K|\leq r-1}\theta^\alpha_K\otimes\frac{\partial}{\partial u^\alpha_K}.
\]
Equivalently, if \(\theta\in J^r\pi\), \(\theta_{r-1}:=\pi_{r,r-1}(\theta)\), and \(v\in \T_\theta J^r\pi\), then \(\T_\theta\pi_{r,r-1}(v)\in \T_{\theta_{r-1}}J^{r-1}\pi\) projects to \(\T_\theta\pi_r(v)\in \T_{\pi_r(\theta)}Q\). The \(r\)-jet \(\theta\) determines the holonomic lift of this base vector to \(\T_{\theta_{r-1}}J^{r-1}\pi\). The canonical structure measures the vertical difference between \(\T_\theta\pi_{r,r-1}(v)\) and this holonomic lift. Thus, \((\mathfrak{S}^r_\pi)_\theta(v)\) measures the failure of \(v\) to project to the tangent vector on \(J^{r-1}\pi\) prescribed by the \(r\)-jet \(\theta\). The kernel of the canonical structure on $J^r\pi$ is precisely the Cartan distribution, $\ker \mathfrak{S}^r_\pi=\mathcal C^r_\pi$. Hence, for a section $\psi:U\to J^r\pi$, the condition $\mathfrak{S}^r_\pi\circ \T\psi=0$ is equivalent to holonomicity. In a dual manner, $\mathfrak{S}^r_\pi$ can be understood as a differential operator mapping functions on $\pi_{r,r-1}^*C^\infty(J^{r-1}\pi)$ into one-forms on $J^r\pi$. In this sense, $\mathfrak{S}^r_\pi$ is understood as a Spencer operator. Then, the contraction of \(\dd f\) with the Spencer operator gives
$$
\iota_{\dd f}\mathfrak{S}_\pi^r=\sum_{\alpha=1}^m\sum_{|K|\leq r-1}\theta^\alpha_K\frac{\partial f}{\partial u^\alpha_K}
$$
which is a truncated vertical derivative of $f$ (see \cite{Olver_Pohjanpelto_2008}). Since $\mathfrak{S}_\pi^r$ takes values in $\pi_{r,r-1}^*V\pi_{r-1}$, it is natural to consider the Spencer operator only on functions where the equivalent class of vector fields $\pi_{r,r-1}^*V\pi_{r-1}\simeq \ker \T\pi_r/\ker \T\pi_{r,r-1}$ acts in a well-defined manner. This is the case for functions of $\pi_{r,r-1}^*C^\infty(J^{r-1}\pi)$.

For later use, let us also recall the algebraic Spencer contraction in the jet setting. Let \(x\in Q\), \(e\in (J^r\pi)_x\), and define $V_e\pi=(\pi_{r,0}^*V\pi)_e$, where $V\pi=\ker \T\pi$. If \(A\in S^s\T_x^\ast Q\otimes V_e\pi\), and \(h\in \T_xQ\) is a horizontal direction, its Spencer contraction is the element
\[
 \iota_hA\in S^{s-1}\T_x^\ast Q\otimes V_e\pi
\]
obtained by inserting \(h\) in one symmetric argument of \(A\). Equivalently, this gives a bilinear map
\[
 \T_xQ\times S^s\T_x^\ast Q\otimes V_e\pi\longrightarrow S^{s-1}\T_x^\ast Q\otimes V_e\pi,\qquad (h,A)\longmapsto\iota_hA .
\]
More geometrically, there is a natural identification \(\ker \T_e\pi_{r,r-1}\simeq \pi_r^*S^r\T_{\pi_r(e)}^\ast Q\otimes \pi_{r,0}^*VE\). In adapted jet coordinates \((x^i,u^\alpha_I)\), this contraction is expressed as follows. If \(h=\sum_{i=1}^n h^i\partial/\partial x^i\) and  \(A=\sum^{\alpha=1,\ldots,m}_{|K|=s}A^\alpha_K\partial/\partial u^\alpha_K\), then
\[
\iota_hA=\sum_{\alpha=1}^m\sum_{|J|=s-1}\left(\sum_{i=1}^n h^iA^\alpha_{J+\widehat e_i}\right)\frac{\partial}{\partial u^\alpha_J}.
\]
This is the algebraic operation behind the Spencer operator in the jet tower.

We call {\it Cartan symmetry} of $J^r\pi$ a vector field $Y\in \mathfrak{X}(J^r\pi)$ such that $[Y,\Gamma(\mathcal{C}^r_\pi)]\subset \Gamma(\mathcal{C}^r_\pi)$. 
Let $X\in\mathfrak X(E)$, namely $X=\sum_{i=1}^n\xi^i(x,u)\partial/\partial x^i+\sum_{\alpha=1}^m\varphi^\alpha(x,u)\partial/\partial u^\alpha$ in local adapted coordinates to $\pi:E\rightarrow Q$. Its \emph{\(r\)-th prolongation} is the Cartan symmetry given by
\begin{equation}\label{eq:prolongation}
     X^{(r)}=\sum_{i=1}^n\xi^i\frac{\partial}{\partial x^i}+\sum_{\alpha=1}^m\sum_{|I|\leq r}\varphi^\alpha_I\frac{\partial}{\partial u^\alpha_I},
\end{equation}
where $\varphi^\alpha_0=\varphi^\alpha$ and $\varphi^\alpha_{I+\widehat{e}_i}=D_i\varphi^\alpha_I-\sum_{j=1}^n u^\alpha_{I+\widehat{e}_j}D_i\xi^j$ for $\alpha=1,\ldots,m$, $i=1,\ldots,n$, and $|I|\leq r-1$. Note that $\varphi^\alpha$ and $\xi^j$ are here and afterwards frequently considered, in the standard way  as functions on $J^r\pi$ projectable onto $E$. Moreover, since only the coordinates $\varphi^\alpha_{I}$ for $|I|=r$ of $X^{(r)}$ depend  on the coordinates $u^\beta_I$ with $|I|=r$, one has
\[
 [X^{(r)},\Gamma(\mathcal{G}^r_\pi)]\subset\Gamma(\mathcal{G}^r_\pi).
\]

The \emph{characteristics} of $X$ associated with adapted coordinates $(x^i,u^\alpha_J)$ are the fibre component $Q^\alpha=\varphi^\alpha-\sum_{i=1}^n u^\alpha_i\xi^i$ for $\alpha=1,\ldots,m$. Meanwhile, for a general Cartan vector field \(Y\in\mathfrak X(J^r\pi)\), in a local contact coframe \(\{\theta^\alpha_I\}_{|I|\le r-1}\), one may define the functions
\[ Q^\alpha_I:=\theta^\alpha_I(Y)=\varphi^\alpha_I-\sum_{i=1}^nu^\alpha_{I+\widehat{e}_i}\xi^i,\qquad |I|\le r-1,\qquad \alpha=1,\ldots,m.
\]
 This family represents the quotient class of \(Y\) in \(\T J^r\pi/\mathcal C^r_\pi\).  For \(X^{(r)}\) one has \(Q^\alpha_I=D_IQ^\alpha\) and the characteristic of $X\in \Gamma(E)$ allows us to retrieve all the $Q^\alpha_I$ and the whole form of $X^{(r)}$ from ${\bm Q}$. For evolutionary Lie symmetries, namely $Y\in \Gamma(\ker \T\pi_r)$, the characteristics of $Y$ are $Yu^\alpha=Q^\alpha$ for $\alpha=1,\ldots,m$. Moreover, again $\iota_{Y^{(r)}}\theta^\alpha_I=D_IQ^\alpha$ for $\alpha=1,\ldots,m$ and $|I|\leq r-1$. Hence, it makes sense to write \(X_{\bm Q}\) for the evolutionary Lie symmetry with characteristic \(\bm Q=(Q^\alpha)\). Moreover, the commutator of the corresponding prolonged vector fields is again evolutionary, and its characteristic is the usual bracket
\[
[Q,P]_{\rm ch}^\alpha=X^{(r)}_Q(P^\alpha)- X^{(r)}_P(Q^\alpha),\qquad \alpha=1,\ldots,m.
\]
Thus, the space of characteristics carries the Lie algebra structure induced by commutators of prolonged vector fields.

\section{A distributional approach to \texorpdfstring{$k$}--contact geometry}\label{Sec:diskcontact}\label{sec:distributional_kcontact}

Let us recall the basic facts on \(k\)-contact geometry that will be used throughout the paper, following the distributional viewpoint of \cite{LRS_24}. This approach is particularly convenient here because the Cartan distributions of jet bundles are globally defined, whereas the adapted \(k\)-contact forms describing them will usually be only local. Moreover, the characterisation of Cartan distributions as $k$-contact distributions can be naturally achieved by using their distributional characterisation (see Theorem 3.14 in \cite{LRS_24}) as will be seen below.  Only the definitions of transcontact maps, $k$-contactomorphisms, and $k$-contact submanifolds are new.  

Throughout this section, \(k\geq 1\) is a fixed integer and \(V\) denotes a \(k\)-dimensional real vector space. After choosing a basis \(\{e_1,\ldots,e_k\}\) of \(V\), a \(V\)-valued $\ell$-form \(\bm\zeta\in\Omega^\ell(U,V)\) is canonically  written as \(\bm\zeta=\sum_{A=1}^k\zeta^A\otimes e_A\) for certain $\ell$-forms $\zeta^1,\ldots,\zeta^k\in \Omega^\ell(U)$. We use the notation
\[
 \ker\bm\zeta:=\bigcap_{A=1}^k\ker\zeta^A,\qquad \ker\dd\bm\zeta:=\bigcap_{A=1}^k\ker\dd\zeta^A.
\]
Moreover, the contraction of a vector field \(X\in \mathfrak{X}(U)\) with a \(V\)-valued differential $\ell$-form $\bm \zeta\in \Omega^\ell(U,V)$ is the \(V\)-valued differential $(\ell-1)$-form \(\iota_X\bm\zeta=\sum_{A=1}^k\iota_X\zeta^A\otimes e_A\in \Omega^{\ell-1}(U,V)\).

Let us now turn to defining one of the main objects of this paper.
\begin{definition}\label{dfn:k-contact-manifold}
A \(V\)-valued one-form \(\bm\eta\in\Omega^1(U,V)\) on an open subset \(U\subset M\) is called a \emph{\(k\)-contact form} if \(\ker\bm\eta\) is a non-zero distribution of corank \(k\) and
\(
 \ker\bm\eta\oplus\ker\dd\bm\eta=\T U.
\)
The family of components of a \(k\)-contact form will be called \emph{a \(k\)-contact structure} as standard in the literature \cite{Riv_22}.
If such a form $\bm\eta$ is globally defined on \(M\), the pair \((M,\bm\eta)\) is called a \emph{co-oriented \(k\)-contact manifold}. The distribution \(\ker\dd\bm\eta\) is its \emph{Reeb distribution}. 
\end{definition}

A one-contact form is a contact form. Indeed, the condition $\ker \eta\oplus\ker\dd\eta=\T U$ amounts to $\langle\eta\rangle \oplus\,\,{\rm Im\,\dd\eta}=\T^*U$, which gives \(\dim M=2n+1\) and nonvanishing \(\eta\wedge(\dd\eta)^n\). The converse is immediate. 

 The formulation above is better adapted to the present paper because it treats the \(k\) components of a \(k\)-contact form as a single vector-valued object, which will be a key to retrieve certain typical structures in jet manifolds. 

\begin{theorem}\label{thm:k-contact-Reeb}
Let \((M,\bm\eta)\) be a co-oriented \(k\)-contact manifold with \(\bm\eta=\sum_{A=1}^k\eta^A\otimes e_A\). There exists a unique family of vector fields \(R_1,\ldots,R_k\in\mathfrak X(M)\) such that
\[
 \iota_{R_A}\eta^B=\delta_A^B,\qquad \iota_{R_A}\dd\eta^B=0,\qquad A,B=1,\ldots,k.
\]
Moreover, \([R_A,R_B]=0\) for all \(A,B\), and \(\ker\dd\bm\eta=\langle R_1,\ldots,R_k\rangle\).
\end{theorem}

\begin{definition}The \(k\)-tuple \(\mathbf R=(R_1,\ldots,R_k)\) of Theorem~\ref{thm:k-contact-Reeb} is called the \emph{Reeb \(k\)-vector field} of \((M,\bm\eta)\), and the vector fields \(R_1,\ldots,R_k\) are called the \emph{Reeb vector fields}, which span the \emph{Reeb distribution}. When understood as a frame of the Reeb distribution, the Reeb vector fields are also called a \emph{Reeb frame}.
\end{definition}

The Reeb distribution depends on the chosen \(k\)-contact form. In contrast, the distributions that are locally given by \(\ker\bm\eta\) for a certain $k$-contact form will be the main object in the sequel. This motivates the following definition.

\begin{definition}\label{def:k-contact-distribution}
A \emph{\(k\)-contact distribution} on \(M\) is a regular distribution \(\mathcal D\subset \T M\) of corank \(k\) such that every point of \(M\) has an open neighbourhood \(U\) and a \(k\)-contact form \(\bm\eta\in\Omega^1(U,V)\) satisfying \(\mathcal D|_U=\ker\bm\eta\). A \emph{\(k\)-contact manifold} is a pair \((M,\mathcal D)\), where \(\mathcal D\) is a \(k\)-contact distribution.  
\end{definition}

Let \(\mathcal D\subset \T M\) be any regular distribution and let \(\operatorname{pr}_{\mathcal D}:\T M\to \T M/\mathcal D\) be the quotient projection. The \emph{Levi tensor} of \(\mathcal D\) is the vector bundle morphism
\[
 \mathcal L_{\mathcal D}:\Lambda^2\mathcal D\longrightarrow \T M/\mathcal D,\qquad \mathcal L_{\mathcal D}(X_x,Y_x):=[X,Y]_x\ \operatorname{mod}\mathcal D_x,
\]
where \(X,Y\in\Gamma(\mathcal D)\) are local extensions of \(X_x,Y_x\in\mathcal D_x\). Note that $[X,Y]_x$ mod $\mathcal{D}_x$ is independent of the chosen vector field extensions of $X_x$ and $Y_x$. If \(\mathcal D=\ker\bm\eta\), then
\begin{equation}\label{eq:iso}
 \mathcal L_{\mathcal D}(X,Y)=-\dd\bm\eta(X,Y),\qquad X,Y\in\Gamma(\mathcal D),
\end{equation}
where we assumed the natural isomorphism \(v\in V\mapsto \bm\eta^{-1}(v)\in \T M/\mathcal D\). It is worth noting that $\mathcal{L}_\mathcal{D}$ is also called a {\it Fr\"obenius tensor} \cite{PiccioneTausk2005}. It can also be understood as a distribution curvature \cite{VITA_15}.

Let us now introduce the geometric subspaces of a $k$-contact distribution that will be used in the sequel \cite{LRS_24,HLM_26,DLLRS_26}.

\begin{definition}\label{def:GeometricSubsapces}
Let \(\mathcal D\subset \T M\) be a regular distribution and let \(E_x\subset\mathcal D_x\). The \emph{Levi orthogonal} of \(E_x\) is
\[
 E_x^{\perp_{\mathcal D}}:=\{v_x\in\mathcal D_x:\mathcal L_{\mathcal D}(v_x,w_x)=0,\ \forall w_x\in E_x\}.
\]
The subspace \(E_x\) is \emph{isotropic} if \(E_x\subset E_x^{\perp_{\mathcal D}}\) and \emph{Legendrian} if it is isotropic and admits an isotropic complement in \(\mathcal D_x\), namely if there exists an isotropic subspace \(W_x\subset\mathcal D_x\) such that \(\mathcal D_x=E_x\oplus W_x\). A subspace $E_x\subset \T_xM$ is \emph{coisotropic} if $(E_x\cap \mathcal{D}_x)^{\perp_\mathcal{D}}\subset E_x\cap \mathcal{D}_x$ and \emph{maximally coisotropic} if $(E_x\cap \mathcal{D}_x)^{\perp_\mathcal{D}}=E_x\cap \mathcal{D}_x$.
\end{definition}

If \(\mathcal D=\ker\bm\eta\) for a $k$-contact form $\bm \eta$, then Definition \ref{def:GeometricSubsapces} recovers isotropic, coisotropic, maximally coisotropic, and Legendrian spaces in $k$-contact geometry \cite{LRS_24}. For instance, a subspace \(E_x\subset\mathcal D_x\) is isotropic if and only if \(\dd\bm \eta_x|_{E_x\times E_x}=0\). A submanifold \(N\subset M\) with \(\T N\subset\mathcal D\) is called isotropic, coisotropic, maximally coisotropic, or Legendrian if \(\T_xN\subset\mathcal D_x\) has the corresponding property for every \(x\in N\). 

\begin{definition}
A regular distribution \(0\neq \mathcal D\subset \T M\) is \emph{maximally non-integrable} if \(\mathcal D_x^{\perp_{\mathcal D}}=\{0\}\) for every \(x\in M\). Equivalently, the Levi tensor \(\mathcal L_{\mathcal D}\) is \emph{non-degenerate} in the sense that  \(\ker \mathcal{L}_{\mathcal D}=0\).
\end{definition}

Every \(k\)-contact distribution is maximally non-integrable. Indeed, if \(\mathcal D|_U=\ker\bm\eta\), then the condition \(\ker\bm\eta\oplus\ker\dd\bm\eta=\T U\) implies \(\mathcal D\cap\ker\dd\bm\eta=\{0\}\), which amounts to the non-degeneracy of the Levi tensor.  

\begin{definition}
A vector field \(X\in\mathfrak X(M)\) is a \emph{Lie symmetry of a distribution} \(\mathcal D\subset \T M\) if
\(
 [X,\Gamma(\mathcal D)]\subset\Gamma(\mathcal D).
\)
If \((M,\mathcal D)\) is additionally a \(k\)-contact manifold, such a vector field is also called a \emph{\(k\)-contact vector field}. Moreover, if \(\mathcal D=\ker\bm \eta\) for a \(k\)-contact form \(\bm\eta\in\Omega^1(M,\mathbb R^k)\), then \(X\) is also called an \emph{\(\bm\eta\)-Hamiltonian vector field}, and its \emph{associated \(\bm\eta\)-Hamiltonian \(k\)-function} is defined to be \(\bm h_X=-\iota_X\bm\eta\).
\end{definition}

We hereafter write $\mathfrak{X}_{\mathcal{D}}(M)$ for the space of $k$-contact vector fields on $M$ relative to the $k$-contact distribution $\mathcal{D}$. In the case of a co-oriented $k$-contact manifold $(M,\bm\eta)$, its space of $k$-contact vector fields will be defined by $\mathfrak{X}_{\bm \eta}(M)$. Meanwhile, we write ${\rm Ham}(M,\bm\eta)$ for its space of associated $k$-contact Hamiltonian functions.
 
It is worth noting that every $\bm \eta$-Hamiltonian function is related to a unique $\bm \eta$-Hamiltonian vector field \cite{LRS_24}. This results from the fact that $[X,\Gamma(\ker\bm\eta)]\subset\Gamma(\ker\bm\eta)$ implies that $\iota_X\dd\bm\eta=\dd \bm h_X-\sum_{A=1}^k\eta^A\otimes (\mathcal{L}_{R_A}\bm h_X) $. This yields that the difference of two vector fields with the same $\bm h_X$ takes values in $\ker \bm \eta\cap \ker\dd\bm \eta=0$.

\begin{definition}
Let \(\mathcal D\) be a regular distribution of corank \(k\). A \emph{local Reeb frame} of $\mathcal{D}$ is a family of vector fields \(S_1,\ldots,S_k\in\mathfrak X(U)\) on an open subset \(U\subset M\) such that \([S_A,S_B]=0\), each \(S_A\) is a Lie symmetry of \(\mathcal D|_U\), and
\[
 \mathcal D|_U\oplus\langle S_1,\ldots,S_k\rangle=\T U.
\]
\end{definition}

If \(\mathcal D|_U=\ker\bm\eta\), the Reeb vector fields of \(\bm\eta\) form a local Reeb frame. Conversely, if \(S_1,\ldots,S_k\) is a local Reeb frame for \(\mathcal D\) and \(\mathcal{D}
\) is maximally non-integrable, then there exists a unique $k$-contact form \(\bm\eta=\sum_{A=1}^k\eta^A\otimes e_A\) on \(U\) such that \(\ker\bm\eta=\mathcal D|_U\) and \(\eta^A(S_B)=\delta^A_B\) for $A,B=1,\ldots,k$. The vector fields \(S_1,\ldots,S_k\) are then the Reeb vector fields of \(\bm\eta\), as shown in \cite{LRS_24}. Indeed, one may prove the following result \cite[Theorem 3.14]{LRS_24}.

\begin{theorem}\label{Thm:LocalEqu}
A corank-\(k\) distribution \(\mathcal D\subset \T M\) is a \(k\)-contact distribution if and only if it is maximally non-integrable and, around every point of \(M\), it admits a local Reeb frame.
\end{theorem}

Theorem \ref{Thm:LocalEqu} will allow us to explain why Cartan distributions are $k$-contact distributions although their associated jet contact forms do not give rise to a $k$-contact structure for $r>1$. The difficulty to prove this fact is not the maximal non-integrability of the Cartan distribution, but the construction of an associated local Reeb frame. For many maximally non-integrable distributions, the lack of an associated local Reeb frame can be determined by the following construction \cite{SF_25,LRS_24}, which will also be relevant in further parts of this work.

\begin{definition}\label{def:derived_flag_kcontact_distribution}
Let \(\mathcal D\) be a regular distribution on $M$. The \emph{\(\mathcal D\)-derived flag} of \(\mathcal D\) is the increasing sequence of generalised distributions defined by \(\mathcal D^{(0)}:=\mathcal D\) and, recursively, by
\[
\mathcal D^{(m)}:=\mathcal D^{(m-1)}+[\mathcal D,\mathcal D^{(m-1)}],\qquad m\geq 1,
\]
where \([\mathcal D,\mathcal D^{(m-1)}]\) denotes the generalised distribution spanned by Lie brackets of locally defined vector fields taking values in $\mathcal D $ and $\mathcal D^{(m-1)}$, respectively. 
\end{definition}

Another relevant definition is the following one \cite{BorMontgomery2009,Sussmann1973}. 
\begin{definition} Given a generalised distribution $\mathcal{D}$ on $M$, its {\it symmetry distribution} is the generalised distribution $\mathfrak{sym}(\mathcal D)$ on $M$ of the form
$$
\mathfrak{sym}(\mathcal{D})_x=\{X_x\in \T_xM:\exists X\in \mathfrak{X}(M),[X,\Gamma(\mathcal D)]\subset \Gamma(\mathcal D)\}.
$$
\end{definition}

\begin{definition}
Let \((M,\mathcal D_M)\) and \((N,\mathcal D_N)\) be \(k\)- and \(k'\)-contact manifolds, respectively. A \emph{transcontact map} from \(M\) to \(N\) is a mapping \(\Phi:M\to N\) such that \(\T\Phi(\mathcal D_M)\subset\mathcal D_N\). If a transcontact map $\Phi$ has an inverse that is also a transcontact map, the map is called {\it a $k$-contactomorphism}. 
\end{definition}

The identity map is a transcontact map from a $k$-contact manifold onto itself. Moreover, the composition of transcontact maps is again a transcontact map. It follows immediately that \(k\)-contact manifolds, with their transcontact map as morphisms, form a category. 

\begin{example}\label{Ex:FirstOrder}
Let us recall the description of the first-order jet bundle of a vector bundle $\pi:E\rightarrow Q$ of rank $k$ as a $k$-contact manifold. Consider adapted coordinates \((x^i,y^A,y^A_i)\) on \(U\subset J^1\pi\). Then, the Cartan distribution reads locally on $U$ as
\[
 \mathcal C^1_\pi=\left\langle \frac{\partial}{\partial x^i}+\sum_{A=1}^k y^A_i\frac{\partial}{\partial y^A},\frac{\partial}{\partial y^B_i}\,:\,i=1,\ldots,n,\,\,B=1,\ldots,k\right\rangle.
\]
Moreover, the \(\mathbb{R}^k\)-valued one-form on $U$ of the form
\[
 \bm\eta=\sum_{A=1}^k\left(\dd y^A-\sum_{i=1}^n y^A_i\dd x^i\right)\otimes e_A
\]
becomes a $k$-contact form on $U$, which
satisfies \(\ker\bm\eta=\mathcal C^1_\pi|_U\), and whose Reeb vector fields are \(R_1=\partial/\partial y^1,\ldots,R_k=\partial/\partial y^k\). Applying the above procedure on a neighbourhood $U$ of every point of $J^1\pi$, the generated $k$-contact forms  and the Reeb vector fields change, but the kernels of the $k$-contact forms always reproduce the Cartan distribution $J^1\pi$. Hence, the Cartan distribution of $J^1\pi$ is a $k$-contact distribution, and
\((J^1\pi,\mathcal C^1_\pi)\) is a \(k\)-contact manifold.  
\end{example}

\begin{example}\label{ex:covering_kcontact}
Let \(W=\mathbb R^\ell\) with coordinates \(w^1,\ldots,w^\ell\) and consider the product \(\widehat U=U\times W\), where \(U\subset J^1\pi\) is as in Example~\ref{Ex:FirstOrder}. Let \(A^1,\ldots,A^\ell\) be closed basic one-forms on $U$, pulled back to $J^1\pi$ from the base $Q$ via $\pi_1$, so that locally \(A^a=\sum_{i=1}^nA^a_i(x)\dd x^i\) for $a=1,\ldots,\ell$. Assume $x^i,y^A_i,w^j$ to be defined on $\hat{U}$ in the natural manner.
Let us construct the lifted total derivatives as follows
\[
 \widehat D_i=\frac{\partial}{\partial x^i}+\sum_{A=1}^k y^A_i\frac{\partial}{\partial y^A}+\sum_{a=1}^\ell A^a_i(x)\frac{\partial}{\partial w^a},\qquad i=1,\ldots,n,
\]
and the distribution by
\[
 \widehat{\mathcal C}=\left\langle \widehat D_i,\frac{\partial}{\partial y^A_i}\,:\,i=1,\ldots,n,\quad A=1,\ldots,k\right\rangle\subset \T\widehat U.
\]
Equivalently, \(\widehat{\mathcal C}\) is the kernel of 
\[
\widehat{\bm\eta}=\sum_{A=1}^k\left(\dd y^A-\sum_{i=1}^n y^A_i\dd x^i\right)\otimes e_A+\sum_{a=1}^\ell(\dd w^a- A^a)\otimes \varepsilon_a
\]
where \(\{e_A\}_{A=1}^k\) is the canonical basis of \(\mathbb R^k\) and \(\{\varepsilon_a\}_{a=1}^\ell\) is the canonical basis of \(\mathbb R^\ell\). Then, $\widehat{\bm\eta}$ is a \((k+\ell)\)-contact form on \(\widehat U\). Moreover, \(\dd\theta^A=\sum_{i=1}^n \dd x^i\wedge \dd y^A_i\), and
\[
 \ker\dd\widehat{\bm\eta}
 =
 \left\langle 
 \frac{\partial}{\partial y^A},\frac{\partial}{\partial w^a}
 :A=1,\ldots,k,\quad a=1,\ldots,\ell
 \right\rangle,
\]
and therefore
\(
 \T\widehat U=\ker\widehat{\bm\eta}\oplus\ker\dd\widehat{\bm\eta}.
\)
The corresponding Reeb vector fields are
\(
 R_A=\frac{\partial}{\partial y^A},\,\,S_a=\frac{\partial}{\partial w^a},
\) for $A=1,\ldots,k$ and $a=1,\ldots,\ell$. 
Thus, \((\widehat U,\widehat{\mathcal C})\) is  a \((k+\ell)\)-contact manifold.
\end{example}

Let us recall the definition of polarisation employed in \cite{LRS_24}, which is a distributional generalisation of the initial one \cite{Riv_22}. The notion has been renamed here because, as shown later in this paper, it is a particular case adapted to a first-order jet model of a notion that is much more general.

\begin{definition}
Let \((M,\mathcal D)\) be a \(k\)-contact manifold and assume that $\dim M=nk+n+k$ for some positive integers $n,k$. A \emph{first-order polarisation} of \((M,\mathcal D)\) is an integrable subbundle \(\mathcal P\subset\mathcal D\) of rank \(nk\). 
A triple \((M,\mathcal D,\mathcal P)\) satisfying these conditions is called a \emph{polarised \(k\)-contact manifold of first order}.
\end{definition}

The classical Darboux theorem for polarised co-oriented \(k\)-contact manifolds of first order \cite{Riv_22,LRS_24} says that, if \((M,\bm\eta,\mathcal P)\) is a co-oriented \(k\)-contact manifold of dimension \(n+nk+k\) with an integrable distribution \(\mathcal P\subset\ker\bm\eta\) of rank \(nk\), then around every point there are coordinates \((x^i,y^A,y^A_i)\) such that
\begin{equation}\label{eq:Mod1JetkCon}
 \bm\eta=\sum_{A=1}^k\left(\dd y^A-\sum_{i=1}^n y^A_i\dd x^i\right)\otimes e_A,\quad  R_B=\frac{\partial}{\partial y^B},
\end{equation}
where $B=1,\ldots,k$. 
Moreover, $\mathcal P$ admits locally a complement given by the horizontal distribution $\langle D_i=\partial/\partial x^i+\sum_{A=1}^k y^A_i\partial/\partial y^A:i=1,\ldots,n\rangle$ and hence $\mathcal P$ is a  first-order polarisation. 
 
\begin{definition}\label{def:eta_hamiltonian_bracket}
Let \((M,\boldsymbol\eta)\) be a co-oriented \(k\)-contact manifold, and let \(\mathfrak X_{\boldsymbol\eta}(M)\) denote the Lie algebra of \(\boldsymbol\eta\)-Hamiltonian \(k\)-contact vector fields. The \emph{bracket of two \(\boldsymbol\eta\)-Hamiltonian functions} is defined by 
\[
\{\boldsymbol h_X,\boldsymbol h_Y\}_{\boldsymbol\eta}:=\boldsymbol h_{[X,Y]}=-\iota_{[X,Y]}\boldsymbol\eta,\qquad X,Y\in\mathfrak X_{\boldsymbol\eta}(M).
\]

\end{definition}
It can be proved that the bracket of $\bm \eta$-Hamiltonian functions is a Lie bracket and a Lie algebra morphism for the space $X\in \mathfrak{X}_{\bm\eta}(M)\mapsto -\iota_X\bm\eta\in {\rm Ham}(M,\bm\eta)$ \cite{LRS_24}.

Finally, it is useful to consider the following definition.

\begin{definition} A {\it $k'$-contact submanifold} of a $k$-contact manifold $(M,\mathcal{D})$ is a submanifold $N\subset M$ such that $\T N\cap \mathcal{D}$ is a $k'$-contact distribution.
\end{definition}

\section{A clarifying example}\label{Sec:IntrExample}

It has been clear since the first developments of \(k\)-contact geometry that first-order jet manifolds admit, locally, natural \(k\)-contact forms induced by the standard jet contact forms in adapted coordinates \cite{Riv_22,LRS_24}. But, as shown below, the  jet contact forms for $J^r\pi$ do not give rise to a \(k\)-contact structure when $r>1$. It is therefore noteworthy that the Cartan distribution of any $J^r\pi$ is nevertheless an \(N^r_\pi\)-contact distribution, as it will be shown in Section \ref{Sec:jetkcontact} for the first time in the literature. Let us illustrate this phenomenon in a second-order jet manifold $J^2\pi$. The fact  \(r=2\) allows us to show that standard jet contact forms do not give rise to a \(k\)-contact structure when $r>1$. We choose \(n=2\) because it allows us to construct a jet manifold that can be used to study partial differential equations \cite{Olver1993}. Finally, \(m=2\) is chosen to see how the number of dependent variables affects the form of the $k$-contact form locally associated with the Cartan distribution of $J^2\pi$.   

Let \(\pi:E=\mathbb R^2\times\mathbb R^2\to\mathbb R^2=M\), \((x^1,x^2,u,v)\mapsto(x^1,x^2)\) and consider induced global adapted coordinates on $J^2\pi$ of the form
\[
 (x^1,x^2,q^u,q^v,q^u_1,q^u_2,q^v_1,q^v_2,q^u_{11},q^u_{12},q^u_{22},q^v_{11},q^v_{12},q^v_{22}),
\]
where \(q^\alpha_i,q^\alpha_{ij}\), with \(\alpha\in\{u,v\}\) and \(1\leq i\leq j\leq2\), are  the corresponding jet coordinates for $q^u:=u,q^v:=v$. Then, \(\dim J^2\pi=14\) and the Cartan distribution is
\[
 \mathcal C^2_\pi=\left\langle D_1,D_2,\frac{\partial}{\partial q^\alpha_{11}},\frac{\partial}{\partial q^\alpha_{12}},\frac{\partial}{\partial q^\alpha_{22}}\right\rangle_{\alpha\in\{u,v\}},\,\,
D_i=\frac{\partial}{\partial x^i}+\sum_{\alpha\in\{u,v\}}\left(q^\alpha_i\frac{\partial}{\partial q^\alpha}+q^\alpha_{i1}\frac{\partial}{\partial q^\alpha_1}+q^\alpha_{i2}\frac{\partial}{\partial q^\alpha_2}\right),
\]
for $ i=1,2$ 
with \(q^\alpha_{12}=q^\alpha_{21}\) and $\alpha\in \{u,v\}$. Thus, \(\operatorname{rank}\mathcal C^2_\pi=8\), and its codimension is \(6\). This agrees with
\[
 N^2_\pi=m\binom{n+r-1}{r-1}=2\binom{3}{1}=6.
\]

Let \(\{e_\alpha^i\}_{\alpha\in\{u,v\}}^{i=0,1,2}\) be the canonical basis of \(\mathbb R^6\). The standard jet contact forms on $J^2\pi$ are
\[
 \theta^\alpha_0=\dd q^\alpha-q^\alpha_1\dd x^1-q^\alpha_2\dd x^2,\quad \theta^\alpha_1=\dd q^\alpha_1-q^\alpha_{11}\dd x^1-q^\alpha_{12}\dd x^2,\quad \theta^\alpha_2=\dd q^\alpha_2-q^\alpha_{12}\dd x^1-q^\alpha_{22}\dd x^2,
\]
for $\alpha\in \{u,v\}$. 
They define the \(\mathbb R^6\)-valued one-form
\(
 \bm\Theta=\sum_{\alpha\in\{u,v\}}\left(\theta^\alpha_0\otimes e_\alpha^0+\theta^\alpha_1\otimes e_\alpha^1+\theta^\alpha_2\otimes e_\alpha^2\right).
\)
By construction, \(\ker\bm\Theta=\mathcal C^2_\pi\). Nevertheless, \(\bm\Theta\) is not a six-contact form. Indeed, if
\[
 \Omega^\alpha_0=\dd x^1\wedge\dd q^\alpha_1+\dd x^2\wedge\dd q^\alpha_2,\qquad \Omega^\alpha_1=\dd x^1\wedge\dd q^\alpha_{11}+\dd x^2\wedge\dd q^\alpha_{12},\qquad \Omega^\alpha_2=\dd x^1\wedge\dd q^\alpha_{12}+\dd x^2\wedge\dd q^\alpha_{22},
\]
then
\[
 \dd\bm\Theta=\sum_{\alpha\in\{u,v\}}\left(\Omega^\alpha_0\otimes e_\alpha^0+\Omega^\alpha_1\otimes e_\alpha^1+\Omega^\alpha_2\otimes e_\alpha^2\right)\Longrightarrow
\ker\dd\bm\Theta=\left\langle\frac{\partial}{\partial u},\frac{\partial}{\partial v}\right\rangle.
\]
In particular, \(\ker\bm\Theta\cap\ker\dd\bm\Theta=\{0\}\). In fact, one can see that the maximally non-integrability of the distribution \(\ker\bm\Theta=\mathcal C^2_\pi\) follows from expression \eqref{eq:iso} and previous calculations. However,
\[
 \dim(\ker\bm\Theta\oplus\ker\dd\bm\Theta)=\dim\mathcal C^2_\pi+\dim\ker\dd\bm\Theta=8+2=10<14=\dim J^2\pi,
\]
so \(\ker\bm\Theta\oplus\ker\dd\bm\Theta\neq \T J^2\pi\). Thus, the standard higher-order jet contact forms on \(J^2\pi\) recover the Cartan distribution \(\mathcal C^2_\pi\) as their common kernel, but they do not form a six-contact structure. This does not imply that \(\mathcal C^2_\pi\) fails to be a six-contact distribution; it only means that \(\bm\Theta\) is not an adapted six-contact form.

Let us use Theorem \ref{Thm:LocalEqu} to show that \(\mathcal C^2_\pi\) is a six-contact distribution. A commuting frame of infinitesimal symmetries of \(\mathcal C^2_\pi\), spanning a complement of \(\mathcal C^2_\pi\) in \(\T J^2\pi\), is given by
\[
 R_\alpha^0=\frac{\partial}{\partial q^\alpha}=\operatorname{pr}^2\left(\frac{\partial}{\partial q^\alpha}\right),\qquad R_\alpha^i=x^i\frac{\partial}{\partial q^\alpha}+\frac{\partial}{\partial q^\alpha_i}=\operatorname{pr}^2\left(x^i\frac{\partial}{\partial q^\alpha}\right),\qquad \alpha\in\{u,v\},\quad i=1,2.
\]
These vector fields commute pairwise, since their coefficients depend only on \(x^1,x^2\) and none of them differentiates the base coordinates. They preserve the Cartan distribution \(\mathcal C^2_\pi\) because they are prolongations to \(J^2\pi\) of vector fields on \(E\simeq\mathbb{R}^4\). Moreover, \(R_\alpha^0,R_\alpha^1,R_\alpha^2\), for \(\alpha\in\{u,v\}\), are linearly independent at every point of \(J^2\pi\) and span a complement to \(\mathcal C^2_\pi\). Hence, the distributional characterisation of \(k\)-contact distributions in Theorem \ref{Thm:LocalEqu} implies that \(\mathcal C^2_\pi\) is locally described by a six-contact form $\bm\eta^2_\pi$. In the present global coordinate system, the adapted six-contact form is obtained by imposing \((\eta^2_\pi)^\alpha_i(R_\beta^j)=\delta^\alpha_\beta\delta^i_j\) for \(\alpha,\beta\in\{u,v\}\) and \(i,j\in\{0,1,2\}\) and $\bm\eta^2_\pi|_{\mathcal{C}^2_\pi}=0$. The unique adapted \(\mathbb R^6\)-valued one-form satisfying previous conditions is 
\[
 \bm\eta^2_\pi=\sum_{\alpha\in\{u,v\}}\left((\theta^\alpha_0-x^1\theta^\alpha_1-x^2\theta^\alpha_2)\otimes e_\alpha^0+\theta^\alpha_1\otimes e_\alpha^1+\theta^\alpha_2\otimes e_\alpha^2\right).
\]
A direct computation gives
\[
 \dd\bm\eta^2_\pi=\sum_{\alpha\in\{u,v\}}\left(-(x^1\Omega^\alpha_1+x^2\Omega^\alpha_2)\otimes e_\alpha^0+\Omega^\alpha_1\otimes e_\alpha^1+\Omega^\alpha_2\otimes e_\alpha^2\right).
\]
Thus,
\(
 \ker\dd\bm\eta^2_\pi=\left\langle R_\alpha^0,R_\alpha^1,R_\alpha^2\right\rangle_{\alpha\in\{u,v\}},
\)
and, since this distribution is complementary to \(\mathcal C^2_\pi=\ker\bm\eta^2_\pi\), one obtains
\[
 \T J^2\pi=\ker\bm\eta^2_\pi\oplus\ker\dd\bm\eta^2_\pi.
\]
Therefore, \(\bm\eta^2_\pi\) is an adapted six-contact form for the Cartan distribution \(\mathcal C^2_\pi\). Notice that \(\bm\eta^2_\pi\) has the same kernel as the standard contact form \(\bm\Theta\), but the vector fields \(R_\alpha^0,R_\alpha^1,R_\alpha^2\), with \(\alpha\in\{u,v\}\), constitute a Reeb frame.

This example already displays several structures that will be used systematically in further parts of this work. First,   \(\ker\bm\Theta\) matches the Cartan distribution, but $\bm\Theta$ does not give rise to an adapted six-contact form. Second, the adapted six-contact form \(\bm\eta^2_\pi\) is obtained by choosing a commuting transverse frame of Lie symmetries of the Cartan distribution $\mathcal{C}^r_\pi$ giving rise to a local Reeb frame. Third, 
\[
 \mathcal{G}^2_\pi=\ker \T\pi_{2,1}=\left\langle\frac{\partial}{\partial q^\alpha_{11}},\frac{\partial}{\partial q^\alpha_{12}},\frac{\partial}{\partial q^\alpha_{22}}\right\rangle_{\alpha\in\{u,v\}}
\]
is contained in \(\mathcal C^2_\pi\). These three ingredients will reappear in the general construction of an \(N^r_\pi\)-contact form for arbitrary Cartan distributions. The present example also illustrates that our construction for $J^2\pi$ is not a peculiarity but a general feature of jet manifolds.

The analogy goes much further. Note that the canonical jet structure \cite{Campos_2010} on \(J^2\pi\) is
\[
 \mathfrak{S}_\pi^2=\sum_{\alpha\in\{u,v\}}\left(\theta^\alpha_0\otimes\frac{\partial}{\partial q^\alpha}+\theta^\alpha_1\otimes\frac{\partial}{\partial q^\alpha_1}+\theta^\alpha_2\otimes\frac{\partial}{\partial q^\alpha_2}\right).
\]
Indeed, one can understand the  $\partial/\partial q^\alpha$ and $\partial/\partial q^\alpha_{i}$ as a basis of sections of $\Gamma(\pi_{2,1}^*V\pi_1)$. 

Let \(\mathfrak R:\mathbb R^6\to \mathfrak{X}(J^2\pi)\) be the linear mapping defined by \(\mathfrak R(e_\alpha^i)=R_\alpha^i\) for $\alpha\in \{u,v\}$ and $i=0,1,2$. Then,
\[
 \mathfrak R\circ\bm\eta^2_\pi=\sum_{\alpha\in\{u,v\}}\left((\theta^\alpha_0-x^1\theta^\alpha_1-x^2\theta^\alpha_2)\otimes R_\alpha^0+\theta^\alpha_1\otimes R_\alpha^1+\theta^\alpha_2\otimes R_\alpha^2\right)=\mathfrak{S}_\pi^2.
\]
Indeed, for each \(\alpha\in\{u,v\}\), one has
\begin{multline*}
 (\theta^\alpha_0-x^1\theta^\alpha_1-x^2\theta^\alpha_2)\otimes\frac{\partial}{\partial q^\alpha}+\theta^\alpha_1\otimes\left(x^1\frac{\partial}{\partial q^\alpha}+\frac{\partial}{\partial q^\alpha_1}\right)+\\\theta^\alpha_2\otimes\left(x^2\frac{\partial}{\partial q^\alpha}+\frac{\partial}{\partial q^\alpha_2}\right)=\theta^\alpha_0\otimes\frac{\partial}{\partial q^\alpha}+\theta^\alpha_1\otimes\frac{\partial}{\partial q^\alpha_1}+\theta^\alpha_2\otimes\frac{\partial}{\partial q^\alpha_2}.
\end{multline*}
Thus, the adapted six-contact form gives a geometric realisation of the Spencer operator in this second-order jet model. More geometrically, $[\mathfrak{R}\circ \bm\eta^2_\pi]$ is the canonical structure on $J^2\pi$ expressed in terms of a Reeb frame (see \cite[Definition~4.8]{Campos_2010}). On the other hand, one can consider the contraction of the differential of a function $F\in C^\infty(J^2\pi)$ with $\mathfrak{R}\circ \bm\eta^2_\pi$ and one obtains the vertical differential up to a truncated order  \cite{Olver_Pohjanpelto_2008,Olver_1995_Equivalence}.

Let us also describe the usual Spencer contraction by means of our six-contact form. Write the Cartan distribution in the adapted form \(\mathcal C^2_\pi=\mathcal H^2_\pi\oplus\mathcal{G}^2_\pi\), where \(\mathcal H^2_\pi=\langle D_1,D_2\rangle\) and \(\mathcal{G}^2_\pi=\ker \T\pi_{2,1}\). Since \(\mathcal C^2_\pi=\ker\bm\eta^2_\pi\), for any \(X,Y\in\Gamma(\mathcal C^2_\pi)\) one has \(\dd\bm\eta^2_\pi(X,Y)=-\bm\eta^2_\pi([X,Y])\). Thus, the mixed part of \(\dd\bm\eta^2_\pi\) on \(\mathcal H^2_\pi\times\mathcal{G}^2_\pi\) measures the failure of a vector field in $\ker \T\pi_{2,1}$ to commute with a horizontal Cartan derivative out of the Cartan distribution. More specifically, in our model, if \(p\in J^2\pi\) and \(x=\pi_0(p)\), we have the natural identifications 
\[
 a^1\frac{\partial}{\partial x^1}+a^2\frac{\partial}{\partial x^2}\in \T_xM\longmapsto a^1D_1|_p+a^2D_2|_p\in(\mathcal H^2_\pi)_p,
\]
and
\begin{multline*}\sum_{\alpha\in\{u,v\}}\left(A^\alpha_{11}\frac{\partial}{\partial q^\alpha_{11}}+A^\alpha_{12}\frac{\partial}{\partial q^\alpha_{12}}+A^\alpha_{22}\frac{\partial}{\partial q^\alpha_{22}}\right)\in(\mathcal{G}^2_\pi)_p\\\longmapsto \sum_{\alpha\in\{u,v\}}\left(A^\alpha_{11}\dd x^1\odot\dd x^1+A^\alpha_{12}\dd x^1\odot\dd x^2+A^\alpha_{22}\dd x^2\odot\dd x^2\right)\otimes\frac{\partial}{\partial q^\alpha}\in S^2\T_x^*M\otimes V_p\pi,
\end{multline*}
with the usual jet-coordinate convention for symmetric multi-indices. Hence, if
\[
 h=a^1D_1+a^2D_2,\qquad A=\sum_{\alpha\in\{u,v\}}\left(A^\alpha_{11}\frac{\partial}{\partial q^\alpha_{11}}+A^\alpha_{12}\frac{\partial}{\partial q^\alpha_{12}}+A^\alpha_{22}\frac{\partial}{\partial q^\alpha_{22}}\right)\in\mathcal{G}^2_\pi,
\]
then
\[
 \mathfrak R\circ\dd\bm\eta^2_\pi(h,A)=\sum_{\alpha\in\{u,v\}}\left((a^1A^\alpha_{11}+a^2A^\alpha_{12})\frac{\partial}{\partial q^\alpha_1}+(a^1A^\alpha_{12}+a^2A^\alpha_{22})\frac{\partial}{\partial q^\alpha_2}\right),
\]
where the right-hand part can be understood naturally as an element in $\T_x^*M\otimes V_p\pi$. 
Under the previous identifications, this is precisely the coordinate expression of the algebraic Spencer contraction \((h,A)\mapsto\iota_hA\), now in the genuine PDE case \(\T_xM\times S^2\T_x^*M\otimes V_p\pi\to \T_x^*M\otimes V_p\pi\). Thus, \(\mathfrak{R}\circ \dd\bm\eta^2_\pi|_{\mathcal{H}^2_\pi\times \mathcal{G}^2_\pi}\) recovers the algebraic Spencer contraction.

\section{Jet bundles as \texorpdfstring{\(k\)}{k}-contact manifolds}\label{Sec:jetkcontact} 

We now explain why the example of the previous section is not accidental, but rather a particular case of a general construction. Recall that \[ N^r_\pi:=m\binom{n+r-1}{r-1} \] is the corank of the Cartan distribution \(\mathcal C^r_\pi\) on \(J^r\pi\). Let us prove that \(\mathcal C^r_\pi\) is an \(N^r_\pi\)-contact distribution. The key point is the existence, in every adapted coordinate system to $J^r
\pi$, of a commuting frame of Lie symmetries of the Cartan distribution $\mathcal{C}^r_\pi$ spanning a complement to $\mathcal{C}^r_\pi$: a (local) \emph{Reeb frame}. To simplify the notation, it is convenient to recall that $J!=\prod_{i=1}^n j_i!, x^J=\prod_{i=1}^n (x^i)^{j_i},$ and write
\[ 
 \binom{J}{I}=\frac{J!}{I!(J-I)!}.
\]
With this notation, \((x+y)^J=\sum_{I\leq J}\binom{J}{I}x^I y^{J-I}\). Let us show now a lemma that gives the explicit expression of a local Reeb frame for the Cartan distribution of $J^r\pi$ in adapted coordinates.

\begin{lemma}\label{lem:adapted_Nr_contact} Let \((x^i,u^\alpha_I)\), with \(\alpha=1,\ldots,m\) and \(|I|\leq r\), be an adapted coordinate system on an open subset \(U\subset J^r\pi\). Define
\begin{equation}\label{relReebCon} R^J_\alpha:=\sum_{I\leq J}\frac{x^{J-I}}{(J-I)!}\frac{\partial}{\partial u^\alpha_I} =\operatorname{pr}^{(r)}\left(\frac{x^J}{J!}\frac{\partial}{\partial u^\alpha}\right),\qquad |J|\leq r-1,\qquad\alpha=1,\ldots,m. 
    \end{equation} 
   The vector fields \(R^J_\alpha\) commute pairwise, are Lie symmetries of \(\mathcal C^r_\pi|_U\), and span a rank-\(N^r_\pi\) subbundle \(\mathcal R\subset \T U\) such that \( \T U=\mathcal C^r_\pi|_U\oplus\mathcal R. \) Moreover, 
    \begin{equation}\label{relConReeb} \frac{\partial}{\partial u^\alpha_J}=\sum_{I\leq J}(-1)^{|J-I|}\frac{x^{J-I}}{(J-I)!}R^I_\alpha,\qquad \alpha=1,\ldots,m,\qquad |J|\leq r-1. \end{equation} 
\end{lemma} 
    
    \begin{proof} The vector fields \(R^J_\alpha\) commute pairwise. Indeed, their coefficients depend only on the coordinates \(x^i\), whereas $R^J_\alpha x^i=0$ for all $i=1,\ldots,n$, $|J|\leq r-1$,  and $\alpha=1,\ldots,m$. The \(R^J_\alpha\) are also Lie symmetries of the Cartan distribution $\mathcal{C}_\pi^r$, because expressions \eqref{relReebCon} show that they are prolongations to \(U\subset J^r\pi\) of vector fields locally defined on \(E\). 
        
        Let us prove that the fields \(R^J_\alpha\) span a complement to \(\mathcal C^r_\pi|_U\). In particular, we show by induction on \(s=0,\ldots,r-1\) that, for every \(\alpha=1,\ldots,m\), one has 
        \begin{equation}\label{eq:claim} \left\langle R^J_\alpha:\alpha=1,\ldots,m,|J|\leq s\right\rangle = \left\langle \frac{\partial}{\partial u^\alpha_J}:\alpha=1,\ldots,m,|J|\leq s\right\rangle . 
        \end{equation}
        For \(s=0\), this follows from \(R^0_\alpha=\partial/\partial u^\alpha\) with $\alpha=1,\ldots,m$. Assume that the claim \eqref{eq:claim} holds up to order \(s-1\). If \(|J|=s\), then \[ R^J_\alpha=\frac{\partial}{\partial u^\alpha_J}+\sum_{I<J}\frac{x^{J-I}}{(J-I)!}\frac{\partial}{\partial u^\alpha_I}. \] By the induction hypothesis, all the lower-order fields \(\partial/\partial u^\alpha_I\), with \(I<J\), already belong to the span of the \(R^I_\alpha\) with \(|I|\leq s-1\). Hence, \(\partial/\partial u^\alpha_J\) belongs to the span of the \(R^I_\alpha\) with \(|I|\leq s\). The converse inclusion is immediate from the defining expression of the \(R^J_\alpha\). Consequently, relation \eqref{eq:claim} is satisfied for $s$, and by induction hypothesis for every $s\leq r-1$.
        
        Since \(\mathcal C^r_\pi|_U\) is locally spanned by the total derivatives \(D_1,\ldots,D_n\) and the fields \(\partial/\partial u^\alpha_K\) with \(|K|=r\), it follows that the \(R^J_\alpha\),  with \(|J|\leq r-1\) and $\alpha=1,\ldots,m$, span a rank-\(N^r_\pi\) distribution \(\mathcal R\) complementary to \(\mathcal C^r_\pi|_U\). 
        
        It remains to prove the inverse formula \eqref{relConReeb}. Put 
        \[ 
        \widetilde{\partial}^J_\alpha:=\sum_{I\leq J}(-1)^{|J-I|}\frac{x^{J-I}}{(J-I)!}R^I_\alpha\,\qquad  \alpha=1,\ldots,m,\qquad |J|\leq r-1. 
        \] 
        
        Using the expression of \(R^I_\alpha\), we get \[ \widetilde{\partial}^J_\alpha=\sum_{I\leq J}(-1)^{|J-I|}\frac{x^{J-I}}{(J-I)!}\sum_{L\leq I}\frac{x^{I-L}}{(I-L)!}\frac{\partial}{\partial u^\alpha_L} = \sum_{L\leq I\leq J}(-1)^{|J-I|}\frac{x^{J-I}}{(J-I)!}\frac{x^{I-L}}{(I-L)!} \frac{\partial}{\partial u^\alpha_L}. \] For fixed \(L\leq J\), the coefficient of \(\partial/\partial u^\alpha_L\) is \[ \sum_{L\leq I\leq J}(-1)^{|J-I|}\frac{x^{J-L}}{(J-I)!(I-L)!} = \frac{x^{J-L}}{(J-L)!}\sum_{L\leq I\leq J}(-1)^{|J-I|}\binom{J-L}{I-L}. \] Writing \(M=I-L\) and \(N=J-L\), this becomes \[ \frac{x^N}{N!}\sum_{M\leq N}(-1)^{|N-M|}\binom{N}{M} = \frac{x^N}{N!}(1-1)^N. \] 
        Hence, it is zero except in the case \(N=0\), that is, when \(L=J\), and in that case it equals \(1\). Therefore, \(\widetilde{\partial}^J_\alpha=\partial/\partial u^\alpha_J\), which proves \eqref{relConReeb}.
    
    \end{proof}

        The previous local Reeb frame $\{R_\alpha^J\}_{\alpha=1,\ldots,m}^{|J|\leq r-1}$ is not unique. A different adapted coordinate system of $J^r\pi$ gives a different local Reeb frame $\{R_\alpha^{'J}\}_{\alpha=1,\ldots,m}^{|J|\leq r-1}$.  
        In general, all local Reeb frames $\{R_\alpha^J\}_{\alpha=1,\ldots,m}^{|J|\leq r-1}$ on $J^r\pi$ need not glue to a globally defined Reeb distribution. If the jet bundle $J^r\pi$ admits a global adapted coordinate system, then the Reeb frame of Lemma~\ref{lem:adapted_Nr_contact} gives rise to a global frame of commuting Lie symmetries of the Cartan distribution spanning a complement to it. Then, $J^r\pi$ becomes a co-oriented \(N^r_\pi\)-contact manifold. In general, the Reeb frame is only locally defined, and the Cartan distribution is a locally \(N^r_\pi\)-contact distribution.

\begin{lemma}
The  Cartan distribution \(\mathcal C^r_\pi\subset \T J^r\pi\) is maximally non-integrable.
\end{lemma}
\begin{proof} The distribution $\mathcal{C}^r_\pi$ is maximally non-integrable if and only if it does not admit a non-vanishing tangent vector $v_p\in (\mathcal{C}^r_\pi)_p$ such that any extension to a vector field $X$ taking values in $\mathcal{C}^r_\pi$ is such that $[X,\Gamma(\mathcal{C}^r_\pi)]_p\subset  (\mathcal{C}^r_\pi)_p$. Assume such a vector field $X$ exists for $v_p\neq 0$. Then, for every $Y\in\Gamma(\mathcal{C}^r_\pi)$, one has $[X,Y]_p\in (\mathcal{C}^r_\pi)_p$. In particular, set $X=\sum_{i=1}^nf^iD_i+\sum_{|I|=r}\sum_{\alpha=1}^mf_\alpha^I\frac{\partial}{\partial u^\alpha_I}$ and recall that $[D_j,\partial/\partial u^\alpha_I]=-\partial/\partial u^\alpha_{I-\widehat{e}_j}$ for $I\geq \widehat{e}_j$, which does not take values in $\mathcal{C}^r_\pi$. Using this, the commutators of $X$ with all the $D_j$ show that the $f_\alpha^J(p)=0$ for $\alpha=1,\ldots,m$ and $|J|=r$. Then, the commutators of $X$ with all the \(\frac{\partial}{\partial u^\alpha_J}\) give that $f^1(p)=\ldots=f^n(p)=0$.  Hence, $X_p=v_p=0$, which contradicts the assumption that $v_p\neq 0$.
\end{proof}

Let us show now that every adapted coordinate system on $U\subset J^r\pi$ gives rise to a $N_\pi^r$-contact form on $U$ whose kernel is $\mathcal{C}^r_\pi|_U$. This will be a fundamental result in our theory. 
\begin{theorem}\label{cor:Cartan_Nr_contact} The Cartan distribution \(\mathcal C^r_\pi\) on \(J^r\pi\) is an \(N^r_\pi\)-contact distribution, where \(N^r_\pi=m\binom{n+r-1}{r-1}\). In every local adapted coordinate system \((x^i,u^\alpha_I)\) to $J^r\pi$ on $U$, an associated local adapted \(N^r_\pi\)-contact form is 
            \begin{equation}\label{eq:adapted_eta}\bm\eta^r_\pi=\sum_{\alpha=1}^m\sum_{|J|\leq r-1}\eta^\alpha_J\otimes e_\alpha^J,\qquad \eta^\alpha_J=\sum_{\substack{L\geq J\\ |L|\leq r-1}}(-1)^{|L-J|}\frac{x^{L-J}}{(L-J)!}\theta^\alpha_L, 
            \end{equation} 
             where \(\{e_\alpha^J\}_{\alpha=1,\ldots,m}^{|J|\leq r-1}\) is the canonical basis of \(\mathbb R^{N^r_\pi}\). The Reeb vector fields of \(\bm\eta^r_\pi\) are the fields \(R_\alpha^J\) of Lemma~\ref{lem:adapted_Nr_contact}. Moreover,  
             \[ 
             \left[\mathfrak{R}\circ \bm\eta^r_\pi\right]= \mathfrak{S}_\pi^r|_U.
             \]
            \end{theorem} 
        
        \begin{proof} Fix an adapted open coordinate domain \(U\subset J^r\pi\) and the Reeb distribution $\mathcal{R}$ spanned by the elements of the Reeb frame \(R^J_\alpha\) of Lemma~\ref{lem:adapted_Nr_contact}. Since \(\T U=\mathcal C^r_\pi|_U\oplus\mathcal R\), there is a unique \(\mathbb R^{N^r_\pi}\)-valued one-form \(\bm\eta^r_\pi=\sum_{\alpha=1}^m\sum_{|J|\leq r-1}\eta^\alpha_J\otimes e_\alpha^J\) such that \(\ker\bm\eta^r_\pi=\mathcal C^r_\pi|_U\) and \(\eta^\alpha_J(R^\beta_I)=\delta^\alpha_\beta\delta^J_I\) for $\alpha,\beta=1,\ldots,m,|I|,|J|<r$. The relation between the standard coframe given by the $\theta^\alpha_I$ and the adapted $\eta^\alpha_I$, for  $\alpha=1,\ldots,m$ and  $|I|<r$, namely 
            \begin{equation}\label{eq:relEtaTheta}    \theta^\alpha_I=\sum_{\substack{J\geq I\\ |J|\leq r-1}}\frac{x^{J-I}}{(J-I)!}\eta^\alpha_J,\qquad \eta^\alpha_I=\sum_{\substack{J\geq I\\ |J|\leq r-1}}(-1)^{|J-I|}\frac{x^{J-I}}{(J-I)!}\theta^\alpha_J, 
            \end{equation}
            follow from relations \eqref{relReebCon} and \eqref{relConReeb} and a proof similar to that for expression \eqref{relConReeb}. The above expressions give the displayed formula for \(\bm\eta^r_\pi\). 
            
            Let us prove that \(\bm\eta^r_\pi\) is an \(N^r_\pi\)-contact form. For every \(R_A=R^J_\alpha\), one has \(\mathcal L_{R_A}\bm\eta^r_\pi=0\). Indeed, if \(X\in\Gamma(\mathcal C^r_\pi|_U)\), then \((\mathcal L_{R_A}\eta^\beta_I)(X)=R_A(\eta^\beta_I(X))-\eta^\beta_I([R_A,X])=0\), since \(\eta^\beta_I(X)=0\) and \(R_A\) preserves \(\Gamma(\mathcal C^r_\pi)\). On the other hand, \((\mathcal L_{R_A}\eta^\beta_I)(R_B)=R_A(\eta^\beta_I(R_B))-\eta^\beta_I([R_A,R_B])=0\), since the functions \(\eta^\beta_I(R_B)\) are constant and the vector fields \(R_A\) commute. As \(\T U=\mathcal C^r_\pi|_U\oplus\mathcal R\), this proves \(\mathcal L_{R_A}\bm\eta^r_\pi=0\). Cartan's formula gives \(\iota_{R_A}\dd\bm\eta^r_\pi=0\), because \(\iota_{R_A}\bm\eta^r_\pi\) is constant, and hence \(\mathcal R\subset\ker\dd\bm\eta^r_\pi\). Since \(\bm\eta^r_\pi\) vanishes on \(\mathcal C^r_\pi|_U\), the Levi tensor of \(\mathcal C^r_\pi\) is represented by \(-\dd\bm\eta^r_\pi|_{\mathcal C^r_\pi\times\mathcal C^r_\pi}\). Hence, $\ker \dd\bm\eta^r_\pi|_{\mathcal C^r_\pi\times\mathcal C^r_\pi}=\ker\dd\bm\eta^r_\pi\cap \mathcal{C}^r_\pi$.
            The Cartan distribution is maximally non-integrable, so 
                        if \(v\in(\mathcal{C}^r_\pi\cap\ker\dd\bm\eta^r_\pi)_p\) and \(\dd\bm\eta^r_\pi(v,w)=0\) for every \(w\in(\mathcal C^r_\pi)_p\), it follows that
            \(v=0\). Therefore, \(\mathcal C^r_\pi|_U\cap\ker\dd\bm\eta^r_\pi=\{0\}\). As \(\mathcal R\subset\ker\dd\bm\eta^r_\pi\) and \(\operatorname{rank}\mathcal R=N^r_\pi=\operatorname{codim}\mathcal C^r_\pi\), it follows that \(\ker\dd\bm\eta^r_\pi=\mathcal R\). Consequently, 
            \[ 
            \T U=\ker\bm\eta^r_\pi\oplus\ker\dd\bm\eta^r_\pi, 
            \] 
            and \(\bm\eta^r_\pi\) is an \(N^r_\pi\)-contact form on $U$. Since this construction can be performed around every point in $J^r\pi$, it follows that  \((J^r\pi,\mathcal C^r_\pi)\) is an \(N^r_\pi\)-contact manifold. Finally, in adapted coordinates, the canonical jet structure on $J^r\pi$ is given by 
            \[ 
            \mathfrak{S}_\pi^r=\sum_{\alpha=1}^m\sum_{|I|\leq r-1}\theta^\alpha_I\otimes\frac{\partial}{\partial u^\alpha_I}, 
            \] 
            where $\partial/\partial u^\alpha_J$ with $|J|\leq r-1$ can be considered as a basis of $\T \pi_r/\ker \T\pi_{r,r-1}$. 
            Relations \eqref{relConReeb} and \eqref{relReebCon} and the dual relations \eqref{eq:relEtaTheta} show that \([\mathfrak{R}\circ \bm\eta^r_\pi]=\mathfrak{S}_\pi^r\). 
            
        \end{proof} 
        
        \begin{remark} In Theorem~\ref{cor:Cartan_Nr_contact}, the local form \(\bm\eta^r_\pi\), its Reeb frame \(R_\alpha^J\), and the complement \(\mathcal R\) depend on the chosen adapted coordinates. The Cartan distribution \(\mathcal C^r_\pi\), the vertical polarisation \( \mathcal{G}^r_\pi=\ker \T\pi_{r,r-1}\), and the   \(\mathfrak{S}_\pi^r\), related to the canonical structure on $J^r\pi$, are the invariant objects.  Note that, in local adapted coordinates,  $\mathfrak{R}\circ \bm \eta^r_\pi$  and $[\mathfrak{R}\circ \bm \eta^r_\pi]$ have the same coordinate expression.  
        \end{remark} 
        
        Thus, every finite-order jet bundle \(J^r\pi\) carries a canonical \(N^r_\pi\)-contact distribution, namely its Cartan distribution \(\mathcal C^r_\pi\). This does not mean that the adapted local \(N^r_\pi\)-contact forms are globally defined or canonical, unless $J^r\pi$ admits global adapted coordinates. Nevertheless, Theorem \ref{cor:Cartan_Nr_contact} shows that jet geometry of finite order lies naturally inside \(k\)-contact geometry. The following sections characterise which types of \(k\)-contact manifolds arise from jet bundles and explain how this viewpoint reformulates holonomic submanifolds, systems of PDEs, Lie symmetries, characteristics, reductions,  and many other topics treated in the geometry of jets and their applications in field theories and their potential generalisations.

\section{Polarisations for \texorpdfstring{$k$}{k}-contact manifolds}\label{Sec:polarisations}
\label{sec:polarisations}

The previous section proved that every $(J^r\pi,\mathcal{C}^r_\pi)$ is an  \(N^r_\pi\)-contact manifold and briefly analysed its relation with some of the natural structures on jet manifolds. In this and in the following section, we are interested in addressing the converse problem: which \(k\)-contact manifolds are locally or globally jet bundles of finite-order? The local answer to the first-order case of this question is solved by the $k$-contact Darboux theorem given in \cite{LRS_24,Riv_22} and the   polarisation notion defined in \cite{Riv_22}. In essence, the result says that the $k$-contact manifold $(M,\mathcal{D})$ has to admit a first-order polarisation $\mathcal{P}\subset\mathcal{D}$, which is locally diffeomorphic to some $\ker \T\pi_{1,0}$, while $\mathcal{D}$ is simultaneously  locally diffeomorphic to some $\mathcal{C}^1_\pi$  and the dimension of the manifold $M$ is $n+nk+k$.

For general jet manifolds, however, the usual first-order definition of polarisation is too rigid, because it assumes, for instance, that the manifold has dimension of the form \(n+nk+k\). Thus, our idea is that the higher-order theory requires a notion of polarisation that still singles out \(\ker \T\pi_{r,r-1}\) and  does not impose the first-order dimensional constraint on the manifold. 

\begin{definition}\label{def:higher_order_polarisation}
A \emph{polarisation} in a $k$-contact manifold \((M,\mathcal D)\) is an integrable Legendrian subbundle \(\mathcal P\subset\mathcal D\) that has maximal rank among all the integrable Legendrian subbundles of \(\mathcal D\). 
A \emph{polarised \(k\)-contact manifold} is a triple \((M,\mathcal D,\mathcal P)\).
\end{definition}

The fact that $\mathcal{P}$ is Legendrian ensures that there exists no other isotropic subbundle $\mathcal{P}'$ such that \(\mathcal P_p'\supsetneq\mathcal P_p\) at some $p\in M$. In fact, since $\mathcal{P}_p$ is Legendrian, it admits an isotropic complement $\mathcal{H}_p$ in $\mathcal{D}$. If such a \(\mathcal P'\) existed, then \(\operatorname{rank}\mathcal P_p'>\operatorname{rank}\mathcal P_p\) would force \(\mathcal H_p\cap\mathcal P'_p\neq 0\). For \(0\neq v\in\mathcal H_p\cap\mathcal P_p'\), the isotropy of \(\mathcal H\) and \(\mathcal P'\), together with \(\mathcal P_p\subset\mathcal P_p'\), gives \(v\in \mathcal{H}_p^{\perp_{\mathcal{D}}}\) and \(v\in \mathcal{P}_p^{\perp_{\mathcal{D}}}\), hence \(v\in \mathcal{D}_p^{\perp_{\mathcal{D}}}\). This contradicts that $\mathcal{D}^{\perp_{\mathcal{D}}}=0$, since $\mathcal{D}$ is maximally non-integrable.

The maximality of the rank of $\mathcal{P}$ prevents the notion of higher-order polarisation from covering smaller, although admissible, isotropic factors. For example, on a jet bundle \(J^r\pi\), the distribution \(\mathcal{G}^r_\pi\) is a Legendrian subbundle of rank $m\left(\begin{array}{c}n+r-1\\r\end{array}\right)$, but the horizontal distribution \(\langle D_1,\ldots,D_n\rangle\) is globally defined and gives rise to a Legendrian subbundle when $J^r\pi$ admits global adapted coordinates, which is of different nature. The maximality restricts then attention to the higher-order jet geometry relevant for applications.

In what follows, our polarisation notion will become a key whose features will allow us to describe many features of $k$-contact manifolds and jet geometry naturally. To show that our previous definition retrieves first-order polarisations as a particular case, and to prove further results of our paper, we have first to study how $\bm\eta^r_\pi$ retrieves Spencer contractions. 

\begin{theorem}
Let \((x^i,u^\alpha_I)\), with \(|I|\le r\), be adapted coordinates on \(U\subset J^r\pi\). Consider the local splitting
\[
\mathcal C^r_\pi|_U=\mathcal H^r_\pi\oplus\mathcal{G}^r_\pi|_U,\qquad \mathcal H^r_\pi=\langle D_1,\ldots,D_n\rangle,\qquad \mathcal{G}^r_\pi=\ker \T\pi_{r,r-1}.
\]
Let \(\bm\eta^r_\pi=\sum_{\alpha=1,\ldots,m}\sum_{|J|\le r-1}\eta^\alpha_J\bm{\otimes} e_\alpha^J\) be the adapted local \(N^r_\pi\)-contact form associated with the Reeb frame $R_\alpha^J$ with
\(
|J|\le r-1,\)\ for \(\alpha=1,\ldots,m,
\) on $U$, 
and let the linear map \(\mathfrak R:\mathbb{R}^{N^r_\pi}\rightarrow \mathfrak{X}(U)\) be the Reeb lift \(\mathfrak R(\bm e_\alpha^J)=R_\alpha^J\). Then, for every \(h=\sum_{i=1}^nh^iD_i\in\Gamma(\mathcal H^r_\pi)\) and every \(A=\sum_{\alpha=1,\ldots,m}\sum_{|K|=r}A^\alpha_K\partial/\partial u^\alpha_K\in\Gamma(\mathcal{G}^r_\pi)\), one has
\begin{equation}
\label{eq:SpencerContraction}
\mathfrak{R}\circ\dd\bm\eta^r_\pi(h,A)=\sum_{\alpha=1}^m\sum_{|I|=r-1}\left(\sum_{i=1}^n h^iA^\alpha_{I+\widehat{e}_i}\right)\frac{\partial}{\partial u^\alpha_I},
\end{equation}
in the adapted coordinate system. If $\mathcal{N}=\langle \partial/\partial u^{\alpha}_J:|J|=r-1,\alpha=1,\ldots,m\rangle$ and under the standard  identifications
\[
(\mathcal H^r_\pi)_p\ni D\mapsto \T\pi_r(D)\in \T_{\pi_r(p)}Q,\,\,(\mathcal{G}^r_\pi)_p\ni\!\!\sum_{\alpha=1}^m\sum_{|I|=r} \lambda^\alpha_I\frac{\partial}{\partial u_I^\alpha}\mapsto\!\!\! \sum_{\alpha=1}^m\sum_{|I|=r}\!\!\!\lambda_I^\alpha \odot_I \dd x^{\langle I\rangle}\otimes \frac{\partial}{\partial u^\alpha}\in  S^r\T_x^*Q\otimes V_p\pi,
\]
\[
\mathcal{N}_p\ni\sum_{\alpha=1}^m\sum_{|I|=r-1}\lambda_I^\alpha \frac{\partial}{\partial u^\alpha_I} \mapsto \sum_{\alpha=1}^m\sum_{|I|=r-1}\lambda_I^\alpha \odot_I\dd x^{\langle I\rangle}\otimes \frac{\partial}{\partial u^\alpha}\in  S^{r-1}\T_x^*Q\otimes V_p\pi,
\]
the bilinear map \((h,A)\mapsto\mathfrak R\circ\dd\bm\eta^r_\pi(h,A)\) is exactly the algebraic Spencer contraction \((h,A)\mapsto\iota_hA\).
\end{theorem}

\begin{proof}
Since \(\mathcal C^r_\pi=\ker\bm\eta^r_\pi\), for \(X,Y\in\Gamma(\mathcal C^r_\pi)\) one has
\(
\dd\bm\eta^r_\pi(X,Y)=-\bm\eta^r_\pi([X,Y]).
\)
Since it was proved that $[\mathfrak{R}\circ \bm\eta^r_\pi]=\mathfrak{S}^r_\pi$, the calculation of $\mathfrak{R}\circ\dd\bm\eta^r_\pi|_{\mathcal{H}^r_\pi\times \mathcal{G}^r_\pi}$ requires only to compute \(-\mathfrak{R}\circ \bm \eta^r_\pi([X,Y])=-\mathfrak{S}^r_\pi([X,Y])\), for \(X\in\Gamma(\mathcal H^r_\pi)\) and \(Y\in\Gamma(\mathcal{G}^r_\pi)\) in adapted coordinates. Using the standard expression for $\mathfrak{S}^r_\pi$ in terms of the jet contact forms, one obtains that, for \(|K|=r\), \(|L|<r\) and $\alpha,\beta=1,\ldots,m$, it follows that
\begin{equation}\label{eq:commutator_D_partial}
-\theta^\alpha_L\left(\left[D_i,\frac{\partial}{\partial u^\beta_K}\right]\right)=\delta^\alpha_\beta\,\delta_{K,L+\widehat{e}_i},\quad \alpha=1,\ldots,m,\,\, |L|<r, \,\,|K|=r.
\end{equation}
and
\[
\mathfrak{R}\circ\dd\bm\eta^r_\pi\left(D_i,\frac{\partial}{\partial u^\beta_K}\right)=
\sum_{\substack{\alpha=1,\ldots,m\\ |L|\leq r-1}}\delta^\alpha_\beta\,\delta_{K,L+\widehat{e}_i}\frac{\partial}{\partial u^\alpha_L}=\sum_ {|L|\leq r-1}\,\delta_{K,L+\widehat{e}_i}\frac{\partial}{\partial u^\beta_L}.
\]
The expression is zero when \(K_i=0\). Otherwise,
\[
\mathfrak R\circ\dd\bm\eta^r_\pi\left(D_i,\frac{\partial}{\partial u^\beta_K}\right)=\frac{\partial}{\partial u^\beta_{K-\widehat{e}_i}}.
\]
By bilinearity, for \(h=\sum_{i=1}^nh^iD_i\) and \(A=\sum_{\alpha=1}^m\sum_{|K|=r}A^\alpha_K\partial/\partial u^\alpha_K\), one obtains
\[
\mathfrak R\circ\dd\bm\eta^r_\pi(h,A)=\sum_{\alpha=1}^m\sum_{|K|=r}\sum_{i=1}^n h^iA^\alpha_K\,\mathfrak R\circ\dd\bm\eta^r_\pi\left(D_i,\frac{\partial}{\partial u^\alpha_K}\right).
\]
Reindexing \(K=I+\widehat{e}_i\), with \(|I|=r-1\), gives
\[
\mathfrak R\circ\dd\bm\eta^r_\pi(h,A)=\sum_{\alpha=1}^m\sum_{|I|=r-1}\left(\sum_{i=1}^n h^iA^\alpha_{I+\widehat{e}_i}\right)\frac{\partial}{\partial u^\alpha_I}.
\]
Under the above  identifications, the last expression is exactly the usual formula for \(\iota_hA\). This proves the claim.
\end{proof}
\begin{lemma}\label{lem:jet_symbol_dominant}
Let \((J^r\pi,\mathcal C^r_\pi)\) be the \(r\)-th order jet model. Then,
\(
\mathcal{G}^r_\pi:=\ker \T\pi_{r,r-1}
\)
is a polarisation of \((J^r\pi,\mathcal C^r_\pi)\).
\end{lemma}

\begin{proof}
In adapted coordinates \((x^i,u^\alpha_I)\) with \(|I|\leq r\) on $U\subset J^r\pi$, the Cartan distribution splits locally as
\[
\mathcal C^r_\pi|_U=\mathcal H_\pi^r\oplus\mathcal{G}^r_\pi|_U,\qquad \mathcal H_\pi^r=\langle D_1,\ldots,D_n\rangle,\qquad \mathcal{G}^r_\pi|_U=\left\langle\frac{\partial}{\partial u^\alpha_K}:|K|=r\right\rangle=\ker \T\pi_{r,r-1}.
\]
Both \(\mathcal H_\pi^r\) and \(\mathcal{G}^r_\pi\) are regular integrable distributions, since \([D_i,D_j]=0\) and \(\mathcal{G}^r_\pi=\ker \T\pi_{r,r-1}\). Thus, both spaces are also isotropic for \(\dd\bm\eta^r_\pi\) because the latter vanishes on integrable distributions contained in $\mathcal{C}^r_\pi$. Hence, \(\mathcal{G}^r_\pi\) is an integrable Legendrian subbundle.

It remains to prove that \(\mathcal{G}^r_\pi\) has maximal rank among the integrable Legendrian subbundles of \(\mathcal C^r_\pi\). This is a pointwise linear algebra statement. Consider the restriction of $\dd \bm\eta^r_\pi$ to $\mathcal{H}_\pi^r\times\mathcal{G}^r_\pi$, which amounts to the Spencer contraction. 
Let \(W\subset(\mathcal C^r_{\pi})_\theta\) be an isotropic subspace. Denote by \(A\subset (\mathcal{H}_\pi^r)_\theta\) the projection of \(W\) onto \((\mathcal{H}_\pi^r)_\theta\), and put \(\ell:=\dim A\). Since \(0\to W\cap \mathcal{G}^r_\pi\to W\to A\to0\) is exact, one has
\begin{equation}\label{eq:dimW}
\dim W=\ell+\dim(W\cap (\mathcal{G}^r_\pi)_\theta).
\end{equation}
If \(B:=W\cap (\mathcal{G}^r_\pi)_\theta\), then the isotropy of \(W\) and $(\mathcal{G}^r_\pi)_\theta$ imply that every element of \(B\) is annihilated by contraction with every tangent vector of \(A\). Therefore,
\[
B\subset\{T\in (\mathcal{G}^r_\pi)_\theta:\dd\bm\eta^r_\pi(a,T)=0, \forall a\in A\}.
\]
After choosing a basis of \((\mathcal{H}_\pi^r)_\theta\) adapted to \(A\) and using \eqref{eq:SpencerContraction}, the  right-hand side of the above expression is naturally identified with \(S^r((\mathcal{H}^r_\pi)_\theta/A )^*\otimes V_\theta\pi\). Hence, in view of \eqref{eq:dimW}, it follows that  
\[
\dim(W\cap (\mathcal{G}^r_\pi)_\theta)\leq m\binom{n-\ell+r-1}{r} \Longrightarrow \dim W\leq \ell+m\binom{n-\ell+r-1}{r}.
\]
Since the number of degree \(r\) symmetric monomials in \(n\) variables involving at least one of the \(\ell\) variables spanning \(A\) is at least \(\ell\), we have
\begin{equation}\label{eq:ConPol}
\binom{n+r-1}{r}-\binom{n-\ell+r-1}{r}\geq \ell\Longrightarrow 
\ell+m\binom{n-\ell+r-1}{r}\leq m\binom{n+r-1}{r}=\operatorname{rank}\mathcal{G}^r_\pi.
\end{equation}
Thus, every isotropic subspace \(W\subset(\mathcal C^r_{\pi})_\theta\) has dimension bounded by the rank of \(\mathcal{G}^r_\pi.
\)
Since \(\mathcal{G}^r_\pi\) itself is Legendrian, it has maximal possible rank among admissible isotropic subbundles. Therefore, \(\mathcal{G}^r_\pi\) satisfies  Definition~\ref{def:higher_order_polarisation}  and becomes a polarisation of \((J^r\pi,\mathcal C^r_\pi)\).
\end{proof}

\begin{proposition}
If \((M,\mathcal D,\mathcal P)\) is a first-order polarised \(k\)-contact manifold, $\mathcal{P}$ is a polarisation in the sense of Definition \ref{def:higher_order_polarisation}.
\end{proposition}

\begin{proof} Applying the classic Darboux theorem for \(k\)-contact manifolds Theorem \cite{Riv_22}, one can assume that \((M,\mathcal D,\mathcal{P})\) is locally given by the standard \(k\)-contact Darboux model \ref{eq:Mod1JetkCon}, which gives a local polarised $k$-contactomorphism with a first-order jet manifold so that \(\mathcal D\) is mapped into the Cartan distribution and \(\mathcal{P}\) is mapped to $\mathcal{G}^1_\pi$. In view of Lemma \ref{lem:jet_symbol_dominant}, one obtains that $\mathcal{P}$ is a polarisation in the sense of Definition \ref{def:higher_order_polarisation}. 
\end{proof}

Defining a polarisation allows us to define a notion of a polarised \(k\)-vector field and  a tower of generalised distributions that will play a relevant role, for instance, in the characterisation of jet bundles as a type of polarised \(k\)-contact manifolds. 

\begin{definition} A \emph{polarised $k$-contact vector field} relative to a polarised $k$-contact manifold $(M,\mathcal{D},\mathcal{P})$ is a $k$-contact vector field $X$ on $M$ such that $[X,\Gamma(\mathcal{P})]\subset \Gamma(\mathcal P)$.
\end{definition}

One of the first results concerning the tower of the polarisation is related to the prolongation of vector fields on $E$ to $J^r\pi$. The following result shows that the prolongation of a vector field on $E$  gives rise to a prolongation that is a polarised \(k\)-contact vector field on the jet bundle. This is a key result for applications to symmetries of PDEs, as it shows that the usual prolongation of vector fields on fibre bundles to jet bundles is compatible with the polarisation of the jet bundle. Nevertheless, we leave the detailed analysis of the consequences of this result for a further work, as it can be much extended.
 
\begin{theorem}\label{thm:prolonged_point_symmetries_are_polarised}
Given \(X\in\mathfrak X(E)\), its prolongation \(X^{(r)}\in\mathfrak X(J^r\pi)\) is a polarised \(N^r_\pi\)-contact vector field of \((J^r\pi,\mathcal C^r_\pi,\mathcal{G}^r_\pi)\), where \(\mathcal{G}^r_\pi:=\ker \T\pi_{r,r-1}\). 
Moreover, 
\[
 [X^{(r)},\Gamma(\ker \T\pi_{r,r-s})]\subset\Gamma(\ker \T\pi_{r,r-s}),\qquad 1\leq s\leq r.
\]
In an adapted coordinate system for an open $U\subset J^r\pi$, it follows that $\ker \T\pi_{r,r-s}=\langle \partial/\partial u_K^\alpha:|K|\geq r-s+1,\alpha=1,\ldots,m\rangle$ for $s=0,1,\ldots,r-1$.
\end{theorem}

\begin{proof} 
Given a fixed $s=1,\ldots,r$, every \(\ker\T\pi_{r,r-s}\) is locally generated by the vertical vector fields \(\partial/\partial u^\alpha_K\) with \(r\geq |K|\ge r-s+1\) and $\alpha=1,\ldots,m$. The preservation of the Cartan distribution $\mathcal{C}^r_\pi$ is the defining infinitesimal property of Lie prolongations to $J^r\pi$. It remains to prove the preservation of the vertical tower, namely the subbundles given by $\ker \T\pi_{r,r-s}$ for $s=1,\ldots,r$.  Locally on an adapted coordinate system $(x^i,u^\alpha_J)$ on an open $U\subset J^r\pi$, one has
\[
 X=\sum_{i=1}^n\xi^i(x,u)\frac{\partial}{\partial x^i}+\sum_{\alpha=1}^m\phi^\alpha(x,u)\frac{\partial}{\partial u^\alpha},\qquad X^{(r)}=\sum_{i=1}^n\xi^i\frac{\partial}{\partial x^i}+\sum_{\alpha=1}^m\sum_{|I|\leq r}\phi^\alpha_I\frac{\partial}{\partial u^\alpha_I},
\]
where $\phi^\alpha_{J+\widehat{e}_i}=D_i\phi^\alpha_{J}-\sum_{j=1}^n(D_i\xi^j) u^\alpha_{J+\widehat e_j}$ for $\alpha=1,\ldots,m$ and $|J|\leq r-1$. Then,  \(\phi^\alpha_I\) depends only on jet coordinates of order smaller or equal to \(|I|\). Hence, for every \(s=0,\ldots,r-1\), the field \(X^{(r)}\) is projectable with respect to the canonical projection \(\pi_{r,r-s-1}:J^r\pi\to J^{r-s-1}\pi\), its projection being \(X^{(r-s-1)}\). Therefore, it preserves the  distribution \(\ker \T\pi_{r,r-s-1}\). Thus,
\[
 [X^{(r)},\Gamma(\ker\T\pi_{r,r-s-1})]\subset\Gamma(\ker\T\pi_{r,r-s-1}),\qquad s=0,\ldots,r-1.
\]
In particular, for \(s=0\),  one gets that \(X^{(r)}\) preserves \(\mathcal{G}^r_\pi\), and therefore it is a polarised \(N^r_\pi\)-contact vector field.
\end{proof}

\begin{definition}
Let \((M,\mathcal D_M,\mathcal P_M)\) and \((N,\mathcal D_N,\mathcal P_N)\) be polarised \(k\)- and \(k'\)-contact manifolds. A \emph{polarised transcontact map} is a transcontact map \(\Phi:M\to N\) such that \(\T\Phi(\mathcal D_M)\subset\mathcal D_N\) and \(\T\Phi(\mathcal P_M)\subset\mathcal P_N\). If $\Phi^{-1}$ exists and it is also a polarised transcontact morphism, then we say that $(M,\mathcal{D}_M,\mathcal{P}_M)$ and $(N,\mathcal{D}_N,\mathcal{P}_N)$ are \emph{polarised $k$-contactomorphic}.
\end{definition}

\section{Jet manifolds as higher-order polarised \texorpdfstring{\(k\)}{k}-contact manifolds}
\label{sec:jet_manifolds_as_higher_order_polarised_k_contact_manifolds}

 The previous section provided a general definition for polarisations that covered the classical first-order case \cite{Riv_22} as a particular instance. The new polarisation notion is very relevant for further developments and will also be a key for other applications. More specifically, the so-called polarisation of jet type and order $r$ will be a key to characterise $k$-contact manifolds that are locally polarised $k$-contactomorphic to $J^r\pi$. The following notion gives the first step in this direction. 

 \begin{definition}\label{def:split_polarisation}A {\it split polarisation} of a \(k\)-contact manifold \((M,\mathcal D)\) is a polarisation $\mathcal{P}\subset \mathcal D$ such that on an open subset $U$ around every point $p\in M$ there exists an integrable isotropic subbundle, \(H\subset\mathcal D\) such that \(\mathcal D|_U=H\oplus\mathcal P|_U\). 
\end{definition}

For $k$-contact manifolds $(M,\mathcal{D})$, where $\dim M=nk+n+k$ and $\mathcal{P}$ is an integrable subbundle of $\mathcal{D}$ of rank $nk$, the $k$-contact Darboux theorem \cite{Riv_22} ensures that there exists, around each point in $M$, an integrable isotropic complement of $\mathcal P$ in $\mathcal D$. In jet manifolds, we also proved that  $\mathcal G^r_\pi$ has an integrable isotropic complement. 
Nevertheless, the integrable Legendrian condition of maximal rank on \(\mathcal P\) only ensures an isotropic complement pointwise: it does not imply in general the existence of a smooth integrable isotropic complement even locally. 
The following counterexample illustrates this fact.

\begin{example}\label{ex:legendrian_without_integrable_isotropic_complement}
Let us show  that an integrable Legendrian subbundle of maximal rank may admit smooth isotropic complements pointwise without admitting any integrable isotropic complement. Let \(M=\mathbb R^8\) with coordinates \((x,y,p,q,z^1,z^2,z^3,z^4)\), and consider the \(\mathbb R^4\)-valued differential one-form \(\bm\eta=\sum_{\alpha=1}^4\eta^\alpha\otimes e_\alpha\) given by
\[
\eta^1=\dd z^1+p\,\dd x+\frac{x^2}{2}\,\dd y,\qquad
\eta^2=\dd z^2+q\,\dd x,\qquad
\eta^3=\dd z^3-p\,\dd y,\qquad
\eta^4=\dd z^4-q\,\dd y .
\]
Then, \(\mathcal D=\ker\bm\eta\) is locally generated by
\[
X_1=\frac{\partial}{\partial x}
-p\frac{\partial}{\partial z^1}
-q\frac{\partial}{\partial z^2},\,\,
X_2=\frac{\partial}{\partial y}
+x\frac{\partial}{\partial p}
-\frac{x^2}{2}\frac{\partial}{\partial z^1}
+p\frac{\partial}{\partial z^3}
+q\frac{\partial}{\partial z^4},
\,\,
P_1=\frac{\partial}{\partial p},\,\,
P_2=\frac{\partial}{\partial q}.
\]
Thus,
\(
\mathcal D=\langle X_1,X_2,P_1,P_2\rangle 
\) and
 $\ker \dd\bm \eta=\langle \partial/\partial{z^1},\ldots,\partial/\partial {z^4}\rangle$, which gives  $\ker\bm\eta\oplus\ker\dd\bm\eta=\T \mathbb{R}^8$. 
Hence, \(\mathcal D\)  is a four-contact distribution and $\bm\eta$ is a four-contact form. 

There are no isotropic subbundles of \(\mathcal D\) of rank four, because $\mathcal{D}$ is a four-contact distribution of rank four and $\mathcal{D}^{\perp_\mathcal{D}}=0$. There are no isotropic subbundles of $\mathcal{D}$ of rank three either. Indeed, let
\(
v=aX_1+bX_2+cP_1+dP_2\in\mathcal D_\xi
\)
be non-zero, and let
\(
w=AX_1+BX_2+CP_1+DP_2\in\mathcal D_\xi
\) for $\xi\in \mathbb{R}^8$.
The condition \(\dd\bm\eta(v,w)=0\) is equivalent to
\[
aC-cA=0,\qquad aD-dA=0,\qquad bC-cB=0,\qquad bD-dB=0.
\]
In other words, 
$$
{\rm rk}\left[\begin{array}{cc}a&c\\A&C\end{array}\right]<2,\qquad {\rm rk}\left[\begin{array}{cc}b&c\\B&C\end{array}\right]<2,\qquad 
{\rm rk}\left[\begin{array}{cc}a&d\\A&D\end{array}\right]<2,\qquad {\rm rk}\left[\begin{array}{cc}b&d\\B&D\end{array}\right]<2,
$$
If \(v\in\langle X_1,X_2\rangle\setminus\{0\}\), these equations force \(C=D=0\), so the \(\dd\bm\eta\)-orthogonal of \(v\) in \(\mathcal D_\xi\) is contained in \(\langle X_1,X_2\rangle\), which has rank \(2\). If \(v\in\langle P_1,P_2\rangle\setminus\{0\}\), they force \(A=B=0\), so the orthogonal of \(v\) is contained in \(\langle P_1,P_2\rangle\), also of rank \(2\). Finally, if \(v\) has both a non-zero \(\langle X_1,X_2\rangle\)-component and a non-zero \(\langle P_1,P_2\rangle\)-component, the equations imply that \(w\) is proportional to \(v\). Hence, for every non-zero \(v\in\mathcal D_\xi\), the space
\(
\{w\in\mathcal D_\xi:\dd\bm\eta(v,w)=0\}
\)
has dimension at most two. Therefore, an isotropic subspace of \(\mathcal D\) cannot have dimension three. 

Now set
\(
\mathcal P=\langle P_1,P_2\rangle\subset\mathcal D,
\)
which is integrable and isotropic. Moreover, it admits an isotropic complement pointwise: indeed,
\(
H_0=\langle X_1,X_2\rangle
\) satisfies \(
\mathcal D=H_0\oplus\mathcal P
\) and is isotropic, since \([X_1,X_2]=P_1\in\Gamma(\mathcal D)\). In particular, \(\mathcal P=\langle P_1,P_2\rangle\) is maximal isotropic. Thus, \(\mathcal P\) is Legendrian relative to $\bm\eta$.

We now prove that \(\mathcal P\) admits no integrable isotropic complement. Let \(H\subset\mathcal D\) be any smooth complement of \(\mathcal P\). Since \(\mathcal D=H_0\oplus\mathcal P\), locally \(H\) is the graph of a bundle map \(H_0\to\mathcal P\). Therefore \(H\) is generated by vector fields of the form
\[
E_1=X_1+aP_1+bP_2,\qquad
E_2=X_2+cP_1+dP_2,
\]
for some smooth functions \(a,b,c,d\). The condition that \(H\) be isotropic gives
\[
\dd\bm \eta(E_1,E_2)=-ce_1-de_2-ae_3-be_4=0.
\]
Consequently, the only isotropic complement of \(\mathcal P\) is \(H_0=\langle X_1,X_2\rangle\). But
\(
[X_1,X_2]=P_1\notin\Gamma(H_0),
\)
so \(H_0\) is not integrable. Hence, \(\mathcal P\) is an integrable Legendrian subbundle admitting isotropic complements pointwise, but it admits no integrable isotropic complement.
\end{example}

The above example also shows that determining an integrable isotropic complement of a Legendrian subbundle may require a possibly long but elementary calculation. More theoretically, our interest in split polarisations is due to their natural appearance in finite-order jet manifolds as $\ker\T\pi_{r,r-1}$ with an integrable complement spanned by some local total derivatives. Moreover, it is worth noting that for $k=1$, split polarisations are related to bi-Legendrian contact geometry \cite{CappellettiMontano2005BiLegendrianConnections}. 

Now, let us give a second necessary condition for a  $k$-contact manifold to be  locally polarised $k$-contactomorphic to  a finite-order jet manifold. 
\begin{definition}\label{def:Spencer_type_derived_flag} We say that $(M,\mathcal{D},\mathcal{P})$ is of {\it Spencer type} and order $r$ if \(\mathcal D^{(r)}=\T M\) and, on a neighbourhood of every point, there exists an integrable isotropic complement \(H\subset\mathcal D\) of \(\mathcal P\) such that, for \(W\simeq \mathcal D^{(r)}/\mathcal D^{(r-1)}\), one has, locally on $U$, that
\[
\mathcal P\simeq S^rH^*\otimes W,\qquad \mathcal D^{(s)}/\mathcal D^{(s-1)}\simeq S^{r-s}H^*\otimes W,\qquad s=1,\ldots,r.
\]
Moreover, under these identifications, the bracket action of \(H\) on the graded quotients,
\(
H\otimes\mathcal P\longrightarrow \mathcal D^{(1)}/\mathcal D, X\otimes Y\longmapsto -[X,Y]\mod \mathcal D,
\)
and
\begin{equation}\label{eq:ConSpec}
H\otimes\big(\mathcal D^{(s)}/\mathcal D^{(s-1)}\big)\longrightarrow \mathcal D^{(s+1)}/\mathcal D^{(s)},\qquad X\otimes[Y]\longmapsto -[X,Y]\mod \mathcal D^{(s)},
\end{equation}
for \(s=1,\ldots,r-1\), coincides  with the Spencer contraction
\[
H\otimes S^{r-s}H^*\otimes W\longrightarrow S^{r-s-1}H^*\otimes W,\qquad X\otimes\omega\otimes w\longmapsto \iota_X\omega\otimes w.
\]
\end{definition}

 The point of introducing the above definition, which is necessary but not sufficient to ensure the  existence of a local $k$-contactomorphism with a jet manifold $J^r\pi$, is to provide necessary conditions that are simple enough to be determined in practical cases. It is also worth noting that the mappings \eqref{eq:ConSpec} are vector bundle morphisms, despite being defined by means of a Lie bracket of vector fields.
 
\begin{definition}\label{def:polarised_jet_type}
A polarised \(k\)-contact manifold \((M,\mathcal D,\mathcal P)\) is of {\it jet type and order \(r\)} if it is of Spencer type of order $r$ and, around every $p\in M$ there exists an open $U\subset M$, an isotropic, integrable complement $H$ of $\mathcal{V}_0:=\mathcal{P}|_U$ in $\mathcal{D}|_U$ and an integrable complement $\mathcal{V}_r$ of $H$ in $\T U$ so that \(\T U=H\oplus\mathcal V_r\) and 
\begin{equation}\label{Tower_condition}
\mathcal{V}_s=\mathcal{D}^{(s)}\cap \mathcal{V}_r,\qquad s=0,\ldots,r-1,
\end{equation}
are regular involutive distributions.
\end{definition}

\begin{lemma}\label{lem:derived-flag-vertical-flag}
Let \((M,\mathcal D, \mathcal{P})\) be a polarised $k$-contact manifold of jet type and order $r$. If the associated $H$, $\mathcal{V}_r$ are locally defined on $U$, then
\(
\mathcal D^{(s)}|_U=H\oplus\mathcal V_s,\) for \(s=0,\ldots,r-1,
\)
and  there are natural identifications \((\mathcal D^{(s)}/\mathcal D^{(s-1)})|_U\simeq\mathcal V_s/\mathcal V_{s-1}\) for \(s=1,\ldots,r-1\).
\end{lemma}

\begin{proof}
Since \(H\subset\mathcal D=\mathcal D^{(0)}\), one has \(H\subset\mathcal D^{(s)}\) for every \(s\). Let \(v_Z\in\mathcal D^{(s)}\). By \(TU=H\oplus\mathcal V_r\), there is a unique decomposition \(v_Z=v_X+v_Y\), with \(v_X\in H\) and \(v_Y\in\mathcal V_r\). Since \(v_X\in H\subset\mathcal D^{(s)}\), it follows that \(v_Y=v_Z-v_X\in\mathcal D^{(s)}\cap\mathcal V_r=\mathcal V_s\). Hence, \(\mathcal D^{(s)}\subset H\oplus\mathcal V_s\). The converse inclusion follows from \(H\subset\mathcal D^{(s)}\) and \(\mathcal V_s\subset\mathcal D^{(s)}\). The sum is direct because \(TU=H\oplus\mathcal V_r\). Taking quotients gives \((\mathcal D^{(s)}/\mathcal D^{(s-1)})|_U\simeq\mathcal V_s/\mathcal V_{s-1}\).
\end{proof}

By Lemma~\ref{lem:derived-flag-vertical-flag}, the identifications
\((\mathcal D^{(s)}/\mathcal D^{(s-1)})|_U\simeq \mathcal V_s/\mathcal V_{s-1}\) identify the bracket action of \(H\) on the derived quotients of \(\mathcal D\) with the bracket action
\[
H\otimes \mathcal V_s/\mathcal V_{s-1}
\longrightarrow
\mathcal V_{s+1}/\mathcal V_s.
\]
Thus, the Spencer contraction condition may be expressed equivalently in terms of the derived flag of \(\mathcal D\) or in terms of the \(\mathcal V_0,\ldots,\mathcal{V}_r\). 
\begin{remark}\label{rem:independence_H}
The jet type polarisation notion  is independent of the chosen integrable isotropic complement $H$. Indeed, suppose that \(H\subset\mathcal D\) is an integrable isotropic complement of a jet type polarisation \(\mathcal P\), and that there exists an integrable regular subbundle \(\mathcal V_r\subset \T U\) such that \(\T U=H\oplus\mathcal V_r\), \(\mathcal V_0=\mathcal P\), and \eqref{Tower_condition} holds. Let \(H'\subset\mathcal D\) be another integrable isotropic complement of \(\mathcal P\). Since \(\mathcal D=H\oplus\mathcal P=H'\oplus\mathcal P\), locally \(H'\) is the graph of a bundle morphism \(A:H\to\mathcal P\), namely \(H'=\{X+A(X):X\in H\}\). Since \(\mathcal P\subset\mathcal V_r\), one has \(H'\cap\mathcal V_r=0\): if \(X+A(X)\in\mathcal V_r\), then \(X\in H\cap\mathcal V_r=0\). Therefore, \(\T U=H'\oplus\mathcal V_r\). As the derived flag \(\mathcal D^{(s)}\) is intrinsic, the intersections \(\mathcal D^{(s)}\cap\mathcal V_r\) are unchanged. Hence, the same subbundles \(\mathcal V_s\) are obtained for \(H'\).
 Moreover, the induced Spencer bracket maps on the graded quotients are also independent of the chosen admissible complement. Let \(X'\in\Gamma(H')\) and write \(X'=X+A(X)\), with \(X\in\Gamma(H)\) and \(A(X)\in\Gamma(\mathcal P)\). If \(Y\in\Gamma(\mathcal V_s)\), then
\(
[X',Y]=[X,Y]+[A(X),Y].
\)
Since \(\mathcal P=\mathcal V_0\subset\mathcal V_s\) and the \(\mathcal V_s\) are involutive, one has \([A(X),Y]\in\Gamma(\mathcal V_s)\). Consequently,
\[
[X',Y]\equiv [X,Y]\mod \mathcal V_s.
\]
Thus, the bracket maps
\(
H\otimes(\mathcal V_s/\mathcal V_{s-1})\longrightarrow \mathcal V_{s+1}/\mathcal V_s
\)
and
\(
H'\otimes(\mathcal V_s/\mathcal V_{s-1})\longrightarrow \mathcal V_{s+1}/\mathcal V_s
\)
coincide after the natural identification \(H'\simeq H\simeq \mathcal D/\mathcal P\). Therefore, if the bracket action for \(H\) is the Spencer contraction, then the bracket action for \(H'\) is the same Spencer contraction.
\end{remark}

Note that the only unknowns to be determined in Definition \ref{def:polarised_jet_type} are $H$ and $\mathcal{V}_{r}$. In any case, all possible $H$ can be worked out by means of a possibly long and tedious computation and one only needs to ensure the existence of one.

\begin{proposition}The triple \((J^r\pi,\mathcal C^r_\pi,\mathcal G^r_\pi)\), where \(\mathcal G^r_\pi=\ker \T\pi_{r,r-1}\), is of jet type and order \(r\).
\end{proposition}

\begin{proof}
Work in adapted coordinates \((x^i,u^\alpha_I)\) on some $U\subset J^r\pi$ with  \(|I|\le r\) and $\alpha=1,\ldots,m$. The Cartan distribution decomposes on $U$ as
\(
\mathcal C^r_\pi|_U=H^r_\pi\oplus\mathcal G^r_\pi|_U,
\)
where \(H^r_\pi=\langle D_1,\ldots,D_n\rangle\) and \(\mathcal G^r_\pi=\langle\partial/\partial u^\alpha_K:|K|=r,\alpha=1,\ldots,m\rangle\). Both \(H^r_\pi\) and \(\mathcal G^r_\pi\) are integrable, and both are isotropic for $\dd\bm \eta^r_\pi$.

If \(\mathcal V^\pi_0:=\mathcal G^r_\pi\), while $\mathcal{V}^\pi_s:=\ker \T\pi_{r,r-s-1}$ for $1\le s\le r-1$, and $\mathcal{V}^\pi_r:=\ker \T\pi_r$, then one gets that \(\mathcal{D}^{(s)}|_{U}=H^r_\pi\oplus\mathcal V^\pi_s|_{U}=H^r_\pi\oplus\ker \T\pi_{r,r-s-1}|_{U}\) for \(0\le s\le r-1\), and  \( \mathcal{D}^{(r)}|_{U}:=H^r_\pi\oplus\ker \T\pi_r|_{U}\). Then, the distributions \(\mathcal V^\pi_s\) are regular and involutive. Moreover,
\[
\left[D_i,\frac{\partial}{\partial u^\alpha_K}\right]=
\begin{cases}
-\dfrac{\partial}{\partial u^\alpha_{K-\widehat e_i}},&k_i>0,\\[4pt]
0,&k_i=0.
\end{cases}
\]
with the bracket by \(D_i\) corresponding to contraction by \(\partial/\partial x^i\), up to sign. 
Previous expressions also show 
 that, on $U$, one has
\[
\mathcal V^\pi_s/\mathcal V^\pi_{s-1}\simeq S^{r-s}H^{r*}_\pi\otimes\pi_{r,0}^*VE,\qquad s=0,\ldots,r,\quad \mathcal V^\pi_{-1}:=0.
\]
Hence, the bracket-induced maps are precisely the Spencer contractions.

\end{proof}

\begin{theorem}\label{thm:local_split_jet_recognition}
A  polarised \(k\)-contact manifold \((M,\mathcal D,\mathcal P )\) is of jet type and order \(r\) if, and only if, locally around every point in $M$, there exists a fibred manifold \(\pi:E\to Q\), together with a local polarised  $k$-contactomorphism 
\(
\Phi:U\longrightarrow \Phi(U)\subset J^r\pi.
\)
Moreover, if \(n=\operatorname{rank}H\) and \(m=\operatorname{rank}(\mathcal V_r/\mathcal V_{r-1})\), then
\[
k=m\binom{n+r-1}{r-1}.
\]
\end{theorem}

\begin{proof}
Assume first that \((M,\mathcal D,\mathcal P)\) is locally polarised $k$-contactomorphic to a jet model as in the statement. Then, there exists a local diffeomorphism $\Phi:U\subset M\rightarrow V\subset J^r\pi$ on a neighbourhood $U$ of every $x\in M$ so that $J^r\pi$ admits an adapted coordinate system on $V$, while $\T\Phi(\mathcal{D}|_U)=\mathcal{C}^r_\pi|_V$ and $\T\Phi(\mathcal{P}|_U)=\ker \T\pi_{r,r-1}|_V$. Hence, there exists an integrable isotropic complement \(H'=\T\Phi^{-1}(H)\) of \(\mathcal P|_U\) in \(\mathcal D|_U\), where \(H=\langle D_1,\ldots,D_n\rangle\), and an integrable complement \(\mathcal V_r=\T\Phi^{-1}(\ker\T\pi_r|_V)\) of \(H'\) in \(\T U\). The standard vertical flag on \(J^r\pi\) is
\(
\mathcal G^r_\pi=\mathcal V^\pi_0\subset\mathcal V^\pi_1\subset\cdots\subset\mathcal V^\pi_r,
\)
where \(\mathcal V^\pi_s=\ker \T\pi_{r,r-s-1}\) for \(0\le s\le r-1\) and \(\mathcal V^\pi_r=\ker \T\pi_r\). Moreover, \(\T\Phi(\mathcal D^{(s)})=(\mathcal C^r_\pi)^{(s)}\) for \(s=0,\ldots,r\). Additionally, $\mathcal{D}$, $\mathcal{P}$ satisfy locally all the conditions of a distribution of jet type and order $r$ since $\mathcal{C}^r_\pi$, $\mathcal{G}^r_\pi$, and $\ker \T\pi_r$ do so. 

Conversely, assume that \((M,\mathcal D,\mathcal P)\) is of jet type and order \(r\). Work on a sufficiently small open set \(U\subset M\), and use the sequence \(\mathcal P=\mathcal V_0\subset\cdots\subset\mathcal V_r\). Since \(\T U=H\oplus\mathcal V_r\) and both are involutive, they define local transverse foliations. Hence, after possibly shrinking \(U\), there is a submersion \(q:U\to Q\), with vertical distribution \(\mathcal V_r\), and local coordinates \(x^1,\ldots,x^n\) on \(Q\), where \(n=\operatorname{rank}H\). Let \(D_1,\ldots,D_n\) be the \(H\)-horizontal lifts of \(\partial/\partial x^1,\ldots,\partial/\partial x^n\). Writing again \(x^j=q^*x^j\) for $j=1,\ldots,n$, with a slight abuse of notation, one has \(D_i(x^j)=\delta_i^j\) for $i,j=1,\ldots,n$. It is worth noting that $\mathcal{V}_r=\ker\T q$ and that the \(D_i\) are \(q\)-projectable.  Then, condition \eqref{Tower_condition} implies that $\mathcal{V}_s=\mathcal{D}^{(s)}\cap\ker\T q$ for $s=0,\ldots,r$.

Since \(\mathcal V_{r-1}\) is involutive and contained in \(\mathcal V_r\), the local quotient \(E:=U/\mathcal V_{r-1}\) is smooth and carries a natural local fibration \(E\to Q\). Choose local functions \(u^1,\ldots,u^m\) constant along the integrals of \(\mathcal V_{r-1}\) whose descents are fibre coordinates on \(E\to Q\), where \(m=\operatorname{rank}(\mathcal V_r/\mathcal V_{r-1})\). Define
\[
u^\alpha_I:=D_Iu^\alpha,\qquad |I|\le r.
\]
The commutativity of the \(D_1,\ldots,D_n\) makes these functions independent of the order of differentiation of each $I$.

We now prove that, for every \(s=0,\ldots,r-1\), every \(Y\in\Gamma(\mathcal V_s)\) satisfies \(Yu^\alpha_I=0\) for $\alpha=1,\ldots,m$ and whenever \(|I|<r-s\). The proof is a descending induction on \(s\), and, for each fixed \(s\), an induction on \(|I|\). For \(s=r-1\), the only possible multi-index $|I|<r-s$ is \(|I|=0\), and the assertion follows from the fact that  the functions \(u^1,\ldots,u^m\) are constant along the leaves of \(\mathcal V_{r-1}\). Assume now that the assertion holds at level \(s+1< r \), and prove it at level \(s\).

First, if \(|I|=0\), then \(Yu^\alpha=0\) for $\alpha=1,\ldots,m$, because \(\mathcal V_s\subset\mathcal V_{r-1}\). Let now \(0<|I|<r-s\), and suppose that the assertion has already been proved at the same level \(s\) for all multi-indices of degree less than \(|I|\). Choose \(i\) such that \(I=J+\widehat e_i\). Then, \(|J|=|I|-1<r-s-1\). Since, by definition, \(u^\alpha_I=D_i u^\alpha_J\), one has
\[
Yu^\alpha_I=Y(D_i u^\alpha_J)=D_i(Yu^\alpha_J)-[D_i,Y](u^\alpha_J),\qquad \alpha=1,\ldots,m.
\]
The first term vanishes by the induction hypothesis on the degree at the fixed level \(s\), because \(|J|<|I|\). For the second term, since \(D_i\) is \(q\)-projectable and \(Y\) is \(q\)-vertical, the bracket \([D_i,Y]\) is again \(q\)-vertical: indeed, \(Tq([D_i,Y])=[Tq(D_i),Tq(Y)]=[Tq(D_i),0]=0\). By the defining Spencer-tower condition, the \(q\)-vertical bracket of a local section of \(H\) with a local section of \(\mathcal V_s\) belongs to \(\mathcal V_{s+1}\). In particular, since $D_i$ is $q$-projectable, one has \([D_i,Y]\in\Gamma(\mathcal V_{s+1})\). The other induction hypothesis, applied at level \(s+1\), gives \([D_i,Y](u^\alpha_J)=0\), since \(|J|<r-s-1=r-(s+1)\). Therefore, \(Yu^\alpha_I=0\). This completes the induction on \(|I|\), and hence the descending induction on \(s\).

It follows that the values \(Yu^\alpha_I\), with \(|I|=r-s\) and $\alpha=1,\ldots,m$, depend only on the class of \(Y\in \Gamma(\mathcal{V}_s)\) modulo \(\mathcal V_{s-1}\). Thus, one obtains maps
\[
\Sigma_s:\mathcal V_s/\mathcal V_{s-1}\longrightarrow S^{r-s}H^*\otimes W,\qquad W:=\mathcal V_r/\mathcal V_{r-1},
\]
by
\[
\Sigma_s([Y])=\sum_{|I|=r-s}\sum_{\alpha=1}^m (Yu^\alpha_I)\,\odot\xi^{\langle I\rangle}\otimes e_\alpha,
\]
where \((\xi^1,\ldots,\xi^n)\) is the unique coframe dual to \((D_1,\ldots,D_n)\) and annihilating \(\mathcal V_r\). Moreover, if \(|J|=r-s-1\), then
\[
[D_i,Y](u^\alpha_J)=-Yu^\alpha_{J+\widehat e_i}.
\]
Hence, the maps \(\Sigma_s\) intertwine the projected bracket maps with the Spencer contractions.

We now prove that all \(\Sigma_s\) are local isomorphisms. For \(s=r\), we define \(\Sigma_r:\mathcal V_r/\mathcal V_{r-1}\to W\) by the classes of the functions \(Yu^\alpha\) with $\alpha=1,\ldots,m$. Then, \(\Sigma_r:\mathcal V_r/\mathcal V_{r-1}\to W\) is an isomorphism by the choice of the functions \(u^1,\ldots,u^m\). Assume that \(\Sigma_{s+1}\) is an isomorphism and that \(Y\in\Gamma(\mathcal V_s)\) satisfies \(\Sigma_s([Y])=0\). Then, \(Yu^\alpha_I=0\) for every \(|I|=r-s\), and the vanishing property gives the same for all \(|I|<r-s\). Consequently, for every \(|J|=r-s-1\), one has \([D_i,Y](u^\alpha_J)=-Yu^\alpha_{J+\widehat e_i}=0\), and hence \(\Sigma_{s+1}([D_i,Y])=0\). Since \(\Sigma_{s+1}\) is an isomorphism, \([D_i,Y]\in\Gamma(\mathcal V_s)\) for all \(i\). Thus, the class of \(Y\) in \(\mathcal V_s/\mathcal V_{s-1}\) has zero Spencer contraction with every element of \(H\). Under the jet-type identification \(\mathcal V_s/\mathcal V_{s-1}\simeq S^{r-s}H^*\otimes W\), and since \(r-s>0\), the class of \(Y\) must vanish. Hence, \(Y\in\Gamma(\mathcal V_{s-1})\), so \(\Sigma_s\) is injective. Since its source and target have the same rank, it is an isomorphism. Descending induction proves the claim.

Define
\[
\Phi:U\longrightarrow J^r(Q,\mathbb R^m),\qquad \Phi(p)=(x^i(p),u^\alpha_I(p))_{|I|\le r}.
\]
If \(Z\in\ker \T\Phi\), then \(Zx^i=0\), so \(Z\in\mathcal V_r\). Since \(Zu^\alpha=0\), its class in \(\mathcal V_r/\mathcal V_{r-1}\) vanishes, and hence \(Z\in\mathcal \Gamma(V_{r-1})\). Repeating the same argument with the isomorphisms \(\Sigma_{r-1},\ldots,\Sigma_0\), one obtains \(Z\in\mathcal V_{-1}=0\) and  \(\T\Phi\) becomes injective. The dimensions agree, since
\[
\dim U=n+m\binom{n+r}{r}=\dim J^r(Q,\mathbb R^m),
\]
and therefore \(\Phi\) is a local diffeomorphism.

It remains to identify the geometric structures. Let \(\theta^\alpha_I\), with \(|I|\le r-1\), be the standard contact forms on \(V\subset J^r(Q,\mathbb R^m)\) associated with $(x^i,u^\alpha_I)$. Since \(D_iu^\alpha_I=u^\alpha_{I+\widehat e_i}\) for \(|I|\le r-1\) and $\alpha=1,\ldots,m$, one has \((\Phi^*\theta^\alpha_I)(D_i)=0\). If \(Y\in\Gamma(\mathcal P)=\Gamma(\mathcal V_0)\), then \(Y(x^i)=0\) and the vanishing property gives \(Y(u^\alpha_I)=0\) for \(|I|\le r-1\), so \((\Phi^*\theta^\alpha_I)(Y)=0\), which gives \(\T\Phi(\mathcal D)\subset\mathcal C^r_\pi\). Since the ranks agree, \(\T\Phi(\mathcal D)=\mathcal C^r_\pi\).

The relation $\pi_r\circ \Phi=q$ yields  that \(\T\Phi(\mathcal P)=\ker \T\pi_{r,r-1}=\mathcal G^r_\pi\). Consequently, \(\Phi\) is a local polarised $k$-contactomorphism between \((M,\mathcal D,\mathcal P)\) and the jet model \((J^r\pi,\mathcal C^r_\pi,\mathcal G^r_\pi)\).

The formula for \(k\) follows from \(\T M/\mathcal D\simeq\mathcal V_r/\mathcal V_0\). Hence,
\[
k=\sum_{s=1}^r\operatorname{rank}(\mathcal V_s/\mathcal V_{s-1})
=m\sum_{j=0}^{r-1}\binom{n+j-1}{j}
=m\binom{n+r-1}{r-1}.
\]\end{proof}

Theorem~\ref{thm:local_split_jet_recognition} reconstructs local jet coordinates from the  \((\mathcal D,\mathcal P,H)\), the  $\mathcal{V}_r$, and the Spencer behaviour of the graded quotients. The canonical Cartan distribution and the highest-order vertical polarisation are recovered intrinsically. The chosen horizontal complement \(H\), however, is recovered as an integrable isotropic complement. Thus, the jet-bundle structure is recovered from polarised \(k\)-contact data, while the splitting keeps track of an additional choice of $H$, which plays the role of a system of total derivatives. In essence, Theorem~\ref{thm:local_split_jet_recognition} characterised finite-order jet geometry  as a  class of polarised \(k\)-contact geometry of jet type and order $r$.

Let us describe the above claim more in detail.  
Theorem \ref{thm:local_split_jet_recognition} introduces local coordinates \((x^i,u^\alpha_I)\) on $M$ in which the \(k\)-contact distribution $\mathcal{D}$ is locally the Cartan distribution of a finite-order jet bundle, while the polarisation $\mathcal{P}$ is identified with the highest-order vertical distribution of $J^r\pi$, namely \(\ker \T\pi_{r,r-1}\). Moreover, the formula for the coordinates $u^\alpha_I$ shows that 
$$
D_i=\frac{\partial}{\partial x^i}+\sum_{\alpha=1}^m\sum_{|I|<r}u_{I+\widehat{e}_i}^\alpha\frac{\partial}{\partial u^\alpha_I}+\sum_{\alpha=1}^m\sum_{|J|=r}A^\alpha_{i,J}\frac{\partial}{\partial u^\alpha_J},\qquad i=1,\ldots,n,
$$
for certain functions $A^\alpha_{i,J}$ with $i=1,\ldots,n$, $\alpha=1,\ldots,m$ and $|J|=r$.
On the other hand, $\mathcal{P}$ consists of vector fields that  vanish on all the functions $u^\alpha_J$ for $|J|<r$ and the $x^1,\ldots,x^n$. This gives that
$$
\mathcal{P}=\left\langle \frac{\partial}{\partial u^\alpha_J}:\alpha=1,\ldots,m,|J|=r\right\rangle.
$$
In particular, the $D_1,\ldots,D_n$ and the vector fields taking values in $\mathcal{P}$ span $\mathcal{D}$ that becomes a Cartan-like distribution
$$
\mathcal{D}=\left\langle \widehat{ D}_i=\frac{\partial}{\partial x^i}+\sum_{\alpha=1}^m\sum_{|I|<r}u_{I+\widehat{e}_i}^\alpha\frac{\partial}{\partial u^\alpha_I},\frac{\partial}{\partial u^\alpha_{J}}:\alpha=1,\ldots,m,|J|=r,i=1,\ldots,n\right\rangle
$$
Note that $H$ is not the standard distribution spanned by some horizontal derivations in the given coordinates, but the existence of $H$ induces the existence of an $H'$ that plays that role. 
In a similar manner, one gets that
$$
\mathcal{V}_r=\left\langle \frac{\partial}{\partial u^\alpha_J}:\alpha=1,\ldots,m,|J|\leq r\right\rangle,\,\, 
\mathcal{V}_s=\mathcal{D}^{(s)}\cap \mathcal{V}_r=\left\langle \frac{\partial}{\partial u^\alpha_J}:\alpha=1,\ldots,m,r\geq |J|\geq r-s\right\rangle.
$$
Hence, Theorem~\ref{thm:local_split_jet_recognition} is a Darboux-type theorem for polarised \(k\)-contact manifolds of jet type and order \(r\). It generalises the first-order Darboux theorem for \(k\)-contact manifolds in \cite{Riv_22}. The proof is more involved than the first-order case, which relies on a technical result on polarised $k$-symplectic manifolds. The spirit  of Theorem \ref{thm:local_split_jet_recognition} is close to the interpretation of Darboux coordinates for geometric structures as coordinates adapted to a canonical normal form \cite{GLRR_24}. 

Finally, $\dd _Hf=\sum_{i=1}^n \widehat{D}_ifdx^i$ gives a horizontal differential, which matches the value of the standard horizontal differential for functions $f=f(x^i,u^\alpha_J)$ with $|J|<r$.

Let us give now a global version of Theorem \ref{thm:local_split_jet_recognition}.

\begin{theorem}\label{thm:global_jet_recognition}
Let \((M,\mathcal D,\mathcal P)\) be a polarised \(k\)-contact manifold of jet type and order \(r\). Assume that \(\mathcal V_r\) and \(\mathcal V_{r-1}=\mathcal{D}^{(r-1)}\cap \mathcal{V}_r\) are globally defined, regular and simple, with smooth leaf spaces
\(
Q:=M/\mathcal V_r\) and \(E:=M/\mathcal V_{r-1}.
\)
Then, there exists a submersion \(\pi:E\to Q\). Assume moreover that the local identifications supplied by Theorem~\ref{thm:local_split_jet_recognition} have transition functions induced by \(r\)-th jet prolongations of local bundle automorphisms of \(\pi:E\to Q\). Then, local identifications glue to a well-defined smooth local diffeomorphism
\(
\Phi:M\to J^r\pi .
\) 
If this map is a global diffeomorphism onto \(J^r\pi\), then \(\Phi\) is a global polarised $k$-contactomorphism. 
\end{theorem}

\begin{proof}
The proof consists in showing that the local jet identifications supplied by Theorem~\ref{thm:local_split_jet_recognition} glue to a global identification with the jet bundle of the quotient fibration. Since \(\mathcal V_{r-1}\subset\mathcal V_r\) are regular simple foliations, the quotients \(E=M/\mathcal V_{r-1}\) and \(Q=M/\mathcal V_r\) are smooth manifolds. Let \(q_1:M\to E\) and \(q_0:M\to Q\) be the quotient maps. Since every leaf of \(\mathcal V_{r-1}\) is contained in a leaf of \(\mathcal V_r\), the map \(q_0\) is constant along the fibres of \(q_1\). Hence, there exists a unique smooth map \(\pi:E\to Q\) such that \(\pi\circ q_1=q_0\) whose rank is given by the rank of $\mathcal{V}_r/\mathcal{V}_{r-1}$. Since \(q_0\) and \(q_1\) are surjective submersions, \(\pi\) is also a surjective submersion.

For every point \(p\in M\), there exists by assumption an open neighbourhood \(U_a\subset M\), a fibred manifold \(\pi_a:E_a\to Q_a\), and a local diffeomorphism \(\Phi_a:U_a\to V_a\subset J^r\pi_a\) such that
\[
(\Phi_a)_*(\mathcal D|_{U_a})=\mathcal C^r_{\pi_a}|_{V_a},\qquad
(\Phi_a)_*(\mathcal P|_{U_a})=\ker \T(\pi_a)_{r,r-1}|_{V_a}.
\]
The local recognition theorem also identifies \(\mathcal V_0\subset\cdots\subset\mathcal V_r\) with the standard vertical jet tower
\[
\ker \T(\pi_a)_{r,r-1}\subset\ker \T(\pi_a)_{r,r-2}\subset\cdots\subset\ker \T(\pi_a)_{r,0}\subset\ker \T(\pi_a)_r.
\]
In particular, \((\Phi_a)_*(\mathcal V_r)=\ker \T(\pi_a)_r|_{V_a}\) and \((\Phi_a)_*(\mathcal V_{r-1})=\ker T(\pi_a)_{r,0}|_{V_a}\). Therefore, the local bases and total spaces of the models identify with open subsets of \(Q\) and \(E\), respectively.

We now construct \(\Phi:M\to J^r\pi\). Given \(p\in M\), set \(x:=q_0(p)\) and \(e:=q_1(p)\), so that \(\pi(e)=x\). Choose a local chart \(U_a\ni p\). Via the identifications of \(Q_a\) and \(E_a\) with open subsets of \(Q\) and \(E\), the point \(\Phi_a(p)\in J^r\pi_a\) determines an \(r\)-jet at \(x\) of a local section of \(\pi:E\to Q\) with value \(e\). Define \(\Phi(p)\) to be this \(r\)-jet. This definition is independent of the chosen chart because, on overlaps \(U_a\cap U_b\), the transition map \(\Phi_b\circ\Phi_a^{-1}\) is assumed to be induced by the \(r\)-th jet prolongation of a local bundle automorphism of \(\pi:E\to Q\). Such maps preserve \(r\)-jets and commute with the canonical projections of the jet tower. Hence, \(\Phi\) is a well-defined smooth map.

Locally, \(\Phi\) coincides with the maps \(\Phi_a\) after identifying the quotient fibration with \(\pi:E\to Q\). Therefore, \(\Phi\) is a local diffeomorphism and satisfies
\(
\Phi_*(\mathcal D)=\mathcal C^r_\pi, \Phi_*(\mathcal P)=\ker \T\pi_{r,r-1}.
\)
Thus, \(\Phi\) is a local polarised $k$-contactomorphism from \((M,\mathcal D,\mathcal P)\) to \((J^r\pi,\mathcal C^r_\pi,\ker \T\pi_{r,r-1})\). If \(\Phi\) is a global diffeomorphism onto \(J^r\pi\), then the claimed global equivalence follows.
\end{proof}
\begin{remark}
The local jet-type polarisation hypotheses encode the Spencer contractions, the vertical symbol tower and local Cartan horizontal directions. The simplicity assumptions on \(\mathcal V_r\) and \(\mathcal V_{r-1}\) produce the quotient maps \(q_0:M\to Q\) and \(q_1:M\to E\), with \(\pi\circ q_1=q_0\). The transition condition identifies the local normal forms through \(r\)-th jet prolongations over this quotient fibration, so they glue to a globally defined smooth map \(\Phi:M\to J^r\pi\). This map is a local diffeomorphism preserving the Cartan distribution and polarisation. Its image \(\Phi(M)\) is open, but \(\Phi\) need not be injective or surjective. It identifies \(M\) with \(\Phi(M)\) if it is injective; a global polarised \(k\)-contactomorphism onto \(J^r\pi\) additionally requires surjectivity.
\end{remark}

Let us now explain the passage from order \(r\) to order \(r+1\) as a polarised Legendrian prolongation.

\begin{definition}\label{def:polarised_legendrian_prolongation}
Let \((M,\mathcal D,\mathcal P)\) be a polarised \(k\)-contact manifold and assume that \(\operatorname{rank}(\mathcal D/\mathcal P)=n\). Let
\[
\operatorname{Leg}_{\mathcal P}(\mathcal D):=
\{(x,E_x): E_x\subset\mathcal D_x,\ \mathcal D_x=E_x\oplus\mathcal P_x,\ E_x\subset E_x^{\perp_{\mathcal D}}\}.
\]
If this set admits a smooth fibre bundle structure over \(M\), with projection
\(
\rho:\operatorname{Leg}_{\mathcal P}(\mathcal D)\to M,  (x,E_x)\mapsto x,
\)
such that the assignment \((x,E_x)\mapsto E_x\subset\mathcal D_x\) is smooth, then \(\operatorname{Leg}_{\mathcal P}(\mathcal D)\) is called the \emph{polarised Legendrian prolongation} of \((M,\mathcal D,\mathcal P)\). Its \emph{tautological distribution} is defined as follows 
\[
\widetilde{\mathcal D}_{(x,E_x)}:=
\{v\in \T_{(x,E_x)}\operatorname{Leg}_{\mathcal P}(\mathcal D):\T\rho(v)\in E_x\}.
\]
\end{definition}

Among all Legendrian subspaces, a relevant role is played by those that are transversal to the polarisation \cite{HLM_26}.

\begin{definition}\label{def:polarised_legendrian_submanifold}
Let \((M,\mathcal D,\mathcal P)\) be a polarised \(k\)-contact manifold. A subspace \(E_x\subset\mathcal D_x\), with $x\in M$, is called \emph{\(\mathcal P\)-polarised Legendrian} if \(E_x\) is isotropic and \(\mathcal D_x=E_x\oplus\mathcal P_x\). A submanifold \(L\subset M\) is called a \emph{\(\mathcal P\)-polarised Legendrian submanifold} if \(\T_xL\subset\mathcal D_x\) is \(\mathcal P\)-polarised Legendrian for every \(x\in L\). Equivalently,
\[
\mathcal D|_L=\T L\oplus\mathcal P|_L.
\]
If the polarisation is clear from the context, we simply say polarised Legendrian.
\end{definition}

\begin{proposition}\label{prop:jet_as_polarised_legendrian_prolongation}
For the jet model \((J^r\pi,\mathcal C^r_\pi,\mathcal G^r_\pi)\), the set
\(\operatorname{Leg}_{\mathcal G^r_\pi}(\mathcal C^r_\pi)\) admits a natural smooth fibre-bundle structure over \(J^r\pi\). More precisely, there is a canonical diffeomorphism over \(J^r\pi\),
\begin{equation}\label{prop:extLeg}
J^{r+1}\pi\simeq \operatorname{Leg}_{\mathcal G^r_\pi}(\mathcal C^r_\pi).
\end{equation}
where the projection on the left is \(\pi_{r+1,r}:J^{r+1}\pi\to J^r\pi\) and the projection on the right is the natural projection
\(\rho:\operatorname{Leg}_{\mathcal G^r_\pi}(\mathcal C^r_\pi)\to J^r\pi\).
Under this identification, the tautological distribution \(\widetilde{\mathcal D}\) coincides with the Cartan distribution \(\mathcal C^{r+1}_\pi\), and the prolongation of a holonomic submanifold \(j^r\phi(Q)\subset J^r\pi\) is \(j^{r+1}\phi(Q)\subset J^{r+1}\pi\).
\end{proposition}

\begin{proof}
Let \(\theta_{r+1}\in J^{r+1}\pi\) and set \(\theta_r:=\pi_{r+1,r}(\theta_{r+1})\). Since \(\ker \T_{\theta_{r+1}}\pi_{r+1,r}\subset(\mathcal C^{r+1}_\pi)_{\theta_{r+1}}\), the image
\[
E_{\theta_{r+1}}:=\T_{\theta_{r+1}}\pi_{r+1,r}\big((\mathcal C^{r+1}_\pi)_{\theta_{r+1}}\big)
\]
has dimension \(n\). Moreover, \(\T_{\theta_r}\pi_r\) maps \(E_{\theta_{r+1}}\) isomorphically onto \(\T_{\pi_r(\theta_r)}Q\), while \((\mathcal G^r_\pi)_{\theta_r}=\ker \T_{\theta_r}\pi_{r,r-1}\) is contained in \(\ker \T_{\theta_r}\pi_r\). Hence, \(E_{\theta_{r+1}}\) is complementary to \((\mathcal G^r_\pi)_{\theta_r}\) in \((\mathcal C^r_\pi)_{\theta_r}\). In adapted coordinates, if \(\theta_{r+1}\) has coordinates \((x^i,u^\alpha_I,u^\alpha_L)\), with \(|I|\le r\) and \(|L|=r+1\), then \(E_{\theta_{r+1}}\) is spanned by
\[
\widetilde D_i=D_i+\sum_{\alpha=1}^m\sum_{|K|=r}u^\alpha_{K+\widehat e_i}\frac{\partial}{\partial u^\alpha_K},\qquad i=1,\ldots,n.
\]
The symmetry \(u^\alpha_{K+\widehat e_i}=u^\alpha_{K'+\widehat e_j}\) whenever \(K+\widehat e_i=K'+\widehat e_j\) implies that \(E_{\theta_{r+1}}\) is isotropic. Thus, \(E_{\theta_{r+1}}\) defines a point of \(\operatorname{Leg}_{\mathcal G^r_\pi}(\mathcal C^r_\pi)\). This gives a map
\(
J^{r+1}\pi\longrightarrow\operatorname{Leg}_{\mathcal G^r_\pi}(\mathcal C^r_\pi).
\)

Conversely, let \(E_{\theta_r}\subset(\mathcal C^r_\pi)_{\theta_r}\) be an isotropic complement of \((\mathcal G^r_\pi)_{\theta_r}\). Since \(\T_{\theta_r}\pi_r:E_{\theta_r}\to \T_{\pi_r(\theta_r)}Q\) is an isomorphism, in adapted coordinates \(E_{\theta_r}\) is spanned by unique vectors of the form
\[
\widetilde D_i=D_i+\sum_{\alpha=1}^m\sum_{|K|=r}A^\alpha_{iK}\frac{\partial}{\partial u^\alpha_K},\qquad i=1,\ldots,n.
\]
The isotropy condition is equivalent to
\(
A^\alpha_{i,I+\widehat e_j}=A^\alpha_{j,I+\widehat e_i}\), for \( |I|=r-1.
\)
Indeed, this is precisely the vanishing of \(\dd\bm\eta^r_\pi(\widetilde D_i,\widetilde D_j)\). Hence, the coefficients \(A^\alpha_{iK}\) are precisely the components of a symmetric \((r+1)\)-st order jet. They determine unique coordinates \(u^\alpha_L\), for \(|L|=r+1\), by \(u^\alpha_{K+\widehat e_i}:=A^\alpha_{iK}\). The displayed symmetry is exactly the compatibility condition ensuring that the value assigned to \(u^\alpha_L\) is independent of the decomposition \(L=K+\widehat e_i\). Thus, \(E_{\theta_r}\) determines a unique element \(\theta_{r+1}\in J^{r+1}\pi\) over \(\theta_r\). The two constructions are inverse to one another and are smooth in adapted coordinates. Hence, they define a canonical diffeomorphism
\[
J^{r+1}\pi\simeq\operatorname{Leg}_{\mathcal G^r_\pi}(\mathcal C^r_\pi).
\]
In particular, \(\operatorname{Leg}_{\mathcal G^r_\pi}(\mathcal C^r_\pi)\) inherits a natural smooth fibre-bundle structure over \(J^r\pi\), and the above diffeomorphism is a bundle diffeomorphism over \(J^r\pi
\).

Under the identification \eqref{prop:extLeg}, the tautological distribution is exactly the Cartan distribution \(\mathcal C^{r+1}_\pi\), because a tangent vector to \(J^{r+1}\pi\) belongs to \(\mathcal C^{r+1}_\pi\) precisely when its projection to \(J^r\pi\) lies in the \(n\)-plane \(E_{\theta_{r+1}}\). 

Finally, let \(S=j^r\phi(Q)\). For each \(x\in Q\), the tangent space
\(
E_{j^r_x\phi}:=\T_{j^r_x\phi}j^r\phi(Q)
\)
is an \(n\)-plane contained in \((\mathcal C^r_{\pi})_{j^r_x\phi}\), complementary to \((\mathcal G^r_{\pi})_{j^r_x\phi}\), and isotropic because \(j^r\phi(Q)\) is an integral submanifold of the Cartan distribution. Hence, \((j^r_x\phi,E_{j^r_x\phi})\) defines an element of \(\operatorname{Leg}_{\mathcal G^r_\pi}(\mathcal C^r_\pi)\). The above identification sends this element to the unique \((r+1)\)-jet whose induced Cartan \(n\)-plane at order \(r\) is \(E_{j^r_x\phi}\), namely \(j^{r+1}_x\phi\). Therefore, the Legendrian prolongation of \(j^r\phi(Q)\) is \(j^{r+1}\phi(Q)\).

\end{proof}

\section{Holonomic sections and polarised Legendrian transcontact morphisms}
\label{subsec:holonomic_sections_jets}

This section is aimed at explaining how holonomic sections and Lie symmetries are expressed in the polarised \(k\)-contact language. Furthermore, $k$-contact geometry provides techniques that are natural in this formalism, but are more difficult to describe in the jet manifold framework. This happens, for instance, for transformations which do not respect the structure of the initial jet manifold.

\begin{proposition}\label{prop:holonomic_sections_are_polarised_legendrian}
Let \(S\subset J^r\pi\) be a connected submanifold. Then \(S\) is, around every one of its points, the image of an \(r\)-th holonomic prolongation \(j^r\sigma\) if and only if \(S\) is a polarised Legendrian submanifold of \((J^r\pi,\mathcal C^r_\pi,\mathcal G^r_\pi)\).
\end{proposition}

\begin{proof}
Assume first that \(j^r\sigma(U)\) is an open subset of $S$ for some local section \(\sigma:U\to E\). By the defining property of the Cartan distribution,
\(
T(j^r\sigma(U))\subset\mathcal C^r_\pi.
\)
Moreover, \(j^r\sigma(U)\) is transverse to the polarisation \(\mathcal{G}^r_\pi=\ker \T\pi_{r,r-1}\), since
\(
\pi_r\circ j^r\sigma=\operatorname{id}_U.
\)
Thus, \(j^r\sigma(U)\) is a maximal integral submanifold complementary to the polarisation in the Cartan distribution, hence a polarised Legendrian submanifold.

Conversely, assume that \(S\) is polarised Legendrian. Since \(\T S\subset\mathcal C^r_\pi\) and \(\T S\) is complementary to \(\mathcal{G}^r_\pi\) in $\mathcal{C}^r_\pi$ at points of $S$, it follows that \(\dim S=n\) and  
\(
\T\pi_r|_{\T S}:\T S\longrightarrow \T Q
\)
is an isomorphism. Therefore, locally, \(S\) is the image of a section
\(
s:U\subset Q\longrightarrow J^r\pi.
\) 
Choose adapted jet coordinates \((x^i,u^\alpha_I)\), for \(0\le |I|\le r\), and write
\(
s(x)=\bigl(x^i,u^\alpha_I(x)\bigr).
\)
Since \(\T s(\T U)\subset\mathcal C^r_\pi\), the contact forms \(\theta^\alpha_I\), with \(|I|\le r-1\) and \(\alpha=1,\ldots,m\), pull back to zero via $s$ on \(U\). In other words,
\(
s^*\theta^\alpha_I=0,
\)
which gives
\[
\frac{\partial u^\alpha_I}{\partial x^i}=u^\alpha_{I+\widehat{e}_i},\qquad |I|\le r-1,\qquad \alpha=1,\ldots,m,\qquad i=1,\ldots,n.
\]
Setting \(u^\alpha:=u^\alpha_0\), these identities imply inductively that
\[
u^\alpha_I=\frac{\partial^{|I|}u^\alpha}{\partial x^I},\qquad \alpha=1,\ldots,m,\qquad |I|\leq r.
\]
Therefore, if
\(
\sigma(x)=\bigl(x^i,u^\alpha(x)\bigr),
\)
then \(s=j^r\sigma\). Consequently,
\(
s(U)=j^r\sigma(U)
\) 
and $s(U)$ is an open subset of $S$.
\end{proof}

The following notion requires the tangent map to send polarised Legendrian planes to polarised Legendrian planes. Every polarised \(k\)-contactomorphism has this property.

\begin{definition}\label{def:legendrian_polarised_transcontact_morphism}
Let \((M,\mathcal D,\mathcal P)\) and \((N,\mathcal D',\mathcal P')\) be polarised \(k\)- and \(k'\)-contact manifolds, respectively. A smooth local map \(\Phi:U\subset M\to N\) is called a \emph{Legendrian-polarised transcontact morphism} if \(\T\Phi(\mathcal D|_U)\subset\mathcal D'\) and, for every \(x\in U\), the tangent map \(\T_x\Phi\) sends every \(\mathcal P\)-polarised Legendrian subspace \(E_x\subset\mathcal D_x\) to a \(\mathcal P'\)-polarised Legendrian subspace \(\T_x\Phi(E_x)\subset\mathcal D'_{\Phi(x)}\).
\end{definition}

\begin{example}\label{ex:jet_projections_transcontact}
In adapted coordinates, \(\T\pi_{r,s}\) with $r>s>0$ sends the total derivatives in $J^r\pi$ to elements of the Cartan distribution $\mathcal{C}^s_\pi$ on $J^s\pi$, while $\T\pi_{r,s}(\mathcal{G}^r_\pi)=0$. Hence, \(\T\pi_{r,s}(\mathcal C^r_\pi)\subset\mathcal C^s_\pi\) and \(\pi_{r,s}\) becomes a transcontact morphism. Moreover, it sends \(\mathcal G^r_\pi\)-polarised Legendrian planes to \(\mathcal G^s_\pi\)-polarised Legendrian planes, and therefore it is a Legendrian-polarised transcontact morphism. Since \(\pi_{r,s}\circ j^r\sigma=j^s\sigma\), it sends holonomic sections to holonomic sections. Jet projections therefore provide canonical examples of non-invertible maps,   which are natural in our \(k\)-contact framework.
\end{example}

\begin{example} \label{ex:first_order_hodograph}
Let \(J^1(\mathbb R,\mathbb R)\) have global adapted coordinates \((x,u,p)\) and adapted one-contact form \(\eta^1_\pi=\dd u-p\,\dd x\). On the open $\mathcal{O}$ of $J^1\pi$ of points with \(p\neq0\), consider the hodograph transformation
\[
 \Phi:(x,u,p)\in \mathcal{O}\subset J^1\pi\rightarrow (X,U,P):=(u,x,1/p)\in \mathcal{O}\subset J^1\pi.
\]
Then, one has
\[
 \Phi^*\eta^1_\pi=\dd x-\frac{1}{p}\dd u=-\frac{1}{p}(\dd u-p\,\dd x)=-\frac{1}{p}\eta^1_\pi.
\]
Thus, \(\Phi\) is a one-contactomorphism. If \(u=u(x)\) and \(p=u_x\neq0\), then \(X=u(x)\) is a valid local coordinate and the image curve is
\(
 (X,U,P)=\bigl(u(x),x,1/u_x(x)\bigr),
\)
which is the first jet of the inverse function \(U=x(X)\), since \(\dd U/\dd X=1/u_x=P\). Moreover,
\(
 \T\Phi(\partial_p)=-\frac{1}{p^2}\partial_P,
\)
so this first-order hodograph transformation preserves the   polarisation \(\mathcal{G}^1_\pi=\langle\partial_p\rangle\).  Hence, \(\Phi\) is a polarised one-contactomorphism on \(\mathcal O\) and, in particular, a Legendrian-polarised tr  anscontact morphism.

Hodograph transformations are not adapted to a particular jet fibration, because the new independent variable \(X=u\) depends on the old dependent variable. In spite of that, their role in ordinary/partial differential equations is fundamental, as many basic methods for solving differential equations are based on their use.
\end{example}

The following proposition shows how our theory can be extended to spaces that are not jet manifolds, but appear in the study of PDEs \cite{PiraniRobinsonShadwick1979,Barth2019}.

\begin{proposition}\label{prop:product_jet_spaces_polarised_k_contact}
Let \(\pi:E\to M\) and \(\pi':E'\to M\) be fibred manifolds over the same base \(M\), with \(\dim M=n\), fibre dimensions \(m\) and \(m'\), and let \(p,q\geq 1\). Set
\(
Z=J^p\pi\times_M J^q\pi',
\)
and denote by \(\rho:Z\to J^p\pi\) and \(\rho':Z\to J^q\pi'\) the canonical projections. Let \(\mathcal C^p_\pi\subset \T J^p\pi\) and \(\mathcal C^q_{\pi'}\subset \T J^q\pi'\) be the Cartan distributions, and define
\[
\mathcal D_Z=\{X\in \T Z:\T\rho(X)\in\mathcal C^p_\pi,\ \T\rho'(X)\in\mathcal C^q_{\pi'}\}.
\]
Then, \(\mathcal D_Z\) is a natural \(N\)-contact distribution on \(Z\), where
\[
N=m\binom{n+p-1}{p-1}+m'\binom{n+q-1}{q-1}.
\]
Moreover, \(Z\) carries a natural polarisation
\(
\mathcal G_Z=\ker \T(\pi_{p,p-1}\times_M\pi'_{q,q-1}).
\)
Finally, a local section \(\Psi:M\to Z\), \(\Psi=(\psi,\psi')\), is Legendrian  if and only if there exist local sections \(\sigma:M\to E\) and \(\sigma':M\to E'\) such that
\(
\Psi=(j^p\sigma,j^q\sigma').
\)
Thus, the holonomic sections of \(J^p\pi\times_M J^q\pi'\) are exactly those that decompose into two holonomic sections of the two jet factors, respectively.
\end{proposition}

\begin{proof} Choose adapted local \(N^p_\pi\)- and \(N^q_{\pi'}\)-contact forms
\(
\bm\eta^p_\pi, \bm\eta^q_{\pi'}
\)
on coordinate domains $U,U'$ of \(J^p\pi\) and \(J^q\pi'\), respectively, with Reeb frames \(\mathcal R_p\) and \(\mathcal R_q\). On \(Z=J^p\pi\times_MJ^q\pi'\), define
\[
\bm\eta_Z:=\rho^*\bm\eta^p_\pi\oplus\rho'^*\bm\eta^q_{\pi'}.
\]
Therefore, \(\bm\eta_Z\) is a local \(N\)-contact form for \(\mathcal D_Z\). 
Since \(\ker\bm\eta^p_\pi=\mathcal C^p_\pi\) and \(\ker\bm\eta^q_{\pi'}=\mathcal C^q_{\pi'}\), one has
\[
\ker\bm\eta_Z=\{X\in \T Z:\bm\eta^p_\pi(\T\rho(X))=0,\ \bm\eta^q_{\pi'}(\T\rho'(X))=0\}=\mathcal D_Z.
\]
Let us prove that \(\bm\eta_Z\) is an adapted \(N\)-contact form for \(\mathcal D_Z\).
First, \(\ker\bm\eta_Z=\mathcal D_Z\). Moreover, the Reeb fields of both factors on \(J^p\pi\) and \(J^q\pi'\) can be understood as vector fields on $Z$ so that they are tangent to the fibre product, since they do not change the base variables. Since they span a complement to \(\mathcal D_Z\), one has
\[
\T Z=\mathcal D_Z\oplus(\mathcal R_p\oplus\mathcal R_q),
\qquad
\ker \dd\bm\eta_Z=\mathcal R_p\oplus\mathcal R_q.
\]
It is immediate that $\ker \bm \eta_Z\oplus \ker  \dd \bm\eta_Z=\T(U\times_M U')$. Hence, \(\mathcal D_Z\) is a natural \(N\)-contact distribution on \(Z\).

The polarisation is obtained by taking the product of the highest-order polarisations of the two jet spaces. Indeed,
\[
\ker \T(\pi_{p,p-1}\times_M\pi'_{q,q-1})=\{X\in \T Z:\T\rho(X)\in\mathcal G^p_\pi,\ \T\rho'(X)\in\mathcal G^q_{\pi'}\}.
\]
Since \(\mathcal G^p_\pi\subset\mathcal C^p_\pi\) and \(\mathcal G^q_{\pi'}\subset\mathcal C^q_{\pi'}\), we have \(\mathcal G_Z\subset\mathcal D_Z\). Its integrability follows from the integrability of the polarisations of the two factors and the fact that they belong to the kernels of $\T\rho$ and $\T \rho'$, respectively. Its isotropy follows from the fact that \(\bm\eta_Z\) is the direct sum of the two adapted \(k\)-contact forms and each factor polarisation is isotropic for the corresponding adapted \(N^p_\pi\)- or \(N^q_{\pi'}\)-contact structure. Let us prove the maximality of the rank. This is a pointwise statement. Fix \(z\in Z\), and let \(L\subset(\mathcal D_Z)_z\) be an isotropic subspace. Since \(\bm\eta_Z=\rho^*\bm\eta^p_\pi\oplus\rho'^*\bm\eta^q_{\pi'}\), the projection \(\T_z\rho(L)\subset(\mathcal C^p_\pi)_{\rho(z)}\) is isotropic. Hence, by the pointwise rank estimate in the proof of Lemma~\ref{lem:jet_symbol_dominant},
\[
 \dim\T_z\rho(L)\leq \operatorname{rank}\mathcal G^p_\pi .
\]
Moreover,
\[
 \ker(\T_z\rho|_L)\subset \ker\T_z\rho\cap(\mathcal D_Z)_z .
\]
If \(X\in\ker\T_z\rho\cap(\mathcal D_Z)_z\), then \(\T_z\rho'(X)\in(\mathcal C^q_{\pi'})_{\rho'(z)}\) and its projection to \(\T M\) vanishes, because \(X\) is tangent to the fibre product and \(\T_z\rho(X)=0\). Therefore, \(\T_z\rho'(X)\in(\mathcal G^q_{\pi'})_{\rho'(z)}\). Thus,
\[
 \dim\ker(\T_z\rho|_L)\leq \operatorname{rank}\mathcal G^q_{\pi'} .
\]
Consequently,
\[
 \dim L
=
 \dim\T_z\rho(L)+\dim\ker(\T_z\rho|_L)
 \leq
 \operatorname{rank}\mathcal G^p_\pi+\operatorname{rank}\mathcal G^q_{\pi'}
 =
 \operatorname{rank}\mathcal G_Z .
\]
Since \((\mathcal G_Z)_z\) itself is isotropic and has this rank, \(\mathcal G_Z\) has maximal possible rank among isotropic subspaces of \((\mathcal D_Z)_z\), and therefore among integrable isotropic subbundles admitting an isotropic complement. Moreover, in adapted coordinates the diagonal total derivatives
\[
D_i^Z=\frac{\partial}{\partial x^i}
+\sum_{\alpha,\,|I|\le p-1}u^\alpha_{I+\widehat e_i}\frac{\partial}{\partial u^\alpha_I}
+\sum_{A,\,|J|\le q-1}v^A_{J+\widehat e_i}\frac{\partial}{\partial v^A_J}
\]
span an integrable isotropic complement of \(\mathcal G_Z\) in \(\mathcal D_Z\). Hence, \(\mathcal G_Z\) is Legendrian and has maximal possible rank among isotropic subbundles of \(\mathcal D_Z\). Together with its integrability and the existence of the isotropic complement induced by the total derivative fields, this proves that \(\mathcal G_Z\) is a natural polarisation of \(\mathcal D_Z\).

Let now \(\Psi:M\to Z\) be a local section and write \(\Psi=(\psi,\psi')\), where \(\psi=\rho\circ\Psi\) and \(\psi'=\rho'\circ\Psi\). By definition of \(\mathcal D_Z\) one has
\(
\T\Psi(\T M)\subset\mathcal D_Z
\)
if and only if
\[
\T\psi(\T M)\subset\mathcal C^p_\pi,\qquad \T\psi'(\T M)\subset\mathcal C^q_{\pi'}.
\]
By the holonomicity criterion for jet bundles, these two conditions are equivalent to the existence of local sections \(\sigma:M\to E\) and \(\sigma':M\to E'\) such that \(\psi=j^p\sigma\) and \(\psi'=j^q\sigma'\). Hence, \(\Psi=(j^p\sigma,j^q\sigma')\). The converse is immediate because holonomic prolongations are tangent to the corresponding Cartan distributions. Therefore, the sections of \(Z\) tangent to \(\mathcal D_Z\) are precisely the simultaneous holonomic sections of the two jet factors.
\end{proof}

\begin{definition}\label{def:cartan_and_polarised_symmetries}
A vector field \(X\in\mathfrak X(J^r\pi)\) is a \emph{Cartan symmetry}, or \emph{\(N^r_\pi\)-contact vector field}, if 
\(
 [X,\Gamma(\mathcal C^r_\pi)]\subset\Gamma(\mathcal C^r_\pi).
\)
It is a \emph{polarised Cartan  symmetry} if, in addition,
\(
 [X,\Gamma(\mathcal{G}^r_\pi)]\subset\Gamma(\mathcal{G}^r_\pi)
\).
\end{definition}

\begin{proposition}
If \(X\in\mathfrak X(E)\), then \(X^{(r)}\) has a local flow formed by polarised transcontactomorphisms of
\(
 (J^r\pi,\mathcal C^r_\pi,\mathcal{G}^r_\pi).
\)
Equivalently, if \(\{\Phi_t\}_{t\in \mathbb{R}}\) is the local flow of \(X\) and \(J^r\Phi_t\) denotes its \(r\)-th jet prolongation, then
\(J^r\Phi_t\) is a local polarised transcontactomorphism of the Cartan \(N^r_\pi\)-contact distribution.
\end{proposition}

\begin{proof}
By Theorem~\ref{thm:prolonged_point_symmetries_are_polarised}, the prolonged vector field \(X^{(r)}\) preserves both the Cartan distribution and the polarisation.
Let \(\Psi_t=J^r\Phi_t\) be the local flow of \(X^{(r)}\). The infinitesimal invariance of \(\mathcal C^r_\pi\) implies that \((\Psi_t)_*\mathcal C^r_\pi=\mathcal C^r_\pi\) wherever the flow is defined. Similarly, the infinitesimal invariance of \(\mathcal{G}^r_\pi\) implies that \((\Psi_t)_*\mathcal{G}^r_\pi=\mathcal{G}^r_\pi\). Hence, every \(\Psi_t\) is a local polarised transcontactomorphism of \((J^r\pi,\mathcal C^r_\pi,\mathcal{G}^r_\pi)\). Since \(\mathcal C^r_\pi\) is the Cartan \(N^r_\pi\)-contact distribution, each \(\Psi_t\) is, in particular, a local \(N^r_\pi\)-contactomorphism.
\end{proof}

Coverings \cite{MurielRomero2012}, Bäcklund transformations and Lax representations  typically require auxiliary variables or relations between different systems, and therefore do not naturally live inside a single fixed jet bundle. In the polarised \(k\)-contact language, one can enlarge the ambient space, restrict it to opens, or consider the induced and extended Cartan distributions systematically. This will be analysed in detail in future works.

\section{Characteristics of Lie symmetries in \texorpdfstring{$k$}{k}-contact geometry}\label{sec:lie_characteristics_hamiltonian}
This section reformulates and extends characteristics and Lie symmetries of differential equations in the language of polarised \(k\)-contact geometry of jet type. In particular, a \(k\)-contact vector field on a co-oriented \(k\)-contact manifold takes values in the $k$-contact distribution exactly at the zero level set of an associated $k$-contact Hamiltonian, which gives a subset that is invariant under its flow. The same result holds locally for general \(k\)-contact vector fields with respect to a local $k$-contact form.
The polarisation allows for the definition of another invariant subset. On jet bundles, prolonged and evolutionary Lie symmetries give rise to local \(N^r_\pi\)-contact Hamiltonians that can be generated locally by $m$-functions, which correspond to their classical characteristics. The adapted \(N^r_\pi\)-contact Hamiltonian bracket induces a bracket in the characteristics that matches the classical one \cite{Olver1993}. For general Cartan symmetries the \(k\)-contact Hamiltonian bracket is still defined, generalising previous formalisms. We briefly describe how our techniques retrieve other known techniques in the literature. Let us start this section by two simple definitions.
\begin{definition}
    The \emph{locus of a vector field} $X\in \mathfrak{X}(M)$ relative to  a $k$-contact manifold $(M,\mathcal{D})$ is the subset
    $$
    \mathcal{L}_\mathcal{D}(X)=\{x\in M:X_x\in \mathcal{D}_x\}.
    $$ Moreover, the \emph{polarised locus} of $X\in \mathfrak{X}(M)$ relative to $(M,\mathcal{D},\mathcal{P})$ is the subset $$\mathcal{L}_{\mathcal{P}}(X)=\{x\in M:X_x\in \mathcal{P}_x\}.$$
\end{definition}
One of the reasons to introduce the above definitions is given in the following proposition. This is an extension of a standard result for reducing the flow of the prolongation to $J^r\pi$ of  a vector field on $\pi:E\rightarrow Q$ (see \cite{Olver1993}).

\begin{proposition}\label{prop:locus_hamiltonian_zero}
Let \((M,\bm\eta)\) be a co-oriented \(k\)-contact manifold and let \(X\in\mathfrak X(M)\) be a \(k\)-contact vector field with $\bm\eta$-Hamiltonian $\bm h_X$. Then, \(\mathcal L_\mathcal{D}(X)=\bm h_X^{-1}(0)\) and \(\mathcal L_\mathcal{D}(X)\) is invariant under the local flow of \(X\). If, in addition, $X$ is a polarised $k$-contact vector field relative to \((M,\ker\bm\eta,\mathcal P)\), then \(\mathcal L_{ \mathcal P}(X)\) is invariant under the local flow of \(X\). 
\end{proposition}

\begin{proof}
Since \(\mathcal D=\ker\bm\eta\), the condition \(X_x\in\mathcal D_x\) is equivalent to \(\bm\eta_x(X_x)=0\), namely to \({\bm h}_X(x)=0\), which implies that \(\mathcal L_{\mathcal{D}}(X)={\bm h}_X^{-1}(0)\). Since \(X\) is a \(k\)-contact vector field, using the Hamiltonian identity
\[
 \iota_X \dd\bm\eta=\dd\bm h_X-\sum_{\alpha=1}^k\eta^\alpha\,\mathcal L_{R_\alpha}\bm h_X,
\]
and contracting with \(X\) and using \(\bm h_X=-\iota_X\bm\eta\), we obtain
\[
0=\iota_X\iota_X \dd\bm\eta
 =\mathcal{L}_X\bm h_X-\sum_{\alpha=1}^k\eta^\alpha(X)\mathcal L_{R_\alpha}\bm h_X=\mathcal{L}_{X}\bm h_X+\sum_{\alpha=1}^kh_X^\alpha \mathcal{L}_{R_\alpha}\bm h_X.
\]
Hence, \(X\) is tangent to \(\bm h_X^{-1}(0)\), and the local flow of \(X\) preserves this locus. Finally, if \(X\) is polarised, then its local flow $\varphi_t$ preserves \(\mathcal P\). Since \(X\) is invariant under its own flow, if \(X_x\in\mathcal P_x\) then \(X_{\varphi_t(x)}\in\mathcal P_{\varphi_t(x)}\).  Hence, \(\mathcal L_{\mathcal P}(X)\) is invariant under the flow of \(X\).
\end{proof}

The following definition extends to polarised $k$-contact manifolds the definition of an evolutionary vector field.

 \begin{definition} On a $k$-contact manifold $(M,\mathcal{D},\mathcal{P})$ with a polarisation of jet type and order $r$, 
 a \emph{classical prolongation} relative to $\mathcal{V}_r$ is a Lie symmetry of $\mathcal{D}$ leaving invariant $\mathcal{V}_{r-1}$. Meanwhile, an \emph{evolutionary vector field} is a $k$-contact vector field taking values in $\mathcal{V}_r$.
 \end{definition}
 
It is worth noting that,  an evolutionary \(k\)-contact vector field on $J^r\pi$ is locally the \(r\)-th prolongation of a vertical vector field on \(E\).  Since  the above $(M,\mathcal{D},\mathcal{P})$ is locally polarised $k$-contactomorphic to $(J^r\pi,\mathcal{C}^r_\pi,\mathcal{G}^r_\pi)$ by Theorem \ref{thm:local_split_jet_recognition}, one has that \(\mathcal V_{r-1}\) is identified with \(\ker \T\pi_{r,0}\). Then, a classical jet prolongation of a vector field amounts to a $k$-contact vector field of  $(J^r\pi,\mathcal{C}^r_\pi)$ projecting onto a vector field on $E$. Meanwhile, an evolutionary vector field is equivalent to an evolutionary vector field on a jet manifold.
 These relations allow us to describe the theory of characteristics of Lie symmetries as a consequence of the local Hamiltonian structure linked to $k$-contact vector fields.

In the case of evolutionary vector fields or prolongations to $J^r\pi$ of vector fields on $E$, it is known that the coordinates of $X\in \mathfrak{X}_{\mathcal{C}^r_\pi}(J^r\pi)$ on an adapted coordinate system on $U\subset J^r\pi$ can be determined by means of characteristics. In particular,  in such a coordinate system $(x^i,u^\alpha_J)$ on $U\subset J^r\pi$, one has that if $X$ is of the above mentioned types and we write $X=\sum_{i=1}^n\xi^i\partial/\partial x^i+\sum_{\alpha=1}^m\sum_{|J|\leq r}\phi^\alpha_J\partial/\partial u^\alpha_J$,  then \(\phi^\alpha_J=D_JQ^\alpha+\sum_{i=1}^nu^\alpha_{J+\widehat e_i}\xi^i\), for \(|J|\le r-1\), and  \begin{equation}\label{eq:special}
\phi_{I+\widehat e_j}^\alpha=D_j\phi^\alpha_I-\sum_{i=1}^n(D_j\xi^i)u^\alpha_{I+\widehat{e}_i},\qquad |I|=r-1,
\end{equation} where \(Q^\alpha=\phi^\alpha-\sum_{i=1}^n\xi^iu^\alpha_{\widehat e_i}\) for $\alpha=1,\ldots,m$. Hence, it makes sense to write \(X|_U=X_{\bm Q}^{(r)}\). Moreover, using the form of $\bm\eta^1_\pi$ on the local coordinate system induced on $\pi_{r,1}U$ by the  one on $U$, it follows that  ${\bm Q}=\iota_X\pi_{r,1}^*\bm \eta^1_\pi$. 

 \begin{proposition}\label{prop:triangular_hamiltonian_characteristics}
Let \(\bm \eta^r_\pi\) be the adapted \(N^r_\pi\)-contact form associated with an adapted coordinate system \((x^i,u^\alpha_I)\) on \(U\subset J^r\pi\). Let \(X\in\mathfrak X(U)\) be either the \(r\)-th prolongation of a vector field on \(E\) or an evolutionary \(N^r_\pi\)-contact vector field on \(J^r\pi\), and set \(\bm Q=\iota_X\pi_{r,1}^*\bm\eta^1_\pi\). Then, 
\(\bm h_X=-\iota_X\bm \eta^r_\pi\) takes the form
\[
h^\alpha_J=-\sum_{L\ge J,\ |L|\le r-1}(-1)^{|L-J|}\frac{x^{L-J}}{(L-J)!}D_LQ^\alpha,\qquad |J|\le r-1,
\]
while
\(
Q^\alpha=-\sum_{|L|\le r-1}\frac{x^L}{L!}h^\alpha_L,
\) for $\alpha=1,\ldots,m$. Thus, $\bm Q$ determines uniquely $\bm h_{X}$ and $X$, and makes sense to write $X=X^{(r)}_{\bm Q}$. Additionally, one obtains the induced Lie bracket on characteristics given by
\[
[Q,P]^\alpha_{\mathrm{ch}}=-\sum_{|L|\le r-1}\frac{x^L}{L!}\,\{\bm h_Q,\bm h_P\}_{\bm\eta^r_\pi,L}^{\alpha}.,\qquad \alpha=1,\ldots,m.
\]
In the case of evolutionary vector fields, this bracket reads $[Q,P]^\alpha_{\mathrm{ch}}=X_Q^{(r)}(P^\alpha)-X_P^{(r)}(Q^\alpha)$ for $\alpha=1,\ldots,m$.
\end{proposition}

\begin{proof} A straightforward computation shows that $\bm Q=\iota_X\pi^*_{r,1}\bm \eta^1_\pi$ means that $Q^\alpha=\varphi^\alpha-\sum_{i=1}^n\xi^i u_{\widehat{e}_i}^\alpha$ for $\alpha=1,\ldots,m$. 
In view of \eqref{eq:adapted_eta} and using \(\iota_{X_Q^{(r)}}\theta^\alpha_L=D_LQ^\alpha\) for $\alpha=1,\ldots,m$ and $|L|<r$, gives
\[
h^\alpha_J=-\iota_{X_Q^{(r)}}\eta^\alpha_J=-\!\!\!\sum_{L\ge J,\ |L|\le r-1}(-1)^{|L-J|}\frac{x^{L-J}}{(L-J)!}D_LQ^\alpha,\qquad \alpha=1,\ldots,m,\qquad |J|<r .
\]
Consider now the left-hand equation in \eqref{eq:relEtaTheta} and  contract it with \(X_{\bm Q}^{(r)}\) and use $J=0$. Hence,
\[
 Q^\alpha =- \sum_{\ |L|\le r-1}\frac{x^{L}}{L!}h^\alpha_L,\qquad \alpha=1,\ldots,m.
\]
This proves that \(\bm h_X\) and $X$ are determined by the ${\bm Q}$. In fact, in the case of evolutionary prolongations, ${\bm Q}$ determines that $Xu^\alpha=Q^\alpha$, for $\alpha=1,\ldots,m$,  and $Xx^i=0$ for $i=1,\ldots,n$, which allows us to determine remaining coordinates of $X$ by using the recurrence obtained by using that it is a $k$-contact vector field, namely $\phi^\alpha_J=D_JQ^\alpha$ for $\alpha=1,\ldots,m$ and $|J|\leq r$. In the case of a prolongation of a vector field on \(E\), the coefficients $\phi^\alpha_J$ for $0<|J|<r$ can be obtained by means of $D_JQ^\alpha=\iota_{X}\theta^\alpha_J$ and $\phi^\alpha_J$, for $|J|=r$, from \eqref{eq:special}. Moreover,
$Q^\alpha(x^j,u^\beta,u_i^\beta)=\phi^\alpha(x^j,u^\beta)-\sum_{i=1}^n u^\alpha_i\xi^i(x^j,u^\beta)$ for $\alpha=1,\ldots,m$. Hence,  the coefficients \(\xi^i\) are recovered from the affine dependence of \(Q^\alpha\) on the first-order jet variables, namely \(\xi^i=-\partial Q^\alpha/\partial u^\alpha_i\), with no summation over \(\alpha\), and the right-hand side  is independent of the \(u^\alpha_i\); moreover \(\partial Q^\alpha/\partial u^\beta_i=0\) for \(\beta\ne\alpha\). Then, \(\phi^\alpha=Q^\alpha+\sum_i u^\alpha_i\xi^i\), which is a well-defined formula because $X$ is projectable onto $E$ by assumption.

Finally, in the evolutionary case, the characteristics are the components of the vector field along \(\partial/\partial u^\alpha\). Hence, one has
\[
[X_Q^{(r)},X_P^{(r)}]=X_{[Q,P]_{\mathrm{ch}}}^{(r)},\qquad [Q,P]^\alpha_{\mathrm{ch}}=X_Q^{(r)}(P^\alpha)-X_P^{(r)}(Q^\alpha).
\]
Meanwhile,
\[
\{\bm h_Q,\bm h_P\}_{\bm\eta^r_\pi}=\bm h_{[X_Q^{(r)},X_P^{(r)}]}=\bm h_{X_{[Q,P]_{\mathrm{ch}}}^{(r)}},
\]
and applying the inverse formula with \(J=\varnothing\) yields
\[
[Q,P]^\alpha_{\mathrm{ch}}=-\sum_{|L|\le r-1}\frac{x^L}{L!}\,[\bm h_Q,\bm h_P]^\alpha_L,\qquad \alpha=1,\ldots,m.
\]
For genuine prolongations of vector fields on \(E\), the formula strictly above gives the characteristic of the commutator.
\end{proof}

The last proposition shows that the standard theory of characteristics for finite-order jet manifolds is recovered from the local Hamiltonian structure of the jet manifolds as $N^r_\pi$-contact manifolds. Moreover, our theory can be applied quite similarly to other cases. 

Let us provide now an extension of the theory of characteristics to general \(k\)-contact vector fields on jet manifolds. The existence of generating characteristics for symmetries of the Cartan
distribution is classical and is usually formulated in terms of evolutionary
representatives on infinite-order jet bundles; see
\cite{Olver1993,KV_99,Saunders_89}. We record below its finite-order form,
which will allow us to identify the generating characteristics of an arbitrary
Cartan vector field with the components of its local
\(N_\pi^r\)-contact Hamiltonian.

\begin{proposition}\label{prop:general_cartan_characteristics}
Let \(X\in\mathfrak X(J^r\pi)\) be a Cartan vector field and let \((x^i,u^\alpha_I)\) be adapted coordinates. Set \(Q^\alpha_I:=\theta^\alpha_I(X)\), for \(|I|\leq r-1\) and \(\alpha=1,\ldots,m\). Then \(Q^\alpha_{I+\widehat e_i}=D_iQ^\alpha_I\) for \(|I|\leq r-2\). Consequently, if \(Q^\alpha:=Q^\alpha_0\), then \(Q^\alpha_I=D_IQ^\alpha\) for every \(|I|\leq r-1\). Hence, the characteristic components of every finite-order Cartan vector field are generated by the \(m\) functions \(\bm Q=(Q^1,\ldots,Q^m)\).
\end{proposition}
\begin{proof}
Write
\[
X=\sum_{j=1}^n\xi^j\frac{\partial}{\partial x^j}
+\sum_{\alpha=1}^m\sum_{|J|\leq r}\phi^\alpha_J\frac{\partial}{\partial u^\alpha_J}.
\]
Since \(X\) preserves \(\mathcal C^r_\pi\), one has \([X,D_i]\in\Gamma(\mathcal C^r_\pi)\). Therefore,
\begin{equation}\label{eq:cartan_characteristic_coefficients}
0=\theta^\alpha_I([X,D_i])
=\phi^\alpha_{I+\widehat e_i}-D_i\phi^\alpha_I
+\sum_{j=1}^nu^\alpha_{I+\widehat e_j}D_i\xi^j,
\qquad |I|\leq r-1.
\end{equation}
On the other hand, \(Q^\alpha_I=\phi^\alpha_I-\sum_{j=1}^nu^\alpha_{I+\widehat e_j}\xi^j\). Hence, for \(|I|\leq r-2\), equation~\eqref{eq:cartan_characteristic_coefficients} gives
\[
\begin{aligned}
D_iQ^\alpha_I
&=D_i\phi^\alpha_I
-\sum_{j=1}^nu^\alpha_{I+\widehat e_i+\widehat e_j}\xi^j
-\sum_{j=1}^nu^\alpha_{I+\widehat e_j}D_i\xi^j\\
&=\phi^\alpha_{I+\widehat e_i}
-\sum_{j=1}^nu^\alpha_{I+\widehat e_i+\widehat e_j}\xi^j
=Q^\alpha_{I+\widehat e_i}.
\end{aligned}
\]
Iteration yields \(Q^\alpha_I=D_IQ^\alpha\) for every \(|I|\leq r-1\).
\end{proof}
The novelty here is not the characteristic recurrence itself, but its
finite-dimensional \(k\)-contact Hamiltonian interpretation. In particular,
the adapted \(N_\pi^r\)-contact form packages the characteristic
\(Q^\alpha\) and its total derivatives into a single vector-valued
Hamiltonian on \(J^r\pi\).
\begin{remark}
Proposition~\ref{prop:general_cartan_characteristics} does not assert that every arbitrary \(m\)-tuple of functions on \(J^r\pi\) determines a Cartan vector field. It states that, whenever a finite-order Cartan vector field exists, all its characteristic components are generated by the \(m\) functions \(Q^\alpha=\theta^\alpha_0(X)\). Point prolongations, evolutionary vector fields and general Cartan vector fields are distinguished by the admissible jet dependence of these functions.
\end{remark}

\section{Systems of PDEs in {\it k}-contact geometry}\label{sec:kConPDEs}
This section develops a \(k\)-contact formulation of systems of PDEs. It describes equations as submanifolds of finite jet spaces endowed with induced Cartan and polarised structures, reformulates classical solutions as polarised Legendrian integral submanifolds, and shows how Hamiltonian loci, initial data and the vertical symbol encode normality, degeneracy and characteristic obstructions.

A \emph{system of partial differential equations of order \(r\)} is a regular embedded submanifold \(\mathcal E\hookrightarrow J^r\pi\). Equivalently, around every point of \(\mathcal E\), there exist an open subset \(U\subset J^r\pi\) and a smooth map \(F=(F^1,\ldots,F^s):U\to\mathbb R^s\) such that \(\mathcal E\cap U=F^{-1}(0)\) and \(F\) has constant rank $\dim J^r\pi-\dim \mathcal E$ along \(\mathcal E\cap U\).

\begin{definition}
The \emph{induced Cartan distribution} of the system of PDEs \(\mathcal E\subset J^r\pi\) is
\(
\mathcal C_{\mathcal E}:=\T\mathcal E\cap\mathcal C^r_\pi.
\)
If \(\mathcal C_{\mathcal E}\) has constant rank, \(\mathcal E\) is called a \emph{regular PDE}. Meanwhile, we write
\(
\mathcal G^r_{\mathcal E}:=\T\mathcal E\cap\mathcal G^r_\pi.
\) 
We call \(\mathcal G^r_{\mathcal E}\) the \emph{geometric symbol} of the differential equation $\mathcal{E}$. The elements of $\mathcal{G}^r_\mathcal{E}$ are called {\it symbol directions}. 
\end{definition}

The induced Cartan distribution is also called the \emph{Vessiot distribution} of the differential equation \(\mathcal E\) (see \cite[Definition 9]{Fesser_Seiler_2009}). In general, \(\mathcal C_{\mathcal E}\) need not be a \(k\)-contact distribution on \(\mathcal E\); it is simply the distribution of infinitesimal directions in \(\mathcal E\) compatible with the ambient Cartan structure. The distribution \(\mathcal C_{\mathcal E}\) is not, in general, the distribution tangent to the solutions. It may contain vertical directions, which are described by 
\(
\mathcal G^r_{\mathcal E}
\)
(see \cite[Definition 7.1]{Seiler_2007_Spencer}) that cannot be tangent to a genuine holonomic solution. Thus, a solution must be tangent to \(\mathcal C_{\mathcal E}\) and transverse to $\mathcal{G}^r_\pi$.

A \emph{classical solution} of \(\mathcal E\) is a local section \(\phi:U\subset M\to E\) such that \(j^r\phi(U)\subset\mathcal E\). Equivalently, it is an \(n\)-dimensional submanifold \(S\subset\mathcal E\), where \(n=\dim M\), satisfying
\[
\T S\subset\mathcal C_{\mathcal E},\qquad \T S\cap\mathcal G^r_\pi=0.
\]
Since \(\mathcal C^r_\pi\cap\ker \T\pi_r=\mathcal G^r_\pi\), the second condition is equivalent, for an \(n\)-dimensional integral submanifold of \(\mathcal C^r_\pi\), to \(\T\pi_r|_{\T S}:\T S\to \T M\) being an isomorphism. Thus, solutions are precisely the holonomic, or polarised Legendrian, integral submanifolds contained in \(\mathcal E\).

\begin{definition}\label{def:polarised_cartan_legendrian_solution}
A \emph{solution} of \(\mathcal E\subset J^r\pi\) is a polarised Legendrian submanifold \(S\subset J^r\pi\) such that \(S\subset\mathcal E\).
\end{definition}

For a solution $S$ contained in \(\mathcal E\), the corresponding decomposition inside the equation is
\[
\mathcal C_{\mathcal E}|_S=\T S\oplus\mathcal G^r_{\mathcal E}|_S.
\]
The decomposition is into subbundles over $S$ provided \(\mathcal G^r_{\mathcal E}\) has constant rank along \(S\). This is the intrinsic version of the usual separation between solution directions and symbol directions. The following proposition is immediate.

\begin{proposition}\label{prop:pde_solutions_as_cartan_legendrian}
Let \(\mathcal E\subset J^r\pi\) be a regular PDE and let \(S\subset\mathcal E\) be an embedded \(n\)-dimensional submanifold. Then, \(S\) is locally the image of \(j^r\phi\) for a solution \(\phi\) of \(\mathcal E\) if and only if
\[
\T S\subset \T\mathcal E\cap\mathcal C^r_\pi,\qquad \T S\cap\mathcal G^r_\pi=0.
\]
Equivalently, \(\T S\subset\mathcal C_{\mathcal E}\) and \(\T\pi_r|_{\T S}:\T S\to \T M\) is an isomorphism.
\end{proposition}

\begin{theorem}\label{thm:pde_solutions_polarised_legendrian}
Let \(\mathcal E\subset J^r\pi\) be a regular system of PDEs. There is a one-to-one correspondence between germs of local classical solutions \(\phi:U\subset M\to E\) of \(\mathcal E\) and \(n\)-dimensional polarised Legendrian submanifolds \(S\subset\mathcal E\).
\end{theorem}

\begin{proof}
This is a consequence of Proposition~\ref{prop:pde_solutions_as_cartan_legendrian} applied to images of local sections. The transversality condition, equivalently complementarity with \(\mathcal G^r_\pi\), excludes purely vertical integral directions and distinguishes genuine holonomic solutions from arbitrary integral submanifolds of \(\mathcal C_{\mathcal E}\).
\end{proof}

\begin{proposition}\label{prop:pde_prolongation_as_polarised_legendrian}
Let \(\mathcal E\subset J^r\pi\) be a regular differential equation. Under the canonical identification \(J^{r+1}\pi\simeq\operatorname{Leg}_{\mathcal G^r_\pi}(\mathcal C^r_\pi)\) of Proposition~\ref{prop:jet_as_polarised_legendrian_prolongation}, the first prolongation \(\mathcal E^{(1)}\subset J^{r+1}\pi\) is identified with
\[
 \operatorname{Leg}_{\mathcal G^r_\pi}(\mathcal C^r_\pi;\mathcal E)
 :=
 \{(e,H_e)\in\operatorname{Leg}_{\mathcal G^r_\pi}(\mathcal C^r_\pi): e\in\mathcal E,\ H_e\subset\T_e\mathcal E\}.
\]
Thus, \(\mathcal E^{(1)}\) is the polarised Legendrian prolongation of \(\mathcal E\): it parametrises the polarised Legendrian \(n\)-planes of \((J^r\pi,\mathcal C^r_\pi,\mathcal G^r_\pi)\) based at points of \(\mathcal E\) and tangent to \(\mathcal E\).
\end{proposition}

\begin{proof}
Let \(\theta_{r+1}\in J^{r+1}\pi\) and set \(e=\pi_{r+1,r}(\theta_{r+1})\). By Proposition~\ref{prop:jet_as_polarised_legendrian_prolongation}, \(\theta_{r+1}\) determines the polarised Legendrian \(n\)-plane \(H_{\theta_{r+1}}\subset\mathcal C^r_{\pi,e}\), namely the image of \((\mathcal C^{r+1}_\pi)_{\theta_{r+1}}\) under \(\T\pi_{r+1,r}\). Suppose that \(\mathcal E\) is locally defined by independent equations \(F^a=0\). Then \(\theta_{r+1}\in\mathcal E^{(1)}\) if and only if \(F^a(e)=0\) and \(D_iF^a(\theta_{r+1})=0\) for all \(a,i\). Since \(H_{\theta_{r+1}}\) is spanned by the total derivatives \(D_i\) evaluated at \(\theta_{r+1}\), these conditions are equivalent to \(e\in\mathcal E\) and \(\dd F^a(H_{\theta_{r+1}})=0\) for all \(a\), that is, to \(H_{\theta_{r+1}}\subset\T_e\mathcal E\). Hence \(\mathcal E^{(1)}\) is exactly the subset of \(\operatorname{Leg}_{\mathcal G^r_\pi}(\mathcal C^r_\pi)\) formed by polarised Legendrian \(n\)-planes based on \(\mathcal E\) and tangent to \(\mathcal E\).
\end{proof}

\begin{proposition}\label{prop:foliations_by_solutions}
Let \(\mathcal H\subset\mathcal C_{\mathcal E}\) be a rank-\(n\) distribution on a regular \(\mathcal E\), i.e. a rank-\(n\) subbundle of \(\T\mathcal E\). Then, \(\mathcal H\) defines a local foliation of \(\mathcal E\) by solutions if and only if \(\mathcal H\cap\mathcal G^r_\pi|_{\mathcal E}=0\) and \([\Gamma(\mathcal H),\Gamma(\mathcal H)]\subset\Gamma(\mathcal H)\). Equivalently, \(\T\pi_r|_{\mathcal H}:\mathcal H\to \T M\) is a fibrewise isomorphism and \(\mathcal H\) is involutive.
\end{proposition}

\begin{proof}
If \(\mathcal H\) is tangent to a local foliation of \(\mathcal E\) by holonomic solutions, then \(\mathcal H\subset\mathcal C_{\mathcal E}\), \(\operatorname{rank}\mathcal H=n\), and \(\mathcal H\) is involutive. Since the leaves are locally images of holonomic sections of \(\pi_r:J^r\pi\to M\), their tangent spaces project isomorphically onto \(\T M\). Hence, they are transverse to \(\ker\T\pi_r\cap\mathcal C^r_\pi=\mathcal G^r_\pi\), and therefore \(\mathcal H\cap\mathcal G^r_\pi|_{\mathcal E}=0\). 
Conversely, assume that \(\mathcal H\subset\mathcal C_{\mathcal E}\) has rank \(n\), satisfies \(\mathcal H\cap\mathcal G^r_\pi|_{\mathcal E}=0\), and is involutive. Since \(\ker(\T\pi_r|_{\mathcal C^r_\pi})=\mathcal G^r_\pi\), the condition \(\mathcal H\cap\mathcal G^r_\pi|_{\mathcal E}=0\) implies that \(\T\pi_r|_{\mathcal H}\) is injective. As both \(\mathcal H\) and \(\T M\) have rank \(n\), it is a fibrewise isomorphism. By Frobenius' theorem, \(\mathcal H\) has local \(n\)-dimensional integral leaves \(S\subset\mathcal E\). Since \(\T\pi_r|_{\T S}\) is an isomorphism, each such leaf is locally the image of a section of \(\pi_r\). Since \(\T S\subset\mathcal C^r_\pi\), the standard holonomicity criterion for jet spaces implies that this section is holonomic, say \(S=j^r\sigma(M)\) locally. Finally, because \(S\subset\mathcal E\), the section \(\sigma\) is a local solution of \(\mathcal E\).
\end{proof}

We now describe the infinitesimal space of polarised solution directions through a prescribed initial datum. Let \(e\in\mathcal E\) and let \(I_e\subset\mathcal C_{\mathcal E,e}\) be an isotropic subspace, usually the tangent space to an initial submanifold contained in \(\mathcal E\).
Since every solution direction \(H_e\) is isotropic, any solution direction containing \(I_e\) must satisfy \(I_e\subset H_e\subset I_e^{\perp_{\mathcal C}}\). This gives the following intrinsic replacement for the usual coordinate description of admissible Cauchy directions.

\begin{definition}\label{def:polarisation_space_initial_data}
Let \(I_e\subset\mathcal C_{\mathcal E,e}\) be an isotropic subspace. The \emph{polarised solution set over \(I_e\)} is
\[
\operatorname{Pol}_{\mathcal E}(I_e)
:=
\{H_e\in\operatorname{Gr}_n(I_e^{\perp_{\mathcal C}}):I_e\subset H_e,\ H_e\cap\mathcal G^r_{\pi,e}=0,\ H_e\subset H_e^{\perp_\mathcal C}\cap \mathcal C_{\mathcal E,e}\}.
\]
If \(N\subset\mathcal E\) is an isotropic initial submanifold with \(\T N\subset\mathcal C_{\mathcal E}\), we write \(\operatorname{Pol}_{\mathcal E}(\T N)\to N\) for the corresponding family of polarised solution sets.
\end{definition}

The condition \(H_e\subset I_e^{\perp_{\mathcal C}}\) is the \(k\)-contact form of the compatibility condition between the initial directions and the missing solution directions. It follows from the fact that solutions are isotropic. The condition \(H_e\cap\mathcal G^r_{\pi,e}=0\) is the polarised transversality condition ensuring that the resulting integral manifold is a graph of a section rather than a vertical integral manifold. Thus, \(\operatorname{Pol}_{\mathcal E}(I_e)\) is the intrinsic set of potentially possible solution directions extending the initial datum \(I_e\). It is important to determine the intersection
\(
I_e^{\perp_{\mathcal C}}\cap\mathcal G^r_{\mathcal E,e}.
\)
For instance, if $I_e^{\perp_{\mathcal C}}\subset \mathcal G^r_{\mathcal E,e}$,  then \(\operatorname{Pol}_{\mathcal E}(I_e)\) is empty. 

\begin{proposition}\label{prop:polarised_initial_data_solutions}
Let \(N\subset\mathcal E\) be an isotropic initial submanifold with \(\T N\subset\mathcal C_{\mathcal E}\). Suppose that there exists a smooth rank-\(n\) distribution \(\mathcal H\subset\mathcal C_{\mathcal E}\) on an open neighbourhood $U\subset \mathcal{E}$ of \(N\) such that \(\mathcal H_e\in\operatorname{Pol}_{\mathcal E}(\T_eN)\) for every \(e\in N\) and \([\Gamma(\mathcal H),\Gamma(\mathcal H)]\subset\Gamma(\mathcal H)\) on $U$. Then, the integral leaves of \(\mathcal H\) through \(N\), whenever defined, are local solutions of \(\mathcal E\).
\end{proposition}

\begin{proof}
By assumption, \(\mathcal H\subset\mathcal C_{\mathcal E}\), \(\operatorname{rank}\mathcal H=n\), and \(\mathcal H_e\cap\mathcal G^r_{\pi,e}=0\) for every \(e\in N\). Since transversality to \(\mathcal G^r_\pi\) is an open condition, after shrinking \(U\) we may assume that \(\mathcal H\cap\mathcal G^r_\pi=0\) on \(U\). The involutivity assumption gives integral leaves by Frobenius' theorem. Proposition~\ref{prop:foliations_by_solutions} then implies that these leaves are holonomic solution submanifolds. The inclusion \(\T_eN\subset\mathcal H_e\) means that the leaves extend the prescribed initial datum.
\end{proof}

This construction separates the algebraic Cauchy problem from the analytic existence problem. The family \(\operatorname{Pol}_{\mathcal E}(\T N)\) describes, pointwise, which \(n\)-planes can serve as tangent spaces to solutions extending the initial data. Finding an involutive section of this family is the geometric part of constructing a family of solutions. The \(k\)-contact orthogonal \(\T N^{\perp_{\mathcal C}}\) detects the characteristic obstruction, while the polarisation \(\mathcal G^r_\pi\) records the symbol directions that prevent an integral manifold from being a genuine graph.

Let us turn to another potential application of $k$-contact geometry to describing systems of PDEs. 
Let \((M,\mathcal D=\ker\bm\eta,\mathcal P)\) be a polarised co-oriented \(k\)-contact manifold. For a \(k\)-contact vector field \(X\), its Hamiltonian is \(\bm h_X=-\iota_X\bm\eta\). For a family \(\mathscr X=(X_1,\ldots,X_s)\) of \(k\)-contact vector fields, we set
\[
\mathcal E_{\mathscr X}:=\{p\in M:\bm h_{X_1}(p)=\cdots=\bm h_{X_s}(p)=0\}.
\]
Equivalently, \(p\in\mathcal E_{\mathscr X}\) if and only if \((X_a)_p\in\mathcal D_p\) for every \(a=1,\ldots,s\). On every regular open subset, \(\mathcal E_{\mathscr X}\) inherits the distributions \(\mathcal D_{\mathcal E_{\mathscr X}}:=\T\mathcal E_{\mathscr X}\cap\mathcal D\) and \(\mathcal P_{\mathcal E_{\mathscr X}}:=\T\mathcal E_{\mathscr X}\cap\mathcal P\), whenever these intersections have constant rank.

In the canonical jet model \(M=J^r\pi\), with \(\mathcal D=\mathcal C^r_\pi\) and \(\mathcal P=\mathcal G^r_\pi=\ker \T\pi_{r,r-1}\), the regular Hamiltonian locus \(\mathcal E_{\mathscr X}\) inherits \(\mathcal C_{\mathcal E_{\mathscr X}}:=\T\mathcal E_{\mathscr X}\cap\mathcal C^r_\pi\) and \(\mathcal P_{\mathcal E_{\mathscr X}}:=\T\mathcal E_{\mathscr X}\cap\mathcal G^r_\pi\).  The defining equations of \(\mathcal E_{\mathscr X}\) express that the vector fields \(X_a\) take values in  $\mathcal{C}^r_\pi$,  not the vanishing of the vector fields themselves.
We next discuss symmetries. The polarised \(k\)-contact viewpoint distinguishes three different notions: vector fields preserving the equation as a submanifold, Cartan symmetries preserving the differential structure, and polarised Cartan symmetries preserving the chosen jet presentation.

\begin{definition}\label{def:lie_symmetry_pde_kcontact}
A vector field \(X\in\mathfrak X(J^r\pi)\) is a \emph{Lie symmetry} of \(\mathcal E\) if it is a \(k\)-contact vector field tangent to \(\mathcal E\). It is a \emph{polarised Lie symmetry} if, in addition, it is a polarised \(k\)-contact vector field.
\end{definition}

If \(\mathcal E\cap U=F^{-1}(0)\), the tangency condition is equivalent to
\(
X(F^a)|_{\mathcal E}=0\) for \(a=1,\ldots,s,
\)
or, locally, to \(X(F^a)=\sum_{b=1}^s\lambda^a_bF^b\) for \(a=1,\ldots,s\). Lie symmetries map integral submanifolds of the Cartan distribution to integral submanifolds of the Cartan distribution. Polarised Lie symmetries preserve, moreover, the vertical polarisation and hence the class of polarised Legendrian submanifolds.

\begin{proposition}\label{prop:lie_symmetries_subalgebra_pde}
The space
\[
\mathfrak{sym}_{\mathcal C}(\mathcal E):=\{X\in\mathfrak X(J^r\pi):[X,\Gamma(\mathcal C^r_\pi)]\subset\Gamma(\mathcal C^r_\pi),\ X_p\in \T_p\mathcal E\ \text{for all }p\in\mathcal E\}
\]
is a Lie subalgebra of the Lie algebra of Cartan vector fields on \(J^r\pi\). The polarised Lie symmetries form a Lie subalgebra of \(\mathfrak{sym}_{\mathcal C}(\mathcal E)\).
\end{proposition}

\begin{proof}
If \(X\) and \(Y\) preserve \(\mathcal C^r_\pi\), then Jacobi's identity implies that \([X,Y]\) preserves \(\mathcal C^r_\pi\). If both preserve \(\mathcal G^r_\pi\), the same argument shows that \([X,Y]\) preserves \(\mathcal G^r_\pi\). Finally, if \(X\) and \(Y\) are tangent to \(\mathcal E\), then their restrictions are vector fields on \(\mathcal E\), and the restriction of \([X,Y]\) to \(\mathcal E\) is their intrinsic bracket. Hence \([X,Y]\) is again tangent to \(\mathcal E\), so it is again a Cartan Lie symmetry, and it is polarised when \(X\) and \(Y\) are.
\end{proof}

\begin{proposition}\label{prop:hamiltonian_lie_symmetry_pde}
Let \(\mathcal E\subset J^r\pi\) be a regular system of PDEs, let \(X\in \mathfrak{X}(J^r\pi)\) be a Cartan symmetry, and consider an adapted chart on \(U\subset J^r\pi\). Then, \(X\) is a Lie symmetry of \(\mathcal E\cap U\) if and only if \(X\) is tangent to \(\mathcal E\cap U\). Equivalently, if \(\mathcal E\cap U=F^{-1}(0)\), then \(X(F^a)|_{\mathcal E}=0\) for every \(a=1,\ldots,s\).
\end{proposition}

\begin{proof}
This is simply the tangency criterion for a vector field to be tangent to a regular embedded submanifold, together with the assumption that \(X\) already preserves the Cartan distribution.
\end{proof}

\begin{theorem}\label{thm:hamiltonian_invariant_subsystems}
Let \(\mathfrak s\subset\mathfrak{sym}_{\mathcal C}(\mathcal E)\) be a finite-dimensional Lie algebra of Cartan Lie symmetries of \(\mathcal E\). Then, the common zero locus
\[
\mathcal Z_{\mathfrak s}:=\{p\in J^r\pi:X_p\in (\mathcal{C}^r_\pi)_p \text{ for every }X\in\mathfrak s\}
\]
 is invariant under \(\mathfrak s\). If  \(\mathcal E\cap\mathcal Z_{\mathfrak s}\) is a regular submanifold, then \(\mathcal E\cap\mathcal Z_{\mathfrak s}\) is an invariant differential subsystem of \(\mathcal E\).
\end{theorem}

\begin{proof}  
Every point \(p\in J^r\pi\) admits an open neighbourhood \(U\ni p\) with an adapted chart and associated \(\bm\eta^r_\pi\). Hence, locally, $\mathcal{Z}_\mathfrak{s}$ can be written as 
\[
\mathcal{Z}_\mathfrak{s} = \{p \in U : \iota_{X}\bm\eta^r_\pi(p)=0   \text{ for every } X \in \mathfrak{s}\}.
\]
Meanwhile, let \(X,Y\in\mathfrak s\). Since \(Y\) is \(k\)-contact, locally \(\mathcal L_Y\bm\eta^r_\pi=A_Y\bm\eta^r_\pi\) for a matrix-valued function \(A_Y\). Hence,
\[
Y(\bm h_X)=-(\mathcal L_Y\bm\eta^r_\pi)(X)-\bm\eta^r_\pi([Y,X])=A_Y\bm h_X+\bm h_{[Y,X]}.
\]
Since \(\mathfrak s\) is a Lie algebra, \([Y,X]\in\mathfrak s\), and hence \(Y(\bm h_X)\) is zero on $\mathcal{Z}_\mathfrak{s}$. This means that the integral curves of $Y$ with a point in $\mathcal{Z}_\mathfrak{s}$ remain in $\mathcal{Z}_\mathfrak{s}$. The invariance of the common zero locus follows. The invariance of the intersection with \(\mathcal E\) is a consequence of the invariance of \(\mathcal Z_{\mathfrak s}\) and the fact that \(\mathcal E\) is invariant under \(\mathfrak s\).
\end{proof}

\begin{definition}\label{def:contact_transformations_pdes}
Let \(\mathcal E\subset J^r\pi\) and \(\mathcal E'\subset J^s\pi'\) be regular PDEs. A \emph{transcontact map} from \(\mathcal E\) to \(\mathcal E'\) is a transcontact morphism \(\Phi:J^r\pi\to J^s\pi'\) such that \(\Phi(\mathcal E)\subset\mathcal E'\). It is \emph{polarised} if, in addition, \(\T\Phi(\mathcal G^r_\pi)\subset \mathcal G^s_{\pi'}\).
\end{definition}

A Legendrian-polarised transcontact map sends the infinitesimal polarised Legendrian directions of the source equation to those of the target equation and therefore preserves holonomicity on regular images. A polarised transcontactomorphism is stronger: it preserves the chosen highest-order polarisation. A general transcontact transformation may preserve the Cartan distribution without preserving this polarisation. Such transformations may still be useful for relating solutions, but the image of a holonomic submanifold is interpreted as a solution only when it is transverse to the target projection. Classical hodograph prolongations preserve the highest-order polarisation but need not preserve the vertical bundle \(\ker T\pi_r\) associated with the chosen independent variables.

Let \(\mathcal E\subset J^r\pi\) be a regular equation. The intrinsic polarisation induced on \(\mathcal E\), when it has constant rank, is
\(
\mathcal G^r_{\mathcal E}:=\T\mathcal E\cap\mathcal G^r_\pi.
\)
\begin{definition}
  An {\it on-shell polarised symmetry} is a vector field \(Y\in\mathfrak X(\mathcal E)\) such that
\[
[Y,\Gamma(\mathcal C_{\mathcal E})]\subset\Gamma(\mathcal C_{\mathcal E}),\qquad [Y,\Gamma(\mathcal G^r_{\mathcal E})]\subset\Gamma(\mathcal G^r_{\mathcal E}).
\]  
\end{definition}

This condition is intrinsic to \(\mathcal E\). It should be distinguished from the weaker normal condition
\[
[X,\Gamma(\mathcal G^r_\pi)]|_{\mathcal E}\subset\Gamma(\mathcal G^r_\pi|_{\mathcal E})+\Gamma(\T\mathcal E),
\]
which only says that the image of \(\mathcal G^r_\pi|_{\mathcal E}\) in the normal bundle \(\T J^r\pi|_{\mathcal E}/\T\mathcal E\) is preserved. The latter condition is useful for proving invariance of the tangency locus
\[
\Sigma_{\mathcal P}(\mathcal E)=\{e\in\mathcal E\mid(\mathcal G^r_\pi)_e\subset \T_e\mathcal E\},
\]
but it is not the intrinsic definition of a polarised symmetry of \(\mathcal E\).

Let \(\mathcal E\subset J^r\pi\) be a regular equation of codimension \(q\), locally given by independent equations \(F^A=0\), \(A=1,\ldots,q\).  The restriction
\[
\sigma_{\mathcal E}:=\left.\dd F\right|_{\mathcal G^r_\pi|_{\mathcal E}}:\mathcal G^r_\pi|_{\mathcal E}\longrightarrow\mathbb R^q
\]
is the {\it vertical principal symbol} of the equation. The equation is locally in normal form with respect to \(q\) highest-order variables if and only if \(\sigma_{\mathcal E}\) has rank \(q\), equivalently
\[
\T\mathcal E+\mathcal G^r_\pi=\T J^r\pi|_{\mathcal E}.
\]
In that case, after choosing a rank \(q\) subbundle \(Q\subset\mathcal G^r_\pi|_{\mathcal E}\) on which \(\sigma_{\mathcal E}\) is an isomorphism, the implicit function theorem writes \(\mathcal E\) locally as a graph along the polarisation directions \(Q\). The induced polarisation on \(\mathcal E\) is
\[
\mathcal G^r_{\mathcal E}:=\T\mathcal E\cap\mathcal G^r_\pi=\ker\sigma_{\mathcal E},
\]
and therefore measures the residual highest-order freedom left by the equation.

\begin{proposition}\label{prop:normal_form_polarisation_symbol}
Let \(\mathcal E\subset J^r\pi\) be a regular equation of codimension \(q\), and let \(\mathcal G^r_\pi=\ker \T\pi_{r,r-1}\). Then, \(\mathcal E\) can be locally written in normal form with respect to \(q\) highest-order jet coordinates if and only if
\[
\T\mathcal E+\mathcal G^r_\pi=\T J^r\pi|_{\mathcal E}.
\]
Equivalently, for any local defining map \(F=(F^1,\ldots,F^q)\), the vertical symbol \(\sigma_{\mathcal E}:=\left.\dd F\right|_{\mathcal G^r_\pi|_{\mathcal E}}\) has rank \(q\). In this case, the residual polarisation of the equation is \(\mathcal G^r_{\mathcal E}=\T\mathcal E\cap\mathcal G^r_\pi=\ker\sigma_{\mathcal E}\). In particular, if \(q=\operatorname{rank}\mathcal G^r_\pi\), then \(\mathcal E\) is locally the graph of a section of \(\pi_{r,r-1}:J^r\pi\to J^{r-1}\pi\) if and only if \(\T\mathcal E\cap\mathcal G^r_\pi=0\).
\end{proposition}

We now indicate how the highest-order polarisation detects familiar features of concrete differential equations. The following examples are not intended to replace the classical analysis of these equations, but to show that several standard notions can be reformulated in terms of tangency, transversality and degeneracy of the polarisation.

\begin{example}
Let us consider the Gauss hypergeometric equation
\[
x(1-x)y''+\bigl(c-(a+b+1)x\bigr)y'-ab\,y=0.
\]
On \(J^2(\mathbb R,\mathbb R)\), with coordinates \((x,u_0,u_1,u_2)\), it defines the hypersurface \(\mathcal E=\{F=0\}\), where
\[
F=x(1-x)u_2+\bigl(c-(a+b+1)x\bigr)u_1-ab\,u_0.
\] For simplicity, let us assume that $F$ is such that $F^{-1}(0)$ is a regular hypersurface of corank one. 
The polarisation of \(J^2(\mathbb R,\mathbb R)\) is \(\mathcal G^2_\pi=\ker \T\pi_{2,1}\cap\mathcal C^2_\pi=\langle\partial/\partial u_2\rangle\). Hence,
\[
\T_p\mathcal E\cap\mathcal G^2_{\pi,p}\neq0\quad\Longleftrightarrow\quad \frac{\partial F}{\partial u_2}(p)=0\quad\Longleftrightarrow\quad x(1-x)=0.
\]
Thus the finite singular points of the hypergeometric equation are precisely those points over which the differential equation fails to be transverse to the polarisation. In other words, the classical singularities \(x=0\) and \(x=1\) are characterised intrinsically as the points where the equation \(\mathcal E\subset J^2(\mathbb R,\mathbb R)\) has non-trivial intersection with the polarised directions of the Cartan distribution. The singularity at infinity is obtained by compactifying the base to \(\mathbb P^1\) and applying the same criterion in the affine chart \(z=1/x\).
\end{example}

\begin{example}[Linear Fuchsian equations]
More generally, consider the scalar second-order equation \(\mathcal E=F^{-1}(0)\), where \(F:J^2\pi\to\mathbb R\) is given by
\[
F(x,u,u_{x},u_{xx})=a(x)u_{xx}+b(x)u_x+c(x)u.
\]
As before, we assume that $F^{-1}(0)$ is a regular hypersurface of corank one. 
One has
\(
\left.\dd F\right|_{\mathcal G^2_\pi}=a(x)\dd u_{xx}.
\)
Therefore, the equation is locally in normal form with respect to \(u_{xx}\) precisely where
\(
a(x)\neq0.
\)
At the points where \(a(x)=0\), the equation is no longer transverse to the highest-order polarisation. Thus, the usual degeneration of the leading coefficient is equivalently a polarised tangency condition.
\end{example}

\begin{example}
Let \(J^2(\mathbb R^n,\mathbb R)\) have second-order coordinates \(u_{ij}=u_{ji}\), and consider \(\det(u_{ij})=f(x^i,u,u_i)\). Assume that this defines a hypersurface \(\mathcal E\subset J^2(\mathbb R^n,\mathbb R)\). If \(U=(u_{ij})\), then \(F=\det U-f(x^i,u,u_i)\). The restriction of \(\dd F\) to the highest-order polarisation is the linear form
\[
\left.\dd F\right|_{\mathcal G^2_\pi}(\delta U)=\operatorname{tr}\bigl(\operatorname{adj}(U)\,\delta U\bigr).
\]
Therefore, the equation is transverse to the highest-order polarisation precisely where \(\operatorname{adj}(U)\neq0\). At points where \(\operatorname{adj}(U)=0\), the vertical principal symbol vanishes and the whole highest-order polarisation is tangent to the equation. Thus, the degeneracy of the determinant operator as a hypersurface equation is expressed as a polarised degeneracy of \(\mathcal E\subset J^2\pi\). In particular, this condition detects the locus where the linearisation of the highest-order part loses rank, not merely the locus \(\det U=0\).
\end{example}

\begin{example}
The Korteweg--de Vries equation
\[
u_t+6u\,u_x+u_{xxx}=0
\]
defines a hypersurface \(\mathcal E\subset J^3(\mathbb R^2,\mathbb R)\). The highest-order polarisation is \(\mathcal G^3_\pi=\ker \T\pi_{3,2}\). Inside \(\mathcal G^3_\pi\) there is the distinguished spatial third-order line \(Q_{xxx}=\langle\partial/\partial u_{xxx}\rangle\). If \(F=u_t+6u\,u_x+u_{xxx}\), then \(\left.\dd F\right|_{Q_{xxx}}=\dd u_{xxx}\). Therefore, the equation is everywhere transverse to \(Q_{xxx}\), and is globally in normal form with respect to \(u_{xxx}\). 
The  polarisation \(\mathcal G^3_{\mathcal E}=\T\mathcal E\cap\mathcal G^3_\pi\) contains the third-order directions not fixed by this scalar equation. The polarised picture therefore detects the normal form of KdV.
\end{example}

\section{Symmetry reduction between jet bundles}\label{Sec:SymReductionJets}

The recognition theorem allows one to formulate reductions in a way that is not tied to the original jet presentation. This is important because   useful reductions of PDEs do not preserve the Cartan distribution itself. This occurs for $\lambda$- and $\mu$-symmetries, telescopic vector fields and other ones \cite{MurielRomero2012}. These vector fields are particular types of what here call projectable $k$-contact vector fields which allow for the projection of a Cartan distribution onto a quotient manifold.
Then, our local characterisation theorem allows for the determination of cases where the final reduction of the PDE falls again into a PDE in a new jet manifold. Before explaining our reduction theory, it is relevant to review Theorem \ref{thm:local_split_jet_recognition} and the definition of polarisation of jet type in a manner that it has been not been analysed so far.

\begin{lemma}\label{lem:jet_type_Levi_nondegenerate}
Let \(M\) be a manifold and let \(\mathcal D\subset TM\) be a regular distribution. Assume that \(\mathcal D\) admits an involutive subbundle $\mathcal P$ and a local splitting on a neighbourhood $U$ of every point $p\in M$ of the form \(\mathcal D|_U=\mathcal P|_U\oplus H\), where   \(H\) is an involutive distribution. Let \(\mathcal D^{(0)}=\mathcal D\), suppose that \(\mathcal D^{(r)}=TM\) for $r>0$ and that, for a vector bundle \(W\to U\), there are vector bundle identifications
\[
\mathcal P|_U\simeq S^rH^*\otimes W,\qquad 
\mathcal D^{(s)}/\mathcal D^{(s-1)}|_U\simeq S^{r-s}H^*\otimes W,\qquad s=1,\ldots,r.
\]
Assume moreover that, under these identifications, the bracket action of \(H\) on the graded quotients coincides with Spencer contraction, namely
\[
H\otimes\mathcal P|_U\longrightarrow \mathcal D^{(1)}/\mathcal D|_U,\qquad 
X\otimes Y\longmapsto -[X,Y]\mod \mathcal D,
\]
is identified with \(H\otimes S^rH^*\otimes W\to S^{r-1}H^*\otimes W\), \(X\otimes\omega\otimes w\mapsto \iota_X\omega\otimes w\), and, for \(s=1,\ldots,r-1\),
\[
H\otimes\big(\mathcal D^{(s)}/\mathcal D^{(s-1)}\big)\longrightarrow 
\mathcal D^{(s+1)}/\mathcal D^{(s)}|_U,\qquad 
X\otimes[Y]\longmapsto -[X,Y]\mod \mathcal D^{(s)}|_U,
\]
is identified with \(H\otimes S^{r-s}H^*\otimes W\to S^{r-s-1}H^*\otimes W\), \(X\otimes\omega\otimes w\mapsto \iota_X\omega\otimes w\). Then, \(\mathcal D\) is maximally non-integrable. 

Assume additionally that  $\mathcal{V}_r$ is an involutive distribution on $U$ such that $H\oplus \mathcal{V}_r=\T U$ and $\mathcal{V}_s=\mathcal{D}^{(s)}|_U\cap \mathcal{V}_r$ are involutive and regular for $s=0,\ldots,r-1$ with $\mathcal{V}_0=\mathcal{P}$. Then, $\mathcal{D}$ admits on $U$ Reeb vector fields turning it into an $N^r_\pi$-contact structure, and $\mathcal{P}$ becomes  a polarisation of jet type of order $r$.
\end{lemma}

\begin{proof}
Let \(Z\in \Gamma(\mathcal{D})\) such that $[Z,Y]_x\in \mathcal{D}_x$ for every $Y\in \Gamma(\mathcal{D})$ at a certain $x\in U$. Write \(Z=Z_H+Z_P\), with \(Z_H\in \Gamma(H)\) and \(Z_P\in\Gamma(\mathcal P)\). Since both \(H\) and \(\mathcal P\) are involutive, one has $[Z,Y]=[Z_P,Y_H]+[Z_H,Y_P]$ modulo $\mathcal{D}$. 

Take now an arbitrary \(Y'\in\Gamma(\mathcal P)\). It follows that \([Z,Y']_x=[Z_H,Y']_x\in\mathcal D_x\).  Under the identification \(\mathcal P\simeq S^rH^*\otimes W\), this says that \(\iota_{Z_H}\omega=0\) for every \(\omega\in S^rH^*\otimes W\) and \((Z_H)_x=0\). Hence, \(Z_x=(Z_P)_x\). Now, take an arbitrary section \(X'\in\Gamma(H)\). Again, gives \([Z,X']_x\in\mathcal D_x\). By the Spencer-contraction identification, this means \(\iota_{(X')_x}\omega_{(Z_P)_x}=0\) for every \(X'\in H\), where \((Z_P)_x\) is regarded as an element of \(S^rH^*\otimes W\). A symmetric \(W\)-valued \(r\)-tensor whose contraction with every vector in \(H\) vanishes is zero and \(Z_x=(Z_P)_x=0\).
Therefore, \(Z_x=0\) and $\mathcal{D}$ is maximally non-integrable. 

Once $\mathcal{D}$ is shown to be maximally non-integrable, the converse part of the proof of Theorem \ref{thm:local_split_jet_recognition} can be applied to obtain a coordinate system so that  $\mathcal{D}$ takes the form of a Cartan distribution $C^r_\pi$ and $\mathcal{P}$ becomes $\mathcal{G}^r_\pi$. Then, there exist Reeb vector fields given by Lemma \ref{lem:adapted_Nr_contact}, and $\mathcal{D}$ becomes an $N^r_\pi$-contact distribution. In fact, as shown in Lemma \ref{lem:jet_symbol_dominant}, it follows that $\mathcal{P}$ becomes a polarisation. 
\end{proof}

The previous lemma shows that part of the assumptions of the distribution of jet type of order $r$ are indeed consequences of a smaller family of assumptions. As we were previously focused on studying $k$-contact manifolds, removing maximally non-integrability or other conditions that were natural was not useful. But we are now concerned with reducing $k$-contact distributions that give rise to distributions admitting exactly the conditions expressed in the previous lemma. More precisely, the conditions are formulated in terms of the ranks of the relevant distributions, their derived flag, the compatibility of the polarisation with this flag, and the corresponding Lie-bracket relations. The Reeb fields which appear in the local model of \(J^r\pi\) are therefore not an additional assumption in the inverse direction of the theorem. They are obtained only after the distribution has been recognised as being locally equivalent to a finite-order Cartan distribution. Thus, in the inverse part of the localisation theorem, one should regard the Reeb fields as a consequence of the jet-type recognition procedure. Once adapted jet coordinates have been constructed, the corresponding local \(N^r_\pi\)-contact form and its Reeb fields are recovered in those coordinates. They need not be prescribed beforehand, nor do they need to be inherited from some previously chosen Reeb fields. All this justifies defining the following notion.

\begin{definition} A {\it jet distribution} is a triple $(M,\mathcal{D},\mathcal{P})$ satisfying the conditions of Lemma \ref{lem:jet_type_Levi_nondegenerate}.
\end{definition}

Let \(G\) be a Lie group acting locally on an open subset \(U\subset J^r\pi\), and let \(\mathcal D_G:=\langle \xi_{J^r\pi}:\xi\in\mathfrak g\rangle\) be the distribution spanned by the fundamental vector fields of the action. We do not assume that the action preserves \(\mathcal C^r_\pi\). Instead, our reduction is mainly concerned with  \(\mathcal C^r_\pi+\mathcal D_G\) and \(\mathcal G^r_\pi+\mathcal D_G\), where \(\mathcal G^r_\pi=\ker\T\pi_{r,r-1}\). Thus, the Cartan and polarisation data are required to be projectable only modulo the group directions. More specifically, this work is concerned with the following distributions associated with Lie group actions.
\begin{definition}
A regular distribution \(\mathcal D_G\) on \(J^r\pi\) is called a {\it projective
\(N^r_\pi\)-contact reduction distribution} if it is integrable and
\[
        [\Gamma(\mathcal D_G),\Gamma(\mathcal C^r_\pi+\mathcal D_G)]
        \subset
        \Gamma(\mathcal C^r_\pi+\mathcal D_G).
\]
Additionally, \(\mathcal D_G\) is said to be {\it compatible with the polarisation} if
\[
        [\Gamma(\mathcal D_G),\Gamma(\mathcal G^r_\pi+\mathcal D_G)]
        \subset
        \Gamma(\mathcal G^r_\pi+\mathcal D_G).
\]
\end{definition}
Essentially, the above definition implies that $\mathcal{C}^r_\pi$ and $\mathcal{G}^r_\pi$ on some $U\subset J^r\pi$ to give rise to generalised distributions on some manifold $U/\mathcal{D}_G$ whose rank is not larger than the rank of the previous ones, respectively.

\begin{example}
    Let us briefly indicate  \(\lambda\)-symmetries for ordinary differential equations, namely $n=1$, give rise to a projective $N^r_\pi$-contact reduction distribution \cite{MurielRomero2012}. Let \(\mathcal E\subset J^r\pi\) be an ODE with $m=1$ and let \(X^{(r)}_\lambda\) be the \(\lambda\)-prolongation of a vector field \(X=\xi\partial/\partial x+ \varphi \partial/\partial u \in \mathfrak{X}(E)\) for \(\lambda\in C^\infty(J^1\pi)\). Although \(X^{(r)}_\lambda\) need not be an ordinary Cartan symmetry, one has  \cite[Eqs. (2.6) and (2.7)]{MurielRomero2012} and taking into account that our $D^r$ is here the total derivative truncated to order $r$ and not the full total derivative as in \cite{MurielRomero2012}, that
\begin{equation}\label{eq:Rellambda}
\begin{gathered}
\left[X^{(r)}_\lambda,D^r\right]+(D^r\xi+\lambda\xi)D^r-\lambda X^{(r)}_\lambda
=-(D^r+\lambda)(\varphi_r)\frac{\partial}{\partial u_r},\\
\left[X^{(r)}_\lambda,\frac{\partial}{\partial u_K}\right]
=-\sum_{|K|\leq |I|\leq r}
\frac{\partial\varphi_I}{\partial u_K}\frac{\partial}{\partial u_I},
\qquad 1\leq |K|\leq r.
\end{gathered}
\end{equation}
and 
the line distribution \(\mathcal L=\langle X_\lambda^{(r)}\rangle\) is a projective \(k\)-contact reduction distribution. Indeed, 
\[
        [\Gamma(\mathcal L),\Gamma(\mathcal C^r_\pi+\mathcal{L})]
        \subset
        \Gamma(\mathcal C^r_\pi+\mathcal L),\qquad [\Gamma(\mathcal L),\Gamma(\mathcal G^r_\pi+\mathcal{L})]
        \subset
        \Gamma(\mathcal G^r_\pi+\mathcal L).
\]

Moreover, it will be worth recalling that, if $D_x$ is the total derivative with respect to $x$ truncated to order $r$, then the $\lambda$-prolongation of $X$ is given by
\begin{equation}\label{eq:deZlambda}
X^{(r)}_\lambda
=
\xi D_x
+
\sum_{j=0}^{r}
(D_x+\lambda)^j
\bigl(\varphi-\xi u _x\bigr)
\frac{\partial}{\partial u _j},
\end{equation}
where $ \varphi -\xi u _x =
\iota_{X_\lambda}\pi_{r,1}^*\bm\eta^1_\pi$.
\end{example}

Let us turn to our general theory. Let \(N\subset U\) be a regular submanifold.  Think of \(N\) as an ambient constraint, for instance an invariance locus or any other general   constraint. Assume that  $\mathcal{D}_G$ is an $N^r_\pi$-contact reduction distribution compatible with the polarisation. Suppose that \(\bar N=q(N)\) is a smooth submanifold. Whenever the distributions 
\[
\overline{\mathcal C}_{\overline N}:=\T q\big((\T N+\mathcal D_G)\cap(\mathcal C^r_\pi+\mathcal D_G)\big),\qquad
\overline{\mathcal P}_{\overline N}:=\T q\big((\T N+\mathcal D_G)\cap(\mathcal G^r_\pi+\mathcal D_G)\big),
\]
have constant rank, the resulting triple \((\overline N,\overline{\mathcal C}_{\overline N},\overline{\mathcal P}_{\overline N})\) is the {\it reduced  polarised distributional structure}. If it is of jet type, Lemma \ref{lem:jet_type_Levi_nondegenerate} and the local  recognition theorem reconstructs \(\overline N\) locally as a finite-order jet bundle. A reduced equation is then obtained by taking a regular equation \(\mathcal E\subset U\), restricting it to \(\mathcal E_N:=\mathcal E\cap N\), and considering its quotient \(\overline{\mathcal E}:=q(\mathcal E_N)\subset\overline N\) to be a regular submanifold.

\begin{theorem} \label{thm:modular_reduction_jet_type}
Let \(G\) act locally on an open subset \(U\subset J^r\pi\), and let \(q:U\to\overline U:=U/G\) be a smooth local quotient. Let \(N\subset U\) be a regular submanifold such that \(\overline N:=q(N)\) is regular. Assume that the distributions
\[
(TN+\mathcal D_G)\cap(\mathcal C^r_\pi+\mathcal D_G),\qquad
(TN+\mathcal D_G)\cap(\mathcal G^r_\pi+\mathcal D_G)
\]
have constant rank along \(N\) and are projectable through \(q|_N\), so that
their images define regular distributions on \(\bar N=q(N)\). Assume that $\mathcal{D}_G$ is a projective $N^r_\pi$-contact reduction distribution that is compatible with the polarisation and $q(N)=\bar N$ is a manifold. 
Define
\[
\overline{\mathcal C}_{\overline N}:=\T q\big((\T N+\mathcal D_G)\cap(\mathcal C^r_\pi+\mathcal D_G)\big),\qquad
\overline{\mathcal P}_{\overline N}:=\T q\big((\T N+\mathcal D_G)\cap(\mathcal G^r_\pi+\mathcal D_G)\big).
\]
Suppose that  \((\overline N,\overline{\mathcal C}_{\overline N},\overline{\mathcal P}_{\overline N})\) is a jet distribution. Then,  there exists a fibred manifold \(\overline\pi:\overline F\to\overline Q\), an integer \(s\), and a   diffeomorphism \(\Phi:V\subset \overline N\to \Phi(V)\subset J^s\overline\pi\) such that \(\T\Phi(\overline{\mathcal C}_{\overline N})=\mathcal C^s_{\overline\pi}\) and \(\T\Phi(\overline{\mathcal P}_{\overline N})=\mathcal G^s_{\overline\pi}\) on $\Phi(V)$. Let now \(\mathcal E\subset U\) be a regular submanifold such that \(\mathcal E_N:=\mathcal E\cap N\) is regular and \(q|_{\mathcal E_N}\) has constant rank, with regular image \(\overline{\mathcal E}:=q(\mathcal E_N)\subset\overline N\).
Then, \(\Phi(\overline{\mathcal E}\cap V)\subset J^s\overline\pi\) is a new system of PDEs, with induced distributions
\[
\T\Phi\big(\T\overline{\mathcal E}\cap\overline{\mathcal C}_{\overline N}\big)
=
\T(\Phi(\overline{\mathcal E}))\cap\mathcal C^s_{\overline\pi},
\qquad
\T\Phi\big(\T\overline{\mathcal E}\cap\overline{\mathcal P}_{\overline N}\big)
=
\T(\Phi(\overline{\mathcal E}))\cap\mathcal G^s_{\overline\pi}.
\]
If, moreover, the clean-intersection identities
\[
\T\overline{\mathcal E}\cap\overline{\mathcal C}_{\overline N}
=
\T q\big((\T\mathcal E_N+\mathcal D_G)\cap(\mathcal C^r_\pi+\mathcal D_G)\big),
\qquad
\T\overline{\mathcal E}\cap\overline{\mathcal P}_{\overline N}
=
\T q\big((\T\mathcal E_N+\mathcal D_G)\cap(\mathcal G^r_\pi+\mathcal D_G)\big)
\]
hold, then the induced geometry of the reduced equation is exactly the quotient of the corresponding Cartan and polarisation data of \(\mathcal E_N\).
\end{theorem}

\begin{proof}
The projectability and constant-rank assumptions imply that \(\overline N\), \(\overline{\mathcal C}_{\overline N}\) and \(\overline{\mathcal P}_{\overline N}\) are well-defined regular objects on the quotient. Since \((\overline N,\overline{\mathcal C}_{\overline N},\overline{\mathcal P}_{\overline N})\) is assumed to be of jet type, Lemma \ref{lem:jet_type_Levi_nondegenerate} and Theorem~\ref{thm:local_split_jet_recognition} give a local fibred manifold \(\overline\pi:\overline F\to\overline Q\), an integer \(s\), and a local diffeomorphism \(\Phi:\overline N\to J^s\overline\pi\) identifying the reduced ambient Cartan-type distribution and polarisation with \(\mathcal C^s_{\overline\pi}\) and \(\mathcal G^s_{\overline\pi}\), respectively. If \(\overline{\mathcal E}\subset\overline N\) is regular, then \(\Phi(\overline{\mathcal E})\subset J^s\overline\pi\) is a regular system of PDEs. Since \(\Phi\) is a local diffeomorphism identifying the ambient distributions, the induced distributions on the reduced equation are precisely the intersections with $\mathcal{C}^s_{\bar{\pi}}$ and  $\mathcal{G}^s_{\bar{\pi}}$ in  \(J^s\overline\pi\). The final statement follows directly from the two clean-intersection identities.
\end{proof}

The case \(N=\mathcal E\) recovers the more economical reduction in which the quotient of the equation itself is recognised as a jet-type polarised \(k\)-contact object. The ambient version above is more flexible: \(N\) may encode an ansatz, an invariance condition, a Cartan locus, or another system of differential constraints, while the reduced equation \((\mathcal E\cap N)/G\) is then viewed as a submanifold of the reconstructed reduced jet bundle. Thus, different choices of \(N\) may lead to different reduced equations and, when \(N\) is less restrictive than the strict invariance locus, to larger classes of reduced solutions.

\begin{example}\label{ex:kdv_travelling_wave_modular}
Let us analyse a travelling-wave reduction of KdV using the ambient reduction theorem. Let \(\pi:\mathbb R^2_{t,x}\times\mathbb R_u\to\mathbb R^2_{t,x}\) and consider
\[
 \mathcal E_{\mathrm{KdV}}=\{u_t+6uu_x+u_{xxx}=0\}\subset J^3\pi .
\]
Let \(G=\mathbb R\) act by translations generated by \(X=\partial_t+c\partial_x\), and let \(\mathcal D_G=\langle X^{(3)}\rangle\), where \(X^{(3)}=\partial_t+c\partial_x\). The prolonged action preserves the Cartan distribution, but the travelling-wave reduction is more naturally obtained by first restricting to the Cartan locus
\[
 N:=\mathcal L_{\mathcal C^3_\pi}(X^{(3)})=\{p\in J^3\pi:X^{(3)}_p\in\mathcal C^3_{\pi,p}\}.
\]
In adapted coordinates \(N\) is given by
\[
 u_t+cu_x=0,\qquad u_{tt}+cu_{tx}=0,\qquad u_{tx}+cu_{xx}=0,
\]
\[
 u_{ttt}+cu_{ttx}=0,\qquad u_{ttx}+cu_{txx}=0,\qquad u_{txx}+cu_{xxx}=0.
\]
Thus, on \(N\), all \(t\)-derivatives are expressed in terms of pure \(x\)-derivatives. The quotient \(q_N:N\to\overline N:=N/G\) has invariant coordinate \(z=x-ct\), and the remaining coordinates are
\(
 u, p=u_x, q=u_{xx}, r=u_{xxx}.
\)
The reduced ambient distribution
\[
 \overline{\mathcal C}_{\overline N}:=\T q_N\big((\T N+\mathcal D_G)\cap(\mathcal C^3_\pi+\mathcal D_G)\big)
\]
is spanned by
\[
 D=\frac{\partial}{\partial z}+p\frac{\partial}{\partial u}+q\frac{\partial}{\partial p}+r\frac{\partial}{\partial q},
 \qquad
 R=\frac{\partial}{\partial r},
\]
and the reduced highest-order polarisation is \(\overline{\mathcal P}_{\overline N}=\langle R\rangle\). To verify the jet-type hypotheses, take \(H=\langle D\rangle\) and \(\overline{\mathcal V}_3=\ker\dd z=\langle\partial_u,\partial_p,\partial_q,\partial_r\rangle\). Then \(\T\overline N=H\oplus\overline{\mathcal V}_3\), both summands are integrable, and
\(
 \overline{\mathcal V}_s=\overline{\mathcal C}_{\overline N}^{(s)}\cap\overline{\mathcal V}_3\), for \(s=0,1,2,
\)
are
\(
 \overline{\mathcal V}_0=\langle\partial_r\rangle, 
 \overline{\mathcal V}_1=\langle\partial_q,\partial_r\rangle, 
 \overline{\mathcal V}_2=\langle\partial_p,\partial_q,\partial_r\rangle.
\) 
Indeed, \([D,\partial_r]=-\partial_q\), \([D,\partial_q]=-\partial_p\), and \([D,\partial_p]=-\partial_u\). Hence, \(\overline{\mathcal C}_{\overline N}^{(s)}=H\oplus\overline{\mathcal V}_s\) for \(s=0,1,2\), while \(\overline{\mathcal C}_{\overline N}^{(3)}=\T\overline N=H\oplus\overline{\mathcal V}_3\). Thus, \((\overline N,\overline{\mathcal C}_{\overline N},\overline{\mathcal P}_{\overline N})\) is of jet type and order \(3\). By Theorem~\ref{thm:modular_reduction_jet_type}, the reduced ambient space is locally equivalent to \(J^3\rho\) with its canonical Cartan distribution and highest-order polarisation, where \(\rho:\mathbb R_z\times\mathbb R_u\to\mathbb R_z\). In this identification, the reconstructed jet coordinates are \(p=u_z\), \(q=u_{zz}\), and \(r=u_{zzz}\).

The reduced equation is now obtained as a submanifold of this ambient reduced jet space, namely
\[
 \overline{\mathcal E}_{\mathrm{KdV}}:=q_N(\mathcal E_{\mathrm{KdV}}\cap N)\subset\overline N\simeq J^3\rho .
\]
Since on \(N\) one has \(u_t=-cu_x=-cp\), the KdV equation becomes
\[
 u_t+6uu_x+u_{xxx}=-cp+6up+r.
\]
Therefore,
\[
 \overline{\mathcal E}_{\mathrm{KdV}}=\{r+(6u-c)p=0\}\subset J^3\rho,
\]
or, in the reconstructed jet coordinates,
\[
 u_{zzz}+(6u-c)u_z=0.
\]
Thus, the travelling-wave reduction is not obtained by reconstructing a jet space from the quotient of the equation itself. Rather, the ambient quotient \(N/G\) first carries a reduced Cartan-type distribution together with a highest-order polarisation of jet type, hence is recognised as \(J^3\rho\); the reduced KdV equation is then the hypersurface \(\overline{\mathcal E}_{\mathrm{KdV}}\subset J^3\rho\).
\end{example}

\begin{remark}
The choice of the ambient submanifold \(N\) is not unique. In the travelling-wave example above, \(N=\mathcal L_{\mathcal C^3_\pi}(X^{(3)})\) encodes the strict invariance condition \(u_t+cu_x=0\) and its differential consequences, hence leads to the standard travelling-wave reduction. More generally, one may choose a larger or different submanifold \(N\subset J^r\pi\) satisfying the constant-rank and projectability assumptions above, provided the reduced polarised distribution \((\overline N,\overline{\mathcal C}_{\overline N},\overline{\mathcal P}_{\overline N})\) is of jet type. The image \(\overline N=q(N)\) is then recognised as an ambient jet space, and \((\mathcal E\cap N)/G\) becomes a reduced differential equation inside it. Thus, different choices of \(N\) correspond to different reduction ansatzes or differential constraints, and may lead to larger classes of reduced solutions than the strictly invariant ones.
\end{remark}

Thus, the hypotheses of the projective \(k\)-contact reduction theorem are satisfied. The quotient \(U/\mathcal L\) carries the reduced distribution, and when the induced polarisation satisfies the jet-type hypotheses of the recognition theorem, this quotient is locally identified with the jet space of the reduced ODE. In this way, the usual reduction by differential invariants of a \(\lambda\)-symmetry is recovered as a particular case of \(k\)-contact reduction.

\begin{theorem}\label{thm:Muriel_lambda_reduction_jet_type}
Let \(U\subset J^r\pi\), with $r>1$ and $\dim Q=1$ and rank one, be an open subset endowed with the scalar jet-type configuration
\[
\mathcal C^r_\pi=H\oplus \mathcal G^r_\pi,\qquad
H=\langle D_x\rangle,\qquad
 \mathcal G^r_\pi=\ker T\pi_{r,r-1}.
\]
Let \(\mathcal V_s=\ker T\pi_{r,r-s-1}\), with \(s=0,\ldots,r-1\), so that \(\mathcal V_0=\mathcal G^r_\pi\) and \((\mathcal C^r_\pi)^{(s)}=H\oplus\mathcal V_s\), for \(s=0,\ldots,r-1\). Let \(Z_\lambda\) be a \(\lambda\)-prolonged vector field on \(U\) for $\lambda\in C^\infty(\pi_{r,1}(U))$, set \(\mathcal L=\langle Z_\lambda\rangle\), and assume that
\(
\iota_{Z_\lambda}\pi_{r,1}^*\bm\eta^1_\pi\) does not vanish on $U$. 
After possibly shrinking \(U\), the quotient \(q:U\to\bar U:=U/\mathcal L\) is smooth. Define
\[
\bar{\mathcal D}:=Tq(\mathcal C^r_\pi),\qquad
\bar{\mathcal P}:=Tq(\mathcal G^r_\pi),
\]
and, for \(s=0,\ldots,r-1\), we set 
\(
\bar{\mathcal V}_s:=Tq(\mathcal V_s)
\) with $\bar{\mathcal{V}}_{-1}=0$.
Then, \((\bar U,\bar{\mathcal D},\bar{\mathcal P})\) is of jet type and order \(r-1\). Hence, every point of \(\bar U\) has a neighbourhood locally polarised \(k\)-contactomorphic to a scalar jet bundle \(J^{r-1}\bar\pi\). If \(\mathcal E\subset U\) is a regular equation invariant under \(Z_\lambda\), then, after possibly shrinking \(U\) further,  \(q(\mathcal E)\subset\bar U\simeq J^{r-1}\bar\pi\) is a lower-order scalar differential equation.
\end{theorem}

\begin{proof}
By \eqref{eq:deZlambda} and the assumption on \(\iota_{Z_\lambda}\pi_{r,1}^*\bm\eta^1_\pi\), the vector field \(Z_\lambda\) does not vanish on \(U\). Hence, \(\mathcal L\) is regular and, after possibly shrinking \(U\), it admits a smooth local quotient \(q:U\to\bar U\).

Set \( Q:=\iota_{Z_\lambda}\pi_{r,1}^*\eta^1_\pi\). Expression \eqref{eq:deZlambda} shows that the component of \(Z_\lambda\) transverse to \((\mathcal C^r_\pi)^{(r-1)}=H\oplus\mathcal V_{r-1}\) is \(Q\partial_u\). Since \(Q\neq0\), one has
\(
\mathcal L\cap(\mathcal C^r_\pi)^{(r-1)}=0.
\)
In particular, \(\mathcal L\cap\mathcal C^r_\pi=0\), so \(\bar{\mathcal D}\) has rank two and \(\bar{\mathcal P}\) has rank one. The first identity in \eqref{eq:Rellambda} gives
\(
[Z_\lambda,\Gamma(\mathcal C^r_\pi+\mathcal L)]
\subset\Gamma(\mathcal C^r_\pi+\mathcal L).
\)
Moreover, the triangular dependence of the coefficients of \(Z_\lambda\), together with the second identity in \eqref{eq:Rellambda}, gives
\[
[Z_\lambda,\Gamma(\mathcal V_s+\mathcal L)]
\subset\Gamma(\mathcal V_s+\mathcal L),
\qquad s=0,\ldots,r-1.
\]
Therefore, \(\mathcal C^r_\pi+\mathcal L\) and \(\mathcal V_s+\mathcal L\) are projectable through \(q\). Since \((\mathcal C^r_\pi)^{(s)}=H\oplus\mathcal V_s\), the corrected commutator identity also gives
\[
[Z_\lambda,\Gamma((\mathcal C^r_\pi)^{(s)}+\mathcal L)]
\subset\Gamma((\mathcal C^r_\pi)^{(s)}+\mathcal L).
\]
Consequently, all the distributions used below are well defined on \(\bar U\).

Since brackets of projectable vector fields project to brackets, induction on \(s\) yields
\begin{equation}\label{eq:lambda_reduced_derived_flag}
\bar{\mathcal D}^{(s)}
=\T q\bigl((\mathcal C^r_\pi)^{(s)}+\mathcal L\bigr),
\qquad s=0,\ldots,r-1.
\end{equation}
The restriction of \(\T q\) to \((\mathcal C^r_\pi)^{(r-1)}\) is injective. Hence,
\[
\operatorname{rank}\bar{\mathcal D}^{(s)}=s+2,
\qquad s=0,\ldots,r-1.
\]
Since \(\dim\bar U=r+1\), equation \eqref{eq:lambda_reduced_derived_flag} gives \(\bar{\mathcal D}^{(r-1)}=\T\bar U\).

The distributions \(\bar{\mathcal V}_s\) are involutive because \(\mathcal V_s+\mathcal L\) are involutive. Moreover, \(\operatorname{rank}\bar{\mathcal V}_s=s+1\) and \(\bar{\mathcal V}_0=\bar{\mathcal P}\). Since \(\T q\) is injective on \((\mathcal C^r_\pi)^{(r-1)}\), the identities \(\mathcal V_s=(\mathcal C^r_\pi)^{(s)}\cap\mathcal V_{r-1}\) descend to
\begin{equation}\label{eq:lambda_reduced_vertical_flag}
\bar{\mathcal V}_s
=\bar{\mathcal D}^{(s)}\cap\bar{\mathcal V}_{r-1},
\qquad s=0,\ldots,r-2.
\end{equation}

Choose locally a line distribution \(\bar H\subset\bar{\mathcal D}\) complementary to \(\bar{\mathcal P}\). Since \(\bar H\) and \(\bar{\mathcal P}\) have rank one, they are involutive and isotropic. Thus,
\(
\bar{\mathcal D}=\bar H\oplus\bar{\mathcal P}.
\)
Equations \eqref{eq:lambda_reduced_derived_flag} and \eqref{eq:lambda_reduced_vertical_flag} give
\[
\bar{\mathcal D}\cap\bar{\mathcal V}_{r-1}
=\bar{\mathcal P},\qquad
\T\bar U=\bar H\oplus\bar{\mathcal V}_{r-1},
\]
and
\[
\bar{\mathcal D}^{(s)}
=\bar H\oplus\bar{\mathcal V}_s,
\qquad s=0,\ldots,r-1.
\]

It remains to verify the Spencer-type symbol condition. Set \(\bar W:=\bar{\mathcal V}_{r-1}/\bar{\mathcal V}_{r-2}\). All the bundles \(\bar H\), \(\bar W\), \(\bar{\mathcal P}\), and \(\bar{\mathcal V}_s/\bar{\mathcal V}_{s-1}\) have rank one. Moreover, the bracket maps
\[
\bar H\otimes
\bar{\mathcal V}_s/\bar{\mathcal V}_{s-1}
\longrightarrow
\bar{\mathcal V}_{s+1}/\bar{\mathcal V}_s,
\qquad
\bar X\otimes[\bar Y]\longmapsto-[\bar X,\bar Y],
\]
are nonzero for \(s=0,\ldots,r-2\), since otherwise the derived flag would not increase from rank \(s+2\) to rank \(s+3\). Hence, these maps are isomorphisms. Selecting the identifications recursively, one obtains
\[
\bar{\mathcal P}\simeq S^{r-1}\bar H^*\otimes\bar W,\qquad
\bar{\mathcal V}_s/\bar{\mathcal V}_{s-1}
\simeq S^{r-1-s}\bar H^*\otimes\bar W,
\qquad s=1,\ldots,r-1,
\]
and the bracket maps become the one-dimensional Spencer contractions. Therefore, \((\bar U,\bar{\mathcal D},\bar{\mathcal P})\) is of jet type and order \(r-1\) according to Lemma~\ref{lem:jet_type_Levi_nondegenerate}.

The local recognition theorem identifies \((\bar U,\bar{\mathcal D},\bar{\mathcal P})\) locally with \((J^{r-1}\bar\pi,\mathcal C^{r-1}_{\bar\pi},\mathcal G^{r-1}_{\bar\pi})\). Finally, if \(\mathcal E\subset U\) is invariant under \(Z_\lambda\), then \(\mathcal L|_{\mathcal E}\subset\T\mathcal E\). Hence, after shrinking \(U\) if necessary, \(q(\mathcal E)\) is a regular submanifold of \(\bar U\), which becomes a lower-order scalar differential equation under the local jet identification.
\end{proof}

\begin{remark}
The quotient coordinates in Theorem~\ref{thm:Muriel_lambda_reduction_jet_type} recover the usual differential invariants. Write \(D_x=D^r\). Choose invariants \(\zeta(x,u)\) and \(\eta(x,u,u_1)\) of \(Z\) and \(Z_\lambda^{(1)}\), respectively, with \(D_x\zeta\neq0\) and \(\eta_{u_1}\neq0\). Set \(I_0=\eta\) and \(I_{j+1}=D_xI_j/D_x\zeta\), for \(j=0,\ldots,r-2\). By \eqref{eq:Rellambda}, these functions are invariant under \(Z_\lambda^{(r)}\); the highest-order defect annihilates every function to which the recurrence is applied. Their triangular dependence on the jet variables gives local quotient coordinates defined by \(\bar x\circ q=\zeta\) and \(\bar u_j\circ q=I_j\). In these coordinates,
\[
\bar{\mathcal D}=\langle\bar D,\partial_{\bar u_{r-1}}\rangle,\qquad
\bar{\mathcal P}=\langle\partial_{\bar u_{r-1}}\rangle,\qquad
\bar D=\partial_{\bar x}+\sum_{j=0}^{r-2}\bar u_{j+1}\partial_{\bar u_j}.
\]
Thus, the finite invariant recurrence recovers the reduced jet coordinates, although the original horizontal line need not descend.
\end{remark}

 \section{Possible further directions}
This section illustrates the scope of the formalism through Bäcklund transformations, Lax pairs, hodograph maps and soliton correspondences. These examples show that polarised \(k\)-contact geometry naturally treats auxiliary systems and transformations beyond a fixed jet presentation.
\subsection{Lax pairs, coverings and B\"acklund correspondences}
\label{Sec:BacklundLax}

One of the natural advantages of the polarised \(k\)-contact viewpoint is that it does not force one to remain inside a fixed jet bundle. This is useful for structures that are traditionally described by auxiliary systems, coverings, zero-curvature representations and B\"acklund transformations. The aim of this subsection is not to claim that such objects are automatically Hamiltonian in the \(k\)-contact sense. Rather, we explain how their compatibility is encoded by Cartan geometry on enlarged jet spaces.

Let \(\pi:E\to Q\) be a fibred manifold with \(\dim Q=2\), with local coordinates \((x,t)\) on \(Q\), and let \(\mathcal E\subset J^r\pi\) be a differential equation. Suppose that \(\mathcal E\) admits an auxiliary linear problem depending on a spectral parameter \(\lambda\),
\[
 D_x\Psi=U_\lambda(j^s u)\Psi,\qquad
 D_t\Psi=V_\lambda(j^s u)\Psi,
\]
where \(D_x,D_t\) are the total derivatives and \(U_\lambda,V_\lambda\) take values in a matrix Lie algebra. Equivalently, on the extended space with auxiliary variable \(\Psi\), one considers the lifted total derivatives
\[
 \widetilde D_x^\lambda
 =
 D_x+\sum_a (U_\lambda\Psi)^a\frac{\partial}{\partial\Psi^a},
 \qquad
 \widetilde D_t^\lambda
 =
 D_t+\sum_a (V_\lambda\Psi)^a\frac{\partial}{\partial\Psi^a}.
\]
Their curvature is the vertical defect
\[
 [\widetilde D_x^\lambda,\widetilde D_t^\lambda]
 =
 \sum_a
 \bigl((D_xV_\lambda-D_tU_\lambda+[V_\lambda,U_\lambda])\Psi\bigr)^a
 \frac{\partial}{\partial\Psi^a}.
\]
Thus, up to the sign convention for the Lie bracket, the flatness condition is the usual zero-curvature equation
\[
 D_tU_\lambda-D_xV_\lambda+[U_\lambda,V_\lambda]=0.
\]
In the standard integrable examples this condition is equivalent to \(\mathcal E\), or contains \(\mathcal E\) as a distinguished differential subsystem. This is the traditional zero-curvature interpretation of a Lax pair; see, for instance, \cite{Marvan_2008,KrishnaswamiVishnu_2020}. The compatibility of the auxiliary linear system is understood in the usual sense, namely for arbitrary auxiliary values of \(\Psi\), or equivalently for a fundamental matrix solution.

The \(k\)-contact interpretation is obtained by viewing the auxiliary linear problem as a genuine differential equation on an enlarged jet space. Let
\(
 \rho:E\times W\to Q
\)
be the fibration whose sections are pairs \((u,\Psi)\), where \(W\) is the vector space in which the auxiliary wave function takes values. The equations
\[
 \Psi_x=U_\lambda(j^s u)\Psi,\qquad
 \Psi_t=V_\lambda(j^s u)\Psi
\]
define an auxiliary differential equation on 
\(
 \mathcal L_\lambda\subset J^{s+1}\rho
\)
after taking a sufficiently high finite jet order. Its induced Cartan distribution is \(\mathcal C_{\mathcal L_\lambda}=T\mathcal L_\lambda\cap\mathcal C^{s+1}_\rho\), wherever this intersection has constant rank. The first compatibility condition of \(\mathcal L_\lambda\) is obtained by requiring the Cartan plane of a higher jet to be tangent to \(\mathcal L_\lambda\). In coordinates, this gives
\[
 D_t(\Psi_x-U_\lambda\Psi)=0,\qquad
 D_x(\Psi_t-V_\lambda\Psi)=0.
\]
Modulo the defining equations of \(\mathcal L_\lambda\), the mixed compatibility condition becomes
\[
 (D_tU_\lambda-D_xV_\lambda+[U_\lambda,V_\lambda])\Psi=0.
\]
Hence, the zero-curvature equation is precisely the mixed Cartan compatibility condition of the auxiliary system.

This is the point at which the polarised \(k\)-contact framework enters. Since finite-order jet bundles carry canonical \(N^r_\rho\)-contact distributions, the auxiliary system \(\mathcal L_\lambda\) is a submanifold of a jet-type \(k\)-contact manifold. Its compatibility is not imposed by an external calculation: it is the Cartan prolongation condition for \(\mathcal L_\lambda\). Thus, the Lax pair is interpreted as a Cartan-compatible auxiliary system. On holonomic solutions satisfying the zero-curvature condition, this yields the usual flat auxiliary connection.

\begin{proposition}\label{prop:lax_cartan_compatibility}
Let \(\mathcal L_\lambda\) be the auxiliary linear system defined by
\[
D_x\Psi=U_\lambda\Psi,\qquad D_t\Psi=V_\lambda\Psi .
\]
Then, its Cartan compatibility condition is the zero-curvature equation
\[
\Omega_\lambda:=D_tU_\lambda-D_xV_\lambda+[U_\lambda,V_\lambda]=0.
\]
Equivalently, the lifted total derivatives commute for every auxiliary value of \(\Psi\) if and only if \(\Omega_\lambda=0\).
\end{proposition}

\begin{proof}
The bracket computation gives
\(
 [\widetilde D_x^\lambda,\widetilde D_t^\lambda]
 =
 -\sum_a(\Omega_\lambda\Psi)^a\frac{\partial}{\partial\Psi^a},
\)
up to the chosen sign convention for the zero-curvature expression. The displayed bracket vanishes for every auxiliary value of \(\Psi\) if and only if \(\Omega_\lambda=0\). For a fundamental matrix, the compatibility equation \(\Omega_\lambda\Psi=0\) is likewise equivalent to \(\Omega_\lambda=0\). 
\end{proof}

\begin{remark}\label{rem:lax_not_hamiltonian}
The vertical curvature defect can be represented locally by a vertical vector field on the auxiliary variables. However, this vector field is not automatically a \(k\)-contact vector field. Therefore the contraction of such a defect field with an adapted \(k\)-contact form should not be called a \(k\)-contact Hamiltonian unless one has additionally proved that the defect field preserves the relevant Cartan \(k\)-contact distribution. In the present discussion we use only the Cartan-geometric statement: the Lax equation is the vanishing of the curvature defect of a lifted horizontal distribution, or equivalently the mixed compatibility condition of the auxiliary PDE.
\end{remark}

B\"acklund transformations can be described in the same Cartan-compatible language. Let
\[
 \mathcal E\subset J^r\pi,\qquad
 \widehat{\mathcal E}\subset J^{\widehat r}\widehat\pi
\]
be two differential equations over the same base \(Q\), or over bases identified locally. A finite-order B\"acklund correspondence is a submanifold
\[
 \mathcal B\subset J^p\pi\times_QJ^{\widehat p}\widehat\pi
\]
together with the two natural projections
\[
 p:J^p\pi\times_QJ^{\widehat p}\widehat\pi\to J^p\pi,\qquad
 \widehat p:J^p\pi\times_QJ^{\widehat p}\widehat\pi\to J^{\widehat p}\widehat\pi,
\]
such that, after the necessary prolongations, its projections land in \(\mathcal E\) and \(\widehat{\mathcal E}\). The ambient product carries the product Cartan distribution
\[
 \mathcal C^{p,\widehat p}
 =
 (Tp)^{-1}(\mathcal C^p_\pi)\cap
 (T\widehat p)^{-1}(\mathcal C^{\widehat p}_{\widehat\pi}),
\]
together with the corresponding product polarisation. The induced Cartan distribution on \(\mathcal B\) is
\[
 \mathcal C_{\mathcal B}
 =
 T\mathcal B\cap \mathcal C^{p,\widehat p}. 
\]
The compatibility of the B\"acklund system is then expressed by the Cartan prolongation of \(\mathcal B\): one requires that the Cartan planes of higher jets be tangent to the B\"acklund constraints. This is the same mechanism as for ordinary PDE prolongation, but now applied to a correspondence between two jet spaces. 

Thus, after passing to a sufficiently high prolongation, a B\"acklund transformation gives a span of polarised \(k\)-contact PDEs
\[
 (\mathcal E,\mathcal C_{\mathcal E})
 \longleftarrow
 (\mathcal B^{(\mathrm{comp})},\mathcal C_{\mathcal B^{(\mathrm{comp})}})
 \longrightarrow
 (\widehat{\mathcal E},\mathcal C_{\widehat{\mathcal E}}),
\]
where \(\mathcal B^{(\mathrm{comp})}\) denotes the Cartan compatibility locus of the B\"acklund constraints. A holonomic integral submanifold of the middle space projects to holonomic integral submanifolds of the two equations. In this sense, B\"acklund transformations are not merely relations between functions; they are Cartan-compatible differential correspondences between polarised \(k\)-contact PDEs.
 
The novelty of the polarised \(k\)-contact formulation is therefore not the zero-curvature equation itself, nor the classical definition of B\"acklund transformation. Those belong to the traditional theory of integrable systems and differential coverings. The contribution is the uniform finite-order interpretation: Lax pairs, coverings and B\"acklund transformations are treated as auxiliary PDEs or differential correspondences inside enlarged jet-type \(k\)-contact manifolds. Their compatibility is detected by Cartan prolongation and by the induced product Cartan geometry, not by introducing Hamiltonian loci without a genuine \(k\)-contact vector field.
 \subsection{Hodograph transformations as changes of polarised \texorpdfstring{\(k\)}{k}-contact presentation}
\label{Sec:Hodograph}

The purpose of this subsection is twofold. First, we recall the classical hodograph mechanism in jet-geometric terms: a hodograph transformation changes the choice of independent variables and is defined only on the open locus where the transformed submanifold is transverse to the new projection. Secondly, we explain what the polarised \(k\)-contact viewpoint adds: the hodograph prolongation preserves the intrinsic Cartan \(N^r_\pi\)-contact distribution, but it generally changes the jet presentation determined by the chosen fibration. In particular, it may preserve holonomic geometry without preserving   \(\ker\T\pi_r\). Classical accounts of hodograph transformations and their effect on partial derivatives can be found, for instance, in \cite[p. 355]{Ainsaar_1986} or \cite{RogersShadwick_1982,ClarksonFokasAblowitz_1989,EstevezPrada_2005_HodographCH}. The Cartan-distribution viewpoint on jet spaces and differential equations is standard in geometric PDE theory  \cite{Vitagliano_DGPDE,KruglikovLychagin_2007}.

Motivated by the classical hodograph transformations, we shall use the term
\emph{hodograph-type transformation} for the following jet-geometric
notion.

\begin{definition}
    Let \(\pi:E\to Q\) and \(\pi':E'\to Q'\) be fibred manifolds with \(\dim Q=\dim Q'=n\) and fibre rank \(m\). A local diffeomorphism \(\Phi_0:E\to E'\) will be called a \emph{hodograph-type transformation} from \(\pi\) to \(\pi'\) if it sends local graphs of sections of \(\pi\), whenever their images are transverse to \(\pi'\), into local graphs of sections of \(\pi'\). 
\end{definition}

More explicitly, if \(\phi:U\subset Q\to E\) is a local section and \(\Phi_0(\phi(U))\) is transverse to \(\pi':E'\to Q'\), then locally there exists a section \(\phi':U'\subset Q'\to E'\) such that \(\Phi_0(\phi(U))=\phi'(U')\).

In adapted coordinates, such a hodograph transformation takes the local form \((x^i,u^\alpha)\mapsto (X^i(x,u),U^\alpha(x,u))\). Along a local section \(u=\phi(x)\), the new independent variables are \(X^j(x,\phi(x))\). Hence, the relevant non-degeneracy condition is
\[
 \det A\neq0,\qquad A_i^j:=D_iX^j=\frac{\partial X^j}{\partial x^i}+\sum_{\alpha=1}^mu^\alpha_i\frac{\partial X^j}{\partial u^\alpha},\qquad i,j=1,\ldots,n.
\]
Thus, a hodograph-type transformation is naturally defined only on the  subset where the total Jacobian \(A=(A_i^j)\) is invertible. On this subset, the total derivatives with respect to the new independent variables are
\(
 D'_j=\sum_{i=1}^n(A^{-1})^i_jD_i.
\)
Consequently, the first transformed derivatives are
\[
 U^\alpha_j=\sum_{i=1}^n(A^{-1})^i_jD_iU^\alpha,\qquad j=1,\ldots,n,\qquad \alpha=1,\ldots,m,
\]
and the higher-order transformed derivatives are obtained by iterating the transformed total derivatives. This recovers the classical hodograph formulas as the coordinate expression of jet prolongation.
\begin{definition}
For \(r\geq1\), set
\(
 \mathcal U_r:=\{j^r_x\phi\in J^r\pi:\det(D_iX^j)(j^1_x\phi)\neq0\}.
\)
The \emph{\(r\)-th hodograph prolongation} of \(\Phi_0\) is the local map
\(
 \Phi^{(r)}:\mathcal U_r\subset J^r\pi\longrightarrow J^r\pi'
\)
defined as follows: if \(p=j^r_x\phi\in\mathcal U_r\), then \(\Phi_0(\phi(U))\) is locally the graph of a section \(\phi':U'\to E'\), and
\[
 \Phi^{(r)}(p):=j^r_{x'}\phi',\qquad x'=\pi'(\Phi_0(\phi(x))).
\]
\end{definition}
Equivalently, \(\Phi^{(r)}\) is determined in adapted coordinates by the formulas \(D'_j=\sum_{i=1}^n(A^{-1})^i_jD_i\), with $j=1,\ldots,n$, and by their iterates up to order \(r\).

\begin{proposition}\label{prop:classical_hodograph_prolongation}
The hodograph prolongation \(\Phi^{(r)}:\mathcal U_r\to J^r\pi'\) satisfies
\(
 \T\Phi^{(r)}(\mathcal C^r_\pi)=\mathcal C^r_{\pi'}.
\)
Hence, since \(\mathcal C^r_\pi\) and \(\mathcal C^r_{\pi'}\) are the canonical \(N^r_\pi\)- and \(N^r_{\pi'}\)-contact distributions of the corresponding jet spaces, every classical hodograph prolongation is a local \(k\)-contact transformation on its domain of definition. Moreover, it is a Legendrian-polarised transcontact transformation.
\end{proposition}

\begin{proof}
The map \(\Phi^{(r)}\) sends \(r\)-jets of local sections of \(\pi\) to \(r\)-jets of local sections of \(\pi'\). Therefore it sends tangent spaces to prolonged graphs into tangent spaces to prolonged graphs. Since the Cartan distribution is generated by the tangent spaces to prolonged graphs, one obtains \(\T\Phi^{(r)}(\mathcal C^r_\pi)\subset\mathcal C^r_{\pi'}\). Applying the same argument to the inverse hodograph transformation gives the reverse inclusion. Since every \(\mathcal G^r_\pi\)-polarised Legendrian plane is the Cartan \(n\)-plane determined by an \((r+1)\)-jet, and since \(\Phi^{(r)}\) maps prolonged graphs to prolonged graphs on its regular domain, its tangent map sends \(\mathcal G^r_\pi\)-polarised Legendrian planes to \(\mathcal G^r_{\pi'}\)-polarised Legendrian planes. Hence, \(\Phi^{(r)}\) is Legendrian-polarised.
\end{proof}

\begin{definition}\label{def:hodograph_kcontact_transformation}
Let \((J^r\pi,\mathcal C^r_\pi)\) and \((J^r\pi',\mathcal C^r_{\pi'})\) be the   \(N^r_\pi\)- and \(N^r_{\pi'}\)-contact manifolds associated with \(\pi\) and \(\pi'\). A \emph{hodograph \(k\)-contact transformation} is a local diffeomorphism
\[
 \Phi^{(r)}:\mathcal U\subset J^r\pi\longrightarrow\mathcal U'\subset J^r\pi'
\]
such that \(\T\Phi^{(r)}(\mathcal C^r_\pi)=\mathcal C^r_{\pi'}\), and which arises locally from a hodograph-type change of variables as above. It is called \emph{highest-order polarised} if, in addition,
\[
 \T\Phi^{(r)}(\mathcal G^r_\pi)=\mathcal G^r_{\pi'},\qquad
 \mathcal G^r_\pi=\ker\T\pi_{r,r-1},\qquad
 \mathcal G^r_{\pi'}=\ker\T\pi'_{r,r-1}.
\]
\end{definition}

The terminology is specific to the present paper. The classical ingredient is the hodograph change of variables; the \(k\)-contact ingredient is the preservation of the Cartan distribution, which here is understood as the canonical \(N^r_\pi\)-contact distribution of a finite-order jet bundle.

The highest-order polarisation \(\mathcal G^r_\pi=\ker\T\pi_{r,r-1}\) is not always the correct object for detecting the hodograph character of a transformation. Indeed, a classical point transformation, even if non-projectable over the base, still prolongs compatibly with the jet tower \(J^r\pi\to J^{r-1}\pi\) and preserve \(\mathcal G^r_\pi\). What distinguishes a genuine hodograph transformation from a projectable jet transformation is rather the movement of the vertical distribution associated with the chosen independent variables. We therefore introduce the following presentation-level invariant.

\begin{definition}\label{def:hodograph_presentation_defect}
Let \(\Phi^{(r)}:\mathcal U\subset J^r\pi\to\mathcal U'\subset J^r\pi'\) be a hodograph \(k\)-contact transformation. Its \emph{presentation defect} is the relative position of
\[
 \T\Phi^{(r)}(\mathcal V^\pi_r)
 \quad\text{and}\quad
 \mathcal V^{\pi'}_r,
 \qquad
 \mathcal V^\pi_r:=\ker\T\pi_r,\qquad
 \mathcal V^{\pi'}_r:=\ker\T\pi'_r.
\]
We say that the presentation defect vanishes if
\(
 \T\Phi^{(r)}(\mathcal V^\pi_r)=\mathcal V^{\pi'}_r.
\)
\end{definition}

\begin{proposition}\label{prop:hodograph_presentation_defect}
Let \(\Phi_0:E\to E'\) be a local diffeomorphism written in adapted coordinates as \((x^i,u^\alpha)\mapsto (X^i(x,u),U^\alpha(x,u))\), and let \(\Phi^{(r)}\) be its hodograph prolongation. Then, the presentation defect of \(\Phi^{(r)}\) vanishes if and only if \(X^i\) is independent of the fibre variables \(u^\alpha\). Equivalently, \(\Phi_0\) is projectable over a local diffeomorphism \(Q\to Q'\).
\end{proposition}

\begin{proof}
Since the target base coordinates are \(X^i(x,u)\), one has
\[
 \pi'_r\circ\Phi^{(r)}=(X^1,\ldots,X^n).
\]
If \(Y\in\mathcal V^\pi_r=\ker\T\pi_r\), then \(Y(x^i)=0\). Hence
\[
 \T(\pi'_r\circ\Phi^{(r)})(Y)=\sum_{i=1}^nY(X^i)\frac{\partial}{\partial X^i}
 =\sum_{i=1}^n\sum_{\alpha=1}^mY(u^\alpha)\frac{\partial X^i}{\partial u^\alpha}\frac{\partial}{\partial X^i}.
\]
Therefore, \(\T\Phi^{(r)}(Y)\) is vertical for \(\pi'_r\) for all \(Y\in\mathcal V^\pi_r\) precisely when \(\partial X^i/\partial u^\alpha=0\) for all \(i,\alpha\). This is exactly the condition that \(\Phi_0\) be projectable over the base.
\end{proof}

Thus the polarised \(k\)-contact viewpoint separates two notions that are often entangled in coordinates. The Cartan distribution records holonomicity and is preserved by the hodograph prolongation. The distribution \(\mathcal V^\pi_r=\ker\T\pi_r\), or equivalently the choice of independent variables, records the chosen jet presentation and is not preserved by genuine hodograph transformations. This is the sense in which a hodograph transformation is a change of polarised \(k\)-contact presentation rather than merely a change of coordinates.

\begin{proposition}\label{prop:hodograph_preserves_holonomicity}
Let \(\Phi^{(r)}:\mathcal U\subset J^r\pi\to\mathcal U'\subset J^r\pi'\) be a hodograph \(k\)-contact transformation. If \(S\subset\mathcal U\) is a holonomic submanifold of \(J^r\pi\) and \(\Phi^{(r)}(S)\) is transverse to \(\pi'_r:J^r\pi'\to Q'\), then \(\Phi^{(r)}(S)\) is a holonomic submanifold of \(J^r\pi'\).
\end{proposition}

\begin{proof}
Since \(S\) is holonomic, \(\T S\subset\mathcal C^r_\pi\). Since \(\Phi^{(r)}\) is \(k\)-contact, one has \(\T\Phi^{(r)}(\T S)\subset\mathcal C^r_{\pi'}\). Hence \(\Phi^{(r)}(S)\) is an integral submanifold of the target Cartan distribution. By assumption, \(\Phi^{(r)}(S)\) is transverse to \(\pi'_r\), so locally it is the image of a section \(\psi':U'\to J^r\pi'\). The condition \(\T\psi'(\T U')\subset\mathcal C^r_{\pi'}\) implies, by the Cartan criterion for holonomicity, that \(\psi'=j^r\phi'\) for a local section \(\phi':U'\to E'\). Thus, \(\Phi^{(r)}(S)\) is holonomic.
\end{proof}

This proposition explains why hodograph transformations are naturally \(k\)-contact-theoretic. They need not preserve the original fibration, and hence they need not preserve the original vertical directions. Nevertheless, they preserve the Cartan \(N^r_\pi\)-contact distribution, so they preserve holonomicity whenever the transformed integral submanifold is transverse to the new projection.

\begin{definition}\label{def:hodograph_transform_pde}
Let \(\mathcal E\subset J^r\pi\) be a PDE and let \(\Phi^{(r)}:\mathcal U\subset J^r\pi\to\mathcal U'\subset J^r\pi'\) be a hodograph \(k\)-contact transformation. If \(\Phi^{(r)}(\mathcal E\cap\mathcal U)\) is a regular embedded submanifold of \(J^r\pi'\), we call
\[
 \mathcal E':=\Phi^{(r)}(\mathcal E\cap\mathcal U)
\]
the \emph{hodograph transform} of \(\mathcal E\).
\end{definition}

\begin{theorem}\label{thm:hodograph_transforms_solutions}
Let \(\mathcal E\subset J^r\pi\) be a regular PDE and let \(\mathcal E'=\Phi^{(r)}(\mathcal E\cap\mathcal U)\subset J^r\pi'\) be its hodograph transform. If \(\phi:U\to E\) is a solution of \(\mathcal E\), \(j^r\phi(U)\subset\mathcal U\), and \(\Phi^{(r)}(j^r\phi(U))\) is transverse to \(\pi'_r\), then there exists a local section \(\phi':U'\to E'\) solving \(\mathcal E'\) and satisfying
\[
 \Phi^{(r)}(j^r\phi(U))=j^r\phi'(U').
\]
\end{theorem}

\begin{proof}
Since \(\phi\) solves \(\mathcal E\), the holonomic submanifold \(S=j^r\phi(U)\) is contained in \(\mathcal E\cap\mathcal U\). By Proposition~\ref{prop:hodograph_preserves_holonomicity}, the transversality of \(\Phi^{(r)}(S)\) to \(\pi'_r\) implies that \(\Phi^{(r)}(S)=j^r\phi'(U')\) for some local section \(\phi':U'\to E'\). Since \(S\subset\mathcal E\cap\mathcal U\), its image is contained in \(\mathcal E'=\Phi^{(r)}(\mathcal E\cap\mathcal U)\). Hence \(j^r\phi'(U')\subset\mathcal E'\), so \(\phi'\) is a solution of the transformed equation.
\end{proof}
\begin{example}\label{ex:legendre_contact_transformation}
Let \(J^1(\mathbb R,\mathbb R)\) have coordinates \((x,u,p)\) and adapted \(1\)-contact form \(\eta=\dd u-p\,\dd x\). The Legendre transformation
\[
\Phi(x,u,p)=(X,U,P)=(p,u-xp,-x)
\]
satisfies \(\Phi^*(\dd U-P\,\dd X)=\dd u-p\,\dd x\), and hence is a one-contact transformation. It is not adapted to the jet fibration, since the new independent variable is \(X=p\). If \(u=u(x)\), \(p=u_x\), and \(u_{xx}\neq0\), then \(X=u_x\) is a valid local variable and
\[
U(X)=u(x)-xu_x(x),\qquad \frac{\dd U}{\dd X}=-x=P.
\]
Thus, the transformed curve is again holonomic on the regular locus \(u_{xx}\neq0\). However, \(\T\Phi(\partial_p)=\partial_X-x\partial_U\), which is not tangent to the new vertical polarisation \(\langle\partial_P\rangle\), and \(\T\Phi(D_x)=-\partial_P\). Therefore, \(\Phi\) is not a Legendrian-polarised transcontact morphism. It is instead an example of a one-contact transformation preserving selected holonomic curves under a regularity condition.
\end{example}
\begin{example} \label{ex:elementary_hodograph_kcontact}
Let \(\pi:\mathbb R^2_{x,u}\to\mathbb R_x\) and \(\pi':\mathbb R^2_{X,U}\to\mathbb R_X\). Consider the classical hodograph exchange
\[
 X=u,\qquad U=x.
\]
On \(J^1\pi\), with coordinates \((x,u,p)\), the non-degeneracy condition is \(p\neq0\). The transformed derivative is
\[
 P=U_X=\frac{D_xU}{D_xX}=\frac1p.
\]
Hence, the first prolongation is
\[
 \Phi^{(1)}(x,u,p)=\left(u,x,\frac1p\right).
\]
If \(\eta=\dd u-p\,\dd x\) and \(\eta'=\dd U-P\,\dd X\), then
\[
 (\Phi^{(1)})^*\eta'
 =
 \dd x-\frac1p\,\dd u
 =
 -\frac1p(\dd u-p\,\dd x)
 =
 -\frac1p\eta.
\]
Thus, \(\Phi^{(1)}\) preserves the Cartan distribution, and hence is a local \(1\)-contact transformation on the open set \(p\neq0\). However, it does not preserve the vertical distribution associated with the original independent variable. Indeed, \(\mathcal V^\pi_1=\ker\T\pi_1=\langle\partial_u,\partial_p\rangle\), whereas \(\T\Phi^{(1)}(\partial_u)=\partial_X\), which is not vertical for \(\pi'_1\). Thus, the transformation preserves the Cartan distribution \(\mathcal C^1_\pi\) but changes the jet presentation.
\end{example}

\begin{remark}\label{rem:hodograph_not_projectable}
The example shows why the presentation defect is more informative here than the highest-order polarisation alone. For the elementary first-order hodograph transformation, \(\mathcal G^1_\pi=\langle\partial_p\rangle\) is mapped to \(\mathcal G^1_{\pi'}=\langle\partial_P\rangle\), since \(\T\Phi^{(1)}(\partial_p)=-p^{-2}\partial_P\). Nevertheless, the transformation is genuinely non-projectable because it does not preserve \(\ker\T\pi_1\). Thus the essential hodograph feature is not a failure of the highest-order vertical polarisation, but the change of the vertical distribution associated with the chosen independent variables.
\end{remark}

\subsection{Miura maps, solitons and polarised {\it k}-contact correspondences}
\label{subsec:miura_kcontact}

We now describe a classical solitonic construction in a form that is particularly natural from the viewpoint of polarised \(k\)-contact geometry. Let \(Q=\mathbb R^2_{x,t}\), while defining \(\pi:Q\times\mathbb R_u\to Q\) and \(\pi':Q\times\mathbb R_v\to Q\). Consider the KdV equation
\[
 \mathcal E_{\mathrm{KdV}}=\{u_t+6uu_x+u_{xxx}=0\}\subset J^3\pi
\]
and the modified KdV equation with parameter \(\lambda\) of the form
\[
 \mathcal E_{\mathrm{mKdV},\lambda}=\{v_t-6(v^2-\lambda)v_x+v_{xxx}=0\}\subset J^3\pi'.
\]
The Gardner--Miura relation takes the form \(R_\lambda:=u-\lambda-v_x+v^2=0\). To use it to study KdV and mKdV-\(\lambda\) equations, the natural ambient space is \(\mathcal J^{3,4}:=J^3\pi\times_QJ^4\pi'\) because, after substituting \(u=\lambda+v_x-v^2\), the term \(u_{xxx}\) involves derivatives of \(v\) up to order four. The space \(\mathcal J^{3,4}\) carries the product Cartan distribution \(\mathcal C^{3,4}=(Tp_u)^{-1}(\mathcal C^3_\pi)\cap(Tp_v)^{-1}(\mathcal C^4_{\pi'})\), where \(p_u:\mathcal J^{3,4}\to J^3\pi\) and \(p_v:\mathcal J^{3,4}\to J^4\pi'\) are the natural projections. By Proposition~\ref{prop:product_jet_spaces_polarised_k_contact}, \(\mathcal J^{3,4}\) is a polarised \(k\)-contact space, and \(p_u\) and \(p_v\) are Legendrian-polarised transcontact morphisms to the two jet factors.

Let \(\mathcal M_\lambda:=\{D_I R_\lambda=0,\ 0\le |I|\le 3\}\subset\mathcal J^{3,4}\) be the prolonged Miura submanifold. On every regular component, it carries the induced Cartan distribution \(\mathcal C_{\mathcal M_\lambda}:=T\mathcal M_\lambda\cap\mathcal C^{3,4}\). Thus, in the terminology of this paper, the Miura substitution is represented by a  submanifold of a product jet space, together with the two induced projections to the jet factors.

Let \(F_{\mathrm{KdV}}:=u_t+6uu_x+u_{xxx}\) and \(F_{\mathrm{mKdV},\lambda}:=v_t-6(v^2-\lambda)v_x+v_{xxx}\). A direct computation gives
\[
F_{\mathrm{KdV}}\big|_{u=\lambda+v_x-v^2}=(D_x-2v)F_{\mathrm{mKdV},\lambda}.
\]
Indeed, since \(u=\lambda+v_x-v^2\), one has \(u_t=v_{xt}-2vv_t\), \(u_x=v_{xx}-2vv_x\), and \(u_{xxx}=v_{xxxx}-6v_xv_{xx}-2vv_{xxx}\). Substitution gives
\[
F_{\mathrm{KdV}}\big|_{u=\lambda+v_x-v^2}
=v_{xt}-2vv_t+v_{xxxx}-6(v^2-\lambda)v_{xx}
 -12vv_x^2+12v(v^2-\lambda)v_x-2vv_{xxx},
\]
which is exactly \((D_x-2v)F_{\mathrm{mKdV},\lambda}\). Let
\[
\mathcal M_\lambda^{(\mathrm{comp})}:=
\mathcal M_\lambda\cap p_v^{-1}\bigl(\mathcal E_{\mathrm{mKdV},\lambda}^{(1)}\bigr),
\]
where \(\mathcal E_{\mathrm{mKdV},\lambda}^{(1)}\subset J^4\pi'\) is the first prolongation of the modified KdV equation. Equivalently, \(\mathcal M_\lambda^{(\mathrm{comp})}\) is obtained by imposing the prolonged Miura constraints together with \(F_{\mathrm{mKdV},\lambda}=0\), \(D_tF_{\mathrm{mKdV},\lambda}=0\), and \(D_xF_{\mathrm{mKdV},\lambda}=0\). The identity above implies, on this compatibility locus, that \(p_u(\mathcal M_\lambda^{(\mathrm{comp})})\subset\mathcal E_{\mathrm{KdV}}\).

Now, let \(v\) be a holonomic solution of the modified KdV equation and set \(u=\lambda+v_x-v^2\). Then \(\Sigma_v:=\{(j^3_qu,j^4_qv):q\in Q\}\subset\mathcal M_\lambda\) is a holonomic integral surface of \(\mathcal C_{\mathcal M_\lambda}\). Its projections are \(p_u(\Sigma_v)=j^3u(Q)\) and \(p_v(\Sigma_v)=j^4v(Q)\). Since \(\pi'_{4,3}(p_v(\Sigma_v))\subset\mathcal E_{\mathrm{mKdV},\lambda}\), the function \(F_{\mathrm{mKdV},\lambda}\) and its total derivative \(D_xF_{\mathrm{mKdV},\lambda}\) vanish along \(j^4v(Q)\). The identity above therefore implies \(p_u(\Sigma_v)\subset\mathcal E_{\mathrm{KdV}}\). Hence, the Miura construction sends holonomic Legendrian solutions of the modified KdV equation to holonomic Legendrian solutions of the KdV equation through the  submanifold \(\mathcal M_\lambda\) of the product jet space.
This correspondence also contains the usual soliton-generating mechanism. Let \(\lambda=k^2\), with \(k\neq0\), and consider
\[
 v(x,t)=k\tanh(kx-4k^3t+c).
\]
Then, \(v\) satisfies
\[
 v_t-6(v^2-k^2)v_x+v_{xxx}=0.
\]
Applying the Gardner--Miura relation gives
\[
 u=k^2+v_x-v^2=2k^2\operatorname{sech}^2(kx-4k^3t+c),
\]
which is the one-soliton solution, for the above sign convention, of the KdV equation \(u_t+6uu_x+u_{xxx}=0\). In the polarised \(k\)-contact language, the kink solution of the modified equation determines a holonomic integral surface in \(J^4\pi'\). Together with the prolonged Miura constraints, it defines a holonomic integral surface of \(\mathcal M_{k^2}^{(\mathrm{comp})}\). Its projection to \(J^3\pi\) is the holonomic integral surface corresponding to the KdV soliton. Hence, soliton generation is represented inside the present formalism by a holonomic integral surface of the Miura submanifold \(\mathcal M_{k^2}\), whose projections give the modified KdV kink and the KdV soliton.

\begin{remark}
There are two Miura branches,
\[
 u=\lambda+v_x-v^2,\qquad
 \widetilde u=\lambda-v_x-v^2.
\]
For a fixed modified field \(v\), they produce two KdV fields \(u\) and \(\widetilde u\). Eliminating \(v\) gives a B\"acklund-type relation between two KdV solutions. Indeed,
\[
 u-\widetilde u=2v_x,\qquad
 u+\widetilde u=2\lambda-2v^2.
\]
Thus, the two Miura branches factor through an intermediate modified field and can be represented by a middle \ submanifold of a product jet space whose two projections give the two KdV fields. Geometrically, this is again described by a middle space carrying the induced product Cartan distribution and by two projections onto the KdV jet space. The polarised \(k\)-contact structure organises the Bäcklund-type relation as a finite-order product-space construction relating holonomic integral submanifolds.
\end{remark}

\begin{remark}
Iterated Miura, Darboux or B\"acklund steps can be treated in the same way. Each step is represented by a \(k\)-contact submanifold of a suitable product of finite jet spaces, and a chain of such steps gives a chain of transformations of holonomic integral submanifolds. Classical multi-soliton formulae arise after solving the corresponding auxiliary systems or after imposing the usual algebraic superposition rules. The role of the polarised \(k\)-contact formalism is to provide the intrinsic finite-order geometry of these product-space constructions, not to replace the classical inverse-scattering, Darboux or tau-function computations.
\end{remark}

 \newcommand{\changed}[1]{\textbf{#1}}
\section{Conclusions and outlook}

This paper has shown that finite-order jet geometry admits an intrinsic formulation inside polarised \(k\)-contact geometry. The first structural result is Theorem~\ref{cor:Cartan_Nr_contact}: for every fibred manifold \(\pi:E\to Q\),  with \(\dim Q=n\) and fibre dimension \(m\), the Cartan distribution \(\mathcal C^r_\pi\) on \(J^r\pi\) is an \(N^r_\pi\)-contact distribution,  where \(N^r_\pi=m\binom{n+r-1}{r-1}\). For \(r>1\), this  structure is not obtained by merely collecting the standard higher-order contact forms \(\theta^\alpha_I\).  Those forms generate the Cartan codistribution, but they  do not form, in general,  an adapted \(N^r_\pi\)-contact  form. The adapted forms \(\bm\eta^r_\pi\) are  instead obtained locally from  Reeb frames \(R^J_\alpha\). The invariant objects are the Cartan distribution \(\mathcal C^r_\pi\), the highest-order vertical polarisation \(\mathcal G^r_\pi=\ker\T\pi_{r,r-1}\), and the jet structure \(\mathfrak S^r_\pi\). In this sense, the jet structure is not an external element: it is canonically associated with the \(k\)-contact description of the Cartan distribution.

The second structural result is the intrinsic recognition of finite-order jet bundles. Definitions~\ref{def:higher_order_polarisation}, \ref{def:split_polarisation}, \ref{def:Spencer_type_derived_flag} and \ref{def:polarised_jet_type} isolate the precise polarised \(k\)-contact data needed to recover a jet space. The relevant tower is not an arbitrary derived flag. Locally, one chooses an integrable isotropic complement \(H\subset\mathcal D\) of \(\mathcal P\), an integrable complement \(\mathcal V_r\) of \(H\) in \(\T U\), sets \(\mathcal V_0=\mathcal P|_U\), and defines \(\mathcal V_s=\mathcal D^{(s)}\cap\mathcal V_r\) for \(s=1,\ldots,r-1\). The graded quotients of this tower are required to reproduce the symmetric Spencer tower, with \(W=\mathcal V_r/\mathcal V_{r-1}\), and the bracket action by \(H\) must correspond to Spencer contraction. Theorem~\ref{thm:local_split_jet_recognition} proves that these conditions are exactly those characterising the local jet model \((J^r\pi,\mathcal C^r_\pi,\mathcal G^r_\pi)\).

The global  recognition theorem, Theorem~\ref{thm:global_jet_recognition}, separates the local normal form problem from the global descent problem. If \(\mathcal V_r\) and \(\mathcal V_{r-1}\) are globally defined, regular and simple, their leaf spaces reconstruct the base and total space, namely \(Q=M/\mathcal V_r\) and \(E=M/\mathcal V_{r-1}\), while the inclusion \(\mathcal V_{r-1}\subset\mathcal V_r\) induces a fibration \(\pi:E\to Q\). 
The transition condition glues the local jet identifications to a globally defined local diffeomorphism \(\Phi:M\to \Phi(M)\subset J^r\pi\). A global equivalence with \(J^r\pi\) additionally requires \(\Phi\) to be bijective.

The resulting framework gives an intrinsic reformulation of holonomicity. Proposition~\ref{prop:pde_solutions_as_cartan_legendrian} and Theorem~\ref{thm:pde_solutions_polarised_legendrian} show that holonomic solution submanifolds are precisely polarised Legendrian submanifolds. For a PDE \(\mathcal E\subset J^r\pi\), the induced distribution \(\mathcal C_{\mathcal E}=\T\mathcal E\cap\mathcal C^r_\pi\) records the contact-compatible directions, while complementarity to \(\mathcal G^r_\pi\) records the graph condition. 

The set \(\operatorname{Pol}_{\mathcal E}(I_e)\) introduced in Definition~\ref{def:polarisation_space_initial_data} describes the possible \(n\)-dimensional isotropic solution directions extending a prescribed initial direction \(I_e\), while the condition \(H_e\cap\mathcal G^r_{\pi,e}=0\) enforces the polarised graph requirement. Proposition~\ref{prop:polarised_initial_data_solutions} reduces the local construction of solutions extending the initial datum to the problem of finding involutive sections of this family of admissible directions. Similarly, the vertical principal symbol \(\sigma_{\mathcal E}=\left.\dd F\right|_{\mathcal G^r_\pi|_{\mathcal E}}\) measures the failure of \(\mathcal E\) to be transverse to the highest-order polarisation. Proposition~\ref{prop:normal_form_polarisation_symbol} expresses normal forms, symbolic degeneracy and polarisation tangency in the same geometric language.

The Hamiltonian interpretation of Lie characteristics is another consequence of using adapted \(k\)-contact forms. Given a local adapted form \(\bm\eta^r_\pi\), a Cartan vector field \(X\) has Hamiltonian \(\bm h_X=-\iota_X\bm\eta^r_\pi\), and Proposition~\ref{prop:locus_hamiltonian_zero} identifies the zero locus of \(\bm h_X\) with the locus where \(X\) takes values in \(\mathcal C^r_\pi\). Proposition~\ref{prop:triangular_hamiltonian_characteristics} relates the components of \(\bm h_X\) to the usual characteristics and their total derivatives for evolutionary and prolonged symmetries. Theorem~\ref{thm:hamiltonian_invariant_subsystems} then interprets invariant subsystems as Hamiltonian tangency loci.

The reduction results show that the polarised \(k\)-contact setting is not confined to reductions inside a fixed jet presentation. In Section~\ref{Sec:SymReductionJets}, one first reduces an ambient constrained submanifold \(N\subset J^r\pi\) by quotienting enlarged distributions such as \(\mathcal C^r_\pi+\mathcal D_G\), \(\T N+\mathcal D_G\) and \(\mathcal G^r_\pi+\mathcal D_G\).  The reducing directions need not be ordinary Cartan symmetries. What matters is that the enlarged Cartan and polarisation data project to the quotient appropriately. The reduced ambient quotient need not inherit the original jet presentation, but when the reduced polarised distribution satisfies the jet-type recognition conditions, Theorem~\ref{thm:modular_reduction_jet_type} reconstructs it as a genuine finite-order jet model. The reduced system \((\mathcal E\cap N)/G\) then becomes a differential equation in this reconstructed jet bundle.

This point is particularly visible in the treatment of classical \(\lambda\)-symmetries for scalar ordinary differential equations. A \(\lambda\)-prolonged vector field is generally not a Cartan symmetry of the original finite-order jet space. Nevertheless, the line distribution that it spans can be used as a projective reduction direction: the relevant enlarged Cartan and polarisation data descend, and the quotient is recognised as a lower-order jet space by the jet-type criterion. Thus, the usual order-reducing role of \(\lambda\)-symmetries appears here as a consequence of polarised \(k\)-contact reduction, rather than as an ad hoc deformation of the prolongation formula. This also clarifies why the scalar case is geometrically special; extensions to systems require additional structure, such as matrix-valued \(\Lambda\)-symmetries or related \(\mu\)-symmetry formalisms \cite{gaeta2005}.

The same principle underlies the treatment of non-projectable transformations and auxiliary systems. Hodograph transformations, considered in Section~\ref{Sec:Hodograph}, may preserve the Cartan distribution without preserving the original fibration. Proposition~\ref{prop:hodograph_preserves_holonomicity} and Theorem~\ref{thm:hodograph_transforms_solutions} show that such transformations preserve holonomic solutions precisely when the transformed integral submanifold satisfies the required transversality condition for the new projection. Section~\ref{Sec:BacklundLax} gives a parallel interpretation of coverings, Bäcklund transformations and Lax representations: auxiliary compatibility conditions become Cartan compatibility conditions and curvature-defect conditions in enlarged spaces;  a Hamiltonian interpretation additionally requires that the relevant defect fields preserve the \(k\)-contact distribution, as explained in Remark~\ref{rem:lax_not_hamiltonian}.

Overall, the paper changes the role of \(k\)-contact geometry in the geometry of differential equations. It is not used only as a multi-time analogue of contact geometry, nor only as a framework for nonconservative field theories. It provides a geometric category in which finite-order jet bundles form a rigid, intrinsically recognisable sector, while transformations, reductions, Hamiltonian loci and auxiliary systems that are awkward in a fixed jet presentation can still be described naturally. The central message is therefore twofold: finite-order jet geometry is polarised \(k\)-contact geometry of jet type, and polarised \(k\)-contact geometry is broad enough to compare different jet presentations through their Cartan, Spencer and polarisation data.

Several questions remain open. A natural next step is to determine which constructions usually formulated on infinite-order jet spaces admit equivalent finite-dimensional formulations in polarised $k$-contact geometry, and under which hypotheses, as suggested here by reduction theory and the characteristic description of Cartan symmetries. Another is the systematic reformulation of formal integrability, relation with contact integrability and integrability for systems of PDEs. A third direction is the comparison with multisymplectic, multicontact, \(k\)-symplectic, polysymplectic and related field-theoretic formalisms. The present paper does not claim that all these latter structures can be studied via \(k\)-contact structures. It proves instead that finite-order jet geometry has a precise polarised \(k\)-contact characterisation, and that this characterisation supplies a common language for holonomicity, PDEs, symbols, initial data, Lie characteristics, \(\lambda\)-symmetry reduction, general projective reductions, hodograph transformations, Bäcklund transformations, Lax representations, and potentially, many other research fields.

\begin{table}[p]
\centering
\renewcommand{\arraystretch}{1.12}
\begin{tabular}{p{0.30\textwidth}p{0.31\textwidth}p{0.31\textwidth}}
\hline
Classical or jet-geometric object &  \(k\)-contact object & Meaning in the present paper \\
\hline
Cartan distribution $\mathcal{C}^r_\pi$& $N^r_\pi$-contact distribution of corank \(N^r_\pi=m\binom{n+r-1}{r-1}\)&The Cartan distribution is understood as an $N^r_\pi$-contact distribution.\\
Local contact forms \(\theta^\alpha_I\) & Generators of the Cartan codistribution, related to the adapted form \(\bm\eta^r_\pi\) by \eqref{eq:relEtaTheta}. &  Local jet contact forms define \(\mathcal C^r_\pi\), but for \(r>1\) they are not by themselves an adapted \(N^r_\pi\)-contact form; the adapted $N^r_\pi$-contact form is \(\bm\eta^r_\pi\). \\
Cartan symmetries & \(k\)-contact vector fields of \((J^r\pi,\mathcal C^r_\pi)\) & Gives a local Hamiltonian description of Cartan symmetries. \\
Levi tensor for distributions & Map induced locally by  \(-\dd\bm\eta|_{\mathcal D\times\mathcal D}\) & Measures maximal non-integrability and defines orthogonality, isotropy, Legendrianity, etc. \\
Highest-order vertical bundle \(\ker\T\pi_{r,r-1}\) & Polarisation \(\mathcal G^r_\pi=\mathcal V_0\subset\mathcal C^r_\pi\) & A natural polarisation of jet type and order $r$ on jet manifolds. \\
Split Cartan decomposition \(\mathcal C^r_\pi=H^r_\pi\oplus\mathcal G^r_\pi\) & Split polarisation \(\mathcal D=H\oplus\mathcal P\), with \(H\) integrable and isotropic & Provides local total-derivative directions for the reconstruction of a jet tower. \\
Open subset of finite-order jet bundle \(J^r\pi\) & Polarised \(k\)-contact manifold \((M,\mathcal D,\mathcal P)\) of jet type and order \(r\) & Jet bundles are locally recognised intrinsically from \(\mathcal D\), \(\mathcal P\), the split vertical tower and Spencer contractions. \\
Spencer operator and canonical structure \(\mathfrak S^r_\pi\) on \(J^r\pi\) & Canonical quotient-valued form obtained from \([\mathfrak R\circ\bm\eta^r_\pi]\) for a polarisation of jet type & The canonical object is recovered after projection. \\
Algebraic Spencer contraction & Mixed part \(\mathfrak R\circ\dd\bm\eta^r_\pi|_{H\times\mathcal P}\) & Brackets of horizontal Cartan directions with the polarisation reproduce Spencer contraction. \\
Vertical jet tower \(\mathcal G^r_\pi\subset\ker\T\pi_{r,r-2}\subset\cdots\subset\ker\T\pi_r\) & Split tower \(\mathcal V_0\subset\cdots\subset\mathcal V_r\), with \(\mathcal V_s=\mathcal D^{(s)}\cap\mathcal V_r\) for \(s<r\) and \(\T M=H\oplus\mathcal V_r\) & Its quotients satisfy \(\mathcal V_s/\mathcal V_{s-1}\simeq S^{r-s}H^*\otimes W\), \(W=\mathcal V_r/\mathcal V_{r-1}\). \\
\hline
\end{tabular}
\caption{Translation dictionary for the basic structures and for the intrinsic polarised \(k\)-contact model of \(J^r\pi\).}
\label{tab:dictionary_structures_and_jet_type}
\end{table}
\clearpage

\begin{table}[p]
\centering
\renewcommand{\arraystretch}{1.12}
\begin{tabular}{p{0.30\textwidth}p{0.31\textwidth}p{0.31\textwidth}}
\hline
Jet-geometric object & Polarised \(k\)-contact object & Meaning in the present paper \\
\hline
Local jet coordinates \((x^i,u^\alpha_I)\) & Coordinates reconstructed from a polarised distribution of jet type, e.g. \(Q=M/\mathcal V_r\), \(E=M/\mathcal V_{r-1}\) and \(H=\langle D_1,\ldots,D_n\rangle\) & The local recognition theorem recovers jet coordinates algorithmically. \\
Local jet model \(J^r\pi\) & Local polarised transcontactomorphism \((M,\mathcal D,\mathcal P)\simeq(J^r\pi,\mathcal C^r_\pi,\mathcal G^r_\pi)\) & A jet-type polarised \(k\)-contact manifold is locally a finite-order jet bundle. \\
Global jet bundle \(J^r\pi\) & Leaf spaces \(Q=M/\mathcal V_r\), \(E=M/\mathcal V_{r-1}\), with compatible jet-type transition maps & The global theorem glues the local jet identifications when the vertical foliations are simple and the changes of chart are \(r\)-jet prolongations. \\
Projection \(J^r\pi\to Q\) & Projection onto quotient \(M\rightarrow M/\mathcal V_r\) & The base of the reconstructed fibration is the leaf space of the largest vertical distribution. \\
Projection \(J^r\pi\to E\) & Projection onto quotient  \(M/\mathcal V_{r-1}\) & The total space of the reconstructed fibration is recovered from the next vertical quotient. \\
Passage from \(J^r\pi\) to \(J^{r+1}\pi\) & Polarised Legendrian prolongation \(\operatorname{Leg}_{\mathcal P}(\mathcal D)\) for \((M,\mathcal D,\mathcal P)=(J^r\pi,\mathcal C^r_\pi,\mathcal G^r_\pi)\) & The passage from order \(r\) to order \(r+1\) is reconstructed as a bundle of polarised Legendrian planes. \\
Holonomic section \(j^r\phi\) & Polarised Legendrian image \(L\), namely  \(\mathcal D|_L=\T L\oplus\mathcal P|_L\), of a section & Holonomicity is contact tangency plus transversality to the polarisation. \\
Map preserving infinitesimal holonomic directions & Legendrian-polarised transcontact morphism & Sends \(\mathcal P\)-polarised Legendrian planes to \(\mathcal P'\)-polarised Legendrian planes, hence preserves the infinitesimal holonomic geometry. \\
Product or auxiliary jet spaces & Product polarised \(k\)-contact manifolds and correspondences & Allows transformations and auxiliary systems to be treated without forcing all data into one fixed jet bundle. \\
\hline
\end{tabular}
\caption{Translation dictionary for local/global recognition, projections, prolongations and holonomic submanifolds.}
\label{tab:dictionary_recognition_and_holonomicity}
\end{table}
\clearpage

\begin{table}[p]
\centering
\renewcommand{\arraystretch}{1.12}
\begin{tabular}{p{0.30\textwidth}p{0.31\textwidth}p{0.31\textwidth}}
\hline
Jet-geometric object & Polarised \(k\)-contact object & Meaning in the present paper \\
\hline
PDE \(\mathcal E\subset J^r\pi\) & Induced system \((\mathcal E,\mathcal C_{\mathcal E},\mathcal G^r_{\mathcal E})\), with \(\mathcal C_{\mathcal E}=\T\mathcal E\cap\mathcal C^r_\pi\) and \(\mathcal G^r_{\mathcal E}=\T\mathcal E\cap\mathcal G^r_\pi\) & A differential equation inherits $k$-contact-compatible directions and residual symbolic directions. \\
Geometric symbol of a PDE & Residual polarisation \(\mathcal G^r_{\mathcal E}\) & Measures the highest-order freedom left by the equation. \\
Vertical principal symbol of \(F=0\) & Restriction \(\sigma_{\mathcal E}=\left.\dd F\right|_{\mathcal G^r_\pi|_{\mathcal E}}\) & Normality and degeneracy are detected by transversality to the highest-order polarisation. \\
Classical solution of \(\mathcal E\) & Polarised Legendrian image  \(S\subset\mathcal E\) of a section of $\pi_r$& Solutions are integral for \(\mathcal C_{\mathcal E}\), transverse to \(\mathcal G^r_\pi\), and diffeomorphic to their projections onto $Q$. \\
Foliation by solutions & Involutive rank-\(n\) subdistribution \(\mathcal H\subset\mathcal C_{\mathcal E}\) with \(\mathcal H\cap\mathcal G^r_\pi=0\) & Separates the contact condition, the graph condition and Frobenius integrability. \\
Initial datum \(I_e\subset\mathcal C_{\mathcal E,e}\) & Polarised solution set \(\operatorname{Pol}_{\mathcal E}(I_e)\) & Encodes admissible \(n\)-planes extending the initial direction. \\
Characteristic obstruction for initial data & Intersection \(I_e^{\perp_{\mathcal C}}\cap\mathcal G^r_{\mathcal E,e}\) & Detects whether symbolic directions obstruct compatible solution planes. \\
Hamiltonian locus of vector fields & \(\mathcal E_{\mathscr X}=\{\bm h_{X_1}=\cdots=\bm h_{X_s}=0\}\), with induced \(\mathcal D_{\mathcal E_{\mathscr X}}\) and \(\mathcal P_{\mathcal E_{\mathscr X}}\) & Hamiltonian zero sets define polarised differential equations whenever the induced distributions are regular. \\
Lie symmetry of a PDE $\mathcal{E}$ & \(k\)-contact vector field of $(J^r\pi,\mathcal{C}^r_\pi,\mathcal{G}^r_\pi)$ that is tangent to \(\mathcal E\) & Preserves the differential equation and the induced Cartan distribution. \\
Transcontact map of PDEs & Transcontact morphism \(\Phi:J^r\pi\to J^s\pi'\) with \(\Phi(\mathcal E)\subset\mathcal E'\) & Sends contact-compatible directions of one equation to those of another; it is polarised if it also maps initial polarization into final one. \\
\hline
\end{tabular}
\caption{Translation dictionary for PDEs, symbols, solutions, initial data and Hamiltonian loci.}
\label{tab:dictionary_pdes_symbols_initial_data}
\end{table}
\clearpage

\begin{table}[p]
\centering
\renewcommand{\arraystretch}{1.10}
\begin{tabular}{p{0.30\textwidth}p{0.31\textwidth}p{0.31\textwidth}}
\hline
Jet-geometric object & Polarised \(k\)-contact object & Meaning in the present paper \\
\hline
Cartan symmetry on \(J^r\pi\) & \(k\)-contact vector field  for $(J^r\pi,\mathcal{C}^r_\pi)$ & Lie symmetries are infinitesimal symmetries of the Cartan distribution. \\
Prolongation of a vector field on \(E\) & \(k\)-contact vector field preserving \(\mathcal V_{r-1}\) for a polarised \(k\)-contact manifold of jet type \((M,\mathcal D,\mathcal P)\) & Classical point prolongations preserve the tower level corresponding to \(\ker\T\pi_{r,0}\). \\
Evolutionary vector field & \(k\)-contact vector field taking values in \(\mathcal V_r\) for a polarised \(k\)-contact manifold of jet type \((M,\mathcal D,\mathcal P)\)& Intrinsic version of vertical Cartan symmetries; in the jet model these are encoded by characteristics. \\
Characteristic \(Q^\alpha\) of a Cartan symmetry \(X\) & Components of \(\iota_{X}\pi_{r,1}^*\bm\eta^1_\pi\) & Classical characteristic calculus is recovered from local \(N^r_\pi\)-contact Hamiltonians. \\
Characteristic bracket & Bracket induced by the \(k\)-contact Hamiltonian bracket \(\{\bm h_{\bm Q},\bm h_{\bm P}\}_{\bm\eta^r_\pi}=\bm h_{[X_{\bm Q}^{(r)},X_{\bm P}^{(r)}]}\) & The Lie algebra of evolutionary characteristics is recovered from the adapted Hamiltonian bracket in characteristic variables. \\
Cartan and polarised loci of \(X\) & \(\mathcal L_{\mathcal D}(X)=\bm h_X^{-1}(0)\) and \(\mathcal L_{\mathcal P}(X)=\{p:X_p\in\mathcal P_p\}\) & Detect tangency to the Cartan \(N^r_\pi\)-contact distribution and to the symbolic polarisation. \\
Invariant subsystem under a Lie algebra & Common Hamiltonian zero locus \(\mathcal Z_{\mathfrak s}=\{\bm h_X=0:X\in\mathfrak s\}\)& Reformulates invariant-solution conditions as Hamiltonian tangency conditions. \\
Reduction by group directions & Quotients of \(\mathcal C^r_\pi+\mathcal D_G\), \(\T\mathcal E+\mathcal D_G\) and \(\mathcal G^r_\pi+\mathcal D_G\) & Allows reductions that preserve Cartan geometry only modulo fundamental directions. \\
Reduced jet equation & Jet-type reduced polarised structure \((\overline{\mathcal E},\overline{\mathcal C}_{\overline{\mathcal E}},\overline{\mathcal P}_{\overline{\mathcal E}})\) & If the quotient data are of jet type, the recognition theorem reconstructs a finite-order jet equation, possibly over new variables. \\
Hodograph transformation & Non-projectable \(k\)-contact transformation & Preserves the Cartan \(N^r_\pi\)-contact distribution on its regular domain, but may change the independent variables and fail to preserve the vertical bundle $\ker\T \pi_r$.\\
Bäcklund transformations, coverings and Lax pairs & Polarised \(k\)-contact correspondences or Cartan compatibility loci in enlarged spaces & Auxiliary systems and zero-curvature conditions become Cartan prolongation or curvature-defect conditions. \\
\hline
\end{tabular}
\caption{Translation dictionary for symmetries, characteristics, reductions, hodograph transformations and auxiliary integrable-system structures.}
\label{tab:dictionary_symmetries_reductions_applications}
\end{table}
\clearpage
\printbibliography
\end{document}